\documentclass[11pt]{article}
	\usepackage{lineno}
	\usepackage{caption}
    \usepackage{url}
    \usepackage{verbatim}
	\usepackage{graphicx}
	\usepackage{amsmath}
	\usepackage{amsthm}
	\usepackage{amsfonts}
	\usepackage{nicefrac}
	\usepackage{longtable}
	\usepackage{caption}
	\usepackage{graphicx, amssymb, subcaption,graphics}
	\usepackage{tikz}
	\usepackage{pgfplots}
	\usetikzlibrary{spy,calc}
	\usepackage{hyperref}
	\allowdisplaybreaks
	    
\numberwithin{equation}{section}
    \def\norm#1{\left\|#1\right\|}
    
    \def\pd#1#2  {{\frac{\partial #1}{\partial #2}}}
    \def\DOT{\!\cdot\!}

    \def\hf{{\nicefrac{1}{2}}}
    \def\1{\mbox{\boldmath $1$}}

    \def\f{\mbox{\boldmath $f$}}
    \def\g{\mbox{\boldmath $g$}}
    
    \def\n{\mbox{\boldmath $n$}}

    \def\u{\mbox{\boldmath $u$}}
    
    \def\v{\mbox{\boldmath $v$}}

    \def\P{\mbox{$\mathsf{P}_h$}}

    \newcommand{\nrm}[1]{\left\| #1 \right\|}

	\newtheorem{thm}{Theorem}[section]

	\newtheorem{lem}[thm]{Lemma}
	\newtheorem{rem}[thm]{Remark}

\begin{document}
\title{A Second Order Energy Stable Scheme 
	for the Cahn-Hilliard-Hele-Shaw Equations}
	
\author{
Wenbin Chen\thanks{School of Mathematics; Fudan University, 
Shanghai, China 200433 ({\tt wbchen@fudan.edu.cn})}
\and
Wenqiang Feng\thanks{Mathematics Department; University 
of Tennessee; Knoxville, TN 37996, USA 
({\tt wfeng1@utk.edu})}
\and
Yuan Liu\thanks{School of Mathematics; Fudan University, 
Shanghai, China 200433 ({\tt 12110180072@fudan.edu.cn})} 
\and
Cheng Wang\thanks{Mathematics Department; University of 
Massachusetts; North Dartmouth, MA 02747, USA ({\tt corresponding author: cwang1@umassd.edu})}
\and
Steven M. Wise\thanks{Mathematics Department; University 
of Tennessee; Knoxville, TN 37996, USA 
({\tt swise1@utk.edu})}
  }

\maketitle
\begin{abstract}
We present a second-order-in-time finite difference scheme for the Cahn-Hilliard-Hele-Shaw equations. This numerical method is uniquely solvable and unconditionally energy stable. At each time step, this scheme leads to a system of nonlinear equations that can be efficiently solved by a nonlinear multigrid solver. Owing to the energy stability, we derive an $\ell^2 (0,T; H_h^3)$ stability of the numerical scheme.  To overcome the  difficulty associated with the convection term $\nabla \cdot (\phi \u)$, we perform an $\ell^\infty (0,T; H_h^1)$ error estimate instead of the classical $\ell^\infty (0,T; \ell^2)$ one to obtain the optimal rate convergence analysis. In addition, various numerical simulations are carried out, which demonstrate the accuracy and efficiency of the proposed numerical scheme. 
\end{abstract}


\section{Introduction}
The Cahn-Hilliard-Hele-Shaw (CHHS) diffuse interface model has attracted a lot of attention  because it describes two phase flows in a simple way \cite{lee2002modeling2,lee2002modeling1}. It has been used to model spinodal decomposition
of a binary fluid in a Hele-Shaw cell \cite{han2014decoupled}, tumor growth and cell sorting \cite{frieboes2010three,wise2008three}, and two
phase flows in porous media \cite{collins2013efficient}. It describes the process of the phase separation of a viscous, binary fluid into domains. In this model, the Cahn-Hilliard (CH) energy of a binary  fluid with a constant 
 mass density is given by \cite{cahn1958free}: 
 \begin{equation}\label{energy-CH}
 	E(\phi) = \int_\Omega\left\{\frac{1}{4}  \phi^4   
 	- \frac12 \phi^2 
 	+ \frac{\varepsilon^2}{2} \Bigl| \nabla \phi \Bigr|^2 \right\}d{\bf x} , 
 \end{equation} 
 where $\Omega\subset \mathbb{R}^d$ ($d = 2$ or 3), $\phi:\Omega\rightarrow \mathbb{R}$ is the concentration field, and $\varepsilon$ is a constant. The phase equilibria are represented by  the 
 pure fluids $\phi = \pm1$. For simplicity, we assume that 
 $\Omega = (0,L_x) \times (0,L_y) \times(0,L_z)$ and that 
 $\partial_{n} \phi =0$ on $\partial \Omega$, 
 the latter condition representing local thermodynamic equilibrium on the boundary. 
 The dynamic equations of CHHS model  \cite{lee2002modeling2,lee2002modeling1} are given by 
 \begin{eqnarray} \label{equation-CHHS-1} 
 	\partial_t \phi = \Delta \mu - \nabla \cdot ( \phi \u) ,  && \text{in}\quad  \Omega_T:= \Omega\times (0, T),
 	\\
 	\u = - \nabla p - \gamma \phi \nabla \mu   ,  && \text{in} \quad \Omega_T,  \label{equation-CHHS-2}
 	\\
 	\nabla \cdot \u  =  0 ,  && \text{in} \quad \Omega_T, \label{equation-CHHS-3}
 \end{eqnarray}
 where $\gamma > 0$ is related to surface tension and the chemical potential is defined as
 \begin{equation}\label{chem-pot}
 	\mu := \delta_\phi E =  \phi^3 - \phi   - \varepsilon^2 \Delta \phi ; 
 \end{equation}
 $\u$ is the advective velocity; and $p$ is the pressure. We assume no flux 
 boundary condition, namely $\u \cdot \n=0$ and $\partial_{n} \mu =0$, with $\n$ the unit normal vector on $\partial \Omega$: 
 \begin{eqnarray}\label{eqn:bdry}
\frac{\partial\phi}{\partial \n}=\frac{\partial\mu}{\partial \n}=\frac{\partial p}{\partial \n}=0 \quad \text{on} \quad \partial\Omega_T:=\partial\Omega\times (0,T],
\end{eqnarray}

 The system (\ref{equation-CHHS-1})-(\ref{equation-CHHS-3}) is mass conservative and 
 energy dissipative, and the dissipation rate is readily found to be 
 \begin{equation} \label{energy-disspative-PDE} 
 	d_t E = - \int_{\Omega} | \nabla \mu |^2 d{\bf x} 
 	- \frac{1}{\gamma} \int_{\Omega} | \u |^2 d{\bf x} \le 0 . 
 \end{equation}
 Another fundamental observation is that the energy (\ref{energy-CH}) 
 admits a splitting into purely convex and concave parts, i.e., 
 $E = E_c-E_e$: 
 \begin{equation}
 	E_c = \int_\Omega\left\{  \frac14 \phi^4   
 	+\frac{\varepsilon^2}{2} \Bigl| \nabla \phi \Bigr|^2 \right\} d{\bf x},\quad 
 	E_e = \int_\Omega \frac12 \phi^2 \ d{\bf x}  , 
 	\label{canonical-splitting-CH}
 \end{equation}
 where both $E_c$ and $E_e$ are convex. Based on this observation, a first order in time unconditionally energy stable finite difference for the CHHS equations was proposed in \cite{wise2010unconditionally},  and the detailed convergence analysis has become available in a more recent work \cite{chen2015efficient}.  Meanwhile, Feng and Wise presented finite element analysis for the system \eqref{equation-CHHS-1}-\eqref{equation-CHHS-3},  which arise as a diffuse interface model for the two phase Hele-Shaw flow in \cite{feng2012analysis}. Collins et al. proposed an unconditionally energy stable and uniquely solvable  finite difference scheme for the Cahn-Hilliard-Brinkman (CHB) system, which is comprised of a CH-type diffusion equation and a generalized Brinkman equation modeling fluid flow. The detailed convergence analysis for the first order convex splitting scheme to the Cahn-Hilliard-Stokes (CHS) equation was provided in \cite{diegel15a}. In \cite{guo2014efficient}, Guo et al. presented an energy stable fully-discrete local discontinuous Galerkin (LDG) method for the CHHS equations. And also, Han proposed and analyzed a decoupled unconditionally stable numerical scheme for the CHHS equations with variable viscosity in \cite{han2014decoupled}.   
 
Most of the existing schemes are of first oder accuracy in time. In this paper, we propose and analyze a second order convex splitting scheme for the system (\ref{equation-CHHS-1})-(\ref{equation-CHHS-3}), which turns out to be uniquely solvable 
and unconditionally energy stable. A modified Crank-Nicholson approximation is applied to the nonlinear part of the chemical potential, an explicit Adams-Bashforth extrapolation is applied to the concave term, and an Adams-Moulton interpolation formula is applied to the highest order surface diffusion term. In more details, such an Adams-Moulton interpolation formula is applied at the time steps $t^{n+1}$ and $t^{n-1}$ (instead of the standard one at $t^{n+1}$ and $t^n$), so that the diffusion coefficient at $t^{n+1}$ dominates the others. This subtle fact will greatly facilitates the convergence analysis; see the related works for the pure CH flow: \cite{guo16} with the finite difference spatial approximation, \cite{diegel16} with the finite element version. In addition, a semi-implicit approximation is applied to the nonlinear convection term, with the phase variable treated via extrapolation and the velocity field implicitly determined by the Darcy law at the numerical level. A careful analysis reveals a rewritten form of the numerical scheme as the gradient of strictly convex functional, so that both the unique solvability and unconditional energy stability could be theoretically justified.   

Meanwhile, it is noted that an optimal rate convergence analysis for the second order scheme to the CHHS equation remains open. The key difficulty is associated with the high degree of nonlinearity of the convection term, $\nabla \cdot (\phi \u)$, with $\u$ the Helmholtz projection of $- \gamma \phi \nabla \mu$. And also, the Darcy law in the fluid equation has also posed a serious challenge in the numerical analysis, in comparison with the CHS \cite{diegel15a} or Cahn-Hilliard-Navier-Stokes (CHNS) model \cite{diegel16a}, in which a kinematic diffusion is available for the fluid. For the CHHS equation, even the highest order linear diffusion term could not directly control the error estimates for the nonlinear terms, due to the nonlinear convection. For the first order numerical scheme, the methodology to overcome such a difficulty was reported in a few recent works \cite{chen2015efficient, Liu16a}.  However, these analysis techniques could not be directly applied to the second order scheme. In this article, we present a detailed analysis to establish the full order convergence of the proposed numerical scheme, with second order accuracy in both time and space. In more details, a nonlinear energy estimate by taking an inner product with $\mu^{k+1/2}$, the numerical chemical potential at time instant $t^{k+1/2}$, gives an unconditional numerical stability. Moreover, a more careful analysis for the chemical potential gradient, in combination with the Sobolev inequalities at the discrete level, leads to an $\ell^2 (0,T; H_h^3)$ stability estimate of the numerical solution. On the other hand, a subtle observation indicates that the estimate for the nonlinear error associated with $\nabla \cdot (\phi \u)$ cannot be carried out in a standard way, due to a broken structure for this nonlinear error function. As a result, an $\ell^\infty (0,T; H_h^1)$ error estimate has to be performed, instead of the classical $\ell^\infty (0,T; \ell^2)$ one, since the error term associated with the nonlinear convection has a non-positive inner product with the appropriate error test function.  In addition, the $\ell^2 (0,T; H_h^3)$ bound of the numerical solution plays a key role in the nonlinear error estimate, which enables us to apply the discrete Gronwall inequality to obtain the desired convergence result. 

 The remainder of this paper is organized as follows. 
In Section \ref{sec:scheme}, we present the fully-discrete scheme for CHHS equations. 
The $\ell^2 (0,T; H_h^3)$ stability of the numerical scheme is further established in Section \ref{sec:h3stability}. 
In Section \ref{sec:convergence},  we present the optimal rate convergence analysis with the help $\ell^\infty (0,T; H_h^1)$ error estimate.
In Section \ref{sec:numerical}, we provide some numerical results to validate our theoretical analysis and demonstrate the effectiveness of the proposed fully discrete finite difference method. To solve the nonlinear equations at each time step, the nonlinear multigrid solver is applied. 
Finally, we offer our concluding remarks in Section \ref{sec:conclusion}.
\section{The fully discrete scheme and a-priori stabilities}\label{sec:scheme}
In this section, we propose a second order in time fully discrete scheme for the system (\ref{equation-CHHS-1})-(\ref{equation-CHHS-3}) with the discrete homogeneous Neumann boundary conditions (\ref{eqn:bdry}). For simplicity, we consider the cuboid $\Omega = (0, L_x)\times (0, L_y)\times(0, L_z)$, such that there are $N_x, N_y, N_z\in\mathbb{N}$, with $h = L_x/N_x = L_y/N_y = L_z/N_z$, for some $h > 0$. Let $s= \frac{T}{M} >0 $ for some $M\in\mathbb{N}$ , be the time step size and $t_m=ms$. We only consider the three-dimensional version of the fully discrete scheme for the CHHS system since an extension to the two-dimensional case is trivial. For convenience, some of the following notations are defined in Appendix~\ref{app:discrete}.  For each integer $m$, $1 \le m \le M-1$, given $(\phi^{m-1}, \phi^{m})\in\left[\mathcal{C}_\Omega\right]^2$, find the cell-centered grid functions $(\phi^{m+1}, \mu^{m+1/2}, p^{m+1/2})\in \left[\mathcal{C}_\Omega\right]^3$,  such that
               \begin{eqnarray}
		\label{2nd scheme-CHHS-reformulate-s1} 
		\frac{\phi^{m+1}  - \phi^m}{s} &=& \Delta_h  \mu^{m+1/2} 
		- \nabla_h \cdot (A_{h}\phi^{m+1/2}_* \u^{m+1/2}) ,  
		\\
		\mu^{m+1/2}  &=& \chi \left( \phi^{m+1} , \phi^m \right)
		-  \phi^{m+1/2}_* - \varepsilon^2  \Delta_h 
		\Bigl( \frac34 \phi^{m+1}  + \frac14 \phi^{m-1} \Bigr) ,   \label{2nd scheme-CHHS-reformulate-s2} 
		\\
		\u^{m+1/2} &=&  - \nabla_h p^{m+1/2} - \gamma A_h\phi^{m+1/2}_* \nabla_h \mu^{m+1/2}  ,  
		\label{2nd scheme-CHHS-reformulate-s3}
		\end{eqnarray}
with the  boundary conditions $\n\cdot \nabla_h\phi^{m+1} = \n\cdot \nabla_h\mu^{m+1/2} = \n\cdot \nabla_h p^{m+1/2} =0$ (see \eqref{eqn:nbc1}-\eqref{eqn:nbc3}) on $\partial \Omega$, where
\begin{equation}\label{2nd approximation-chi}
\phi^{m+1/2}_* \mathop{:=} \frac32 \phi^m - \frac12 \phi^{m-1} , \quad \chi \left( \varphi , \psi  \right) 
 \mathop{:=}  \frac14 \left(  \varphi^2 + \psi^2 \right) 
 \left( \varphi + \psi  \right), 
 \end{equation}
for any $(\varphi, \psi)\in\left[\mathcal{C}_\Omega\right]^2$. Note that the three component variables of the velocity vector $\u^{m+1/2} \in  \vec{\mathcal{E}}_{\Omega}$ are located at the staggered grid points. To facilitate the unique solvability analysis below, we could eliminate the velocity variable in the numerical scheme and rephrase it in terms of $(\phi^{m+1}, \mu^{m+1/2}, p^{m+1/2})\in \left[\mathcal{C}_\Omega\right]^3$. In more details, we introduce ${\cal M}(\phi) \mathop{:=} 1+ \gamma\phi^2$ and rewrite 
\eqref{2nd scheme-CHHS-reformulate-s1}-\eqref{2nd scheme-CHHS-reformulate-s3} as  
\begin{eqnarray}
\phi^{m+1}-\phi^{m} &=& s \nabla_h\cdot\left(\mathcal{M}(A_{h}\phi_*^{m+1/2}) \nabla_{h}\mu^{m+1/2}\right) +s\nabla_{h}\cdot\left(A_{h}\phi_*^{m+1/2} \nabla_{h}p^{m+1/2}\right) ,
\label{2nd scheme-CHHS-s1}\\
\mu^{m+1/2} &=& \chi \left( \phi^{m+1} , \phi^m \right)-  \phi^{m+1/2}_* - \varepsilon^2  \Delta_h \Bigl( \frac34 \phi^{m+1}  + \frac14 \phi^{m-1} \Bigr) ,
\label{2nd scheme-CHHS-s2}
\\
\Delta_{h}p^{m+1/2} &=& - \gamma \nabla_h\cdot\left(A_h\phi_*^{m+1/2} \nabla_h\mu^{m+1/2}\right) .
\label{2nd scheme-CHHS-s3}
	\end{eqnarray}
The symbol  $\mathcal{M}(A_{h}\phi_*^{m+1/2})\nabla_{h}\mu^{m+1/2}$ represents a discrete vector field.  For instance, the $y$-component at a generic $y$-face grid point is given as 
\[
\left[\mathcal{M}(A_{h}\phi_*^{m+1/2})\nabla_{h}\mu^{m+1/2}\right]^y_{i,j\pm\hf,k} = \mathcal{M}(A_{y}\phi^{m+1/2}_{*,i,j\pm\hf,k})D_y\mu^{m+1/2}_{i,j\pm\hf,k} .
\]
Hence, $\mathcal{M}(A_{h}\phi_*^{m+1/2})\nabla_{h}\mu^{m+1/2}\in\vec{\mathcal{E}}_\Omega$ and similarly for $A_{h}\phi_*^{m+1/2}\nabla_{h}p^{m+1/2}$ and $A_{h}\phi_*^{m+1/2}\nabla_{h}\mu^{m+1/2}$. The definitions of the discrete operators used above can be found in Appendix~\ref{sebsec:discrete operate} and are similar to those found in \cite{wise2010unconditionally}.
	
	
 We now define a fully discrete energy that is consistent with the continuous space energy \eqref{energy-CH} as $h\to 0$. In particular, the discrete energy $E_h: \mathcal{C}_\Omega\to \mathbb{R}$ is
\begin{eqnarray}\label{eqn:dis-energy}
E_h(\phi) := \frac{1}{4}\nrm{\phi}_4^4-\frac{1}{2} \nrm{\phi}_2^2 + \frac{\varepsilon^2}{2}\nrm{\nabla_h\phi}_2^2.
\end{eqnarray}
We also define an alternate numerical energy via
	\begin{equation}
	F_h(\phi,\psi) = E_h (\phi) +  \frac14  \nrm{ \phi  -  \psi }_2^2   
	+ \frac{1}{8}\varepsilon^2   \nrm{ \nabla_h ( \phi -  \psi ) }_2^2. 
	\label{alternate-energy}
	\end{equation}
We can not guarantee that the energy $E_h$ is non-increasing in time, but,  we can guarantee the dissipation of auxiliary energy $F_h$.

For our present and future use, we define the canonical grid projection operator $\P: C^0(\Omega)\to \mathcal{C}_\Omega$ via $[\P v]_{i,j,k}=v(\xi_i,\xi_j,\xi_k)$. Set $u_{h,s}:=\P u(\cdot, s)$.  Then $F_h(u_{h,0},u_{h,s})\to E_h(u(\cdot, t_0))$ as $h\to 0$ and $s\to 0$ for sufficiently regular $u$. The next theorem addresses the unique solvability and unconditional energy stability of the numerical solutions to the scheme \eqref{2nd scheme-CHHS-s1} -- \eqref{2nd scheme-CHHS-s3}:
	
\begin{thm}\label{thm:Wise2010}
Suppose that  $\left(\phi_{e},\mu_{e},\u_{e}\right)$ is the sufficiently regular exact solution to the CHHS system (\ref{equation-CHHS-1})-(\ref{equation-CHHS-3}). Take $\Phi_{i,j,k}^\ell = \P\phi_e (\cdot, t_\ell)$ and suppose that the initial profile $\phi^0:= \Phi^0$, $\phi^1:= \Phi^1\in \mathcal{C}_\Omega$ satisfies homogeneous Neumann boundary conditions $\n  \cdot  \nabla_h \phi^{0} = 0$ and $\n  \cdot  \nabla_h \phi^1 = 0$ on $\partial\Omega$. Given any $(\phi^{m-1}, \phi^m) \in\left[\mathcal{C}_\Omega\right]^2$, there is a unique solution $\phi^{m+1} \in \mathcal{C}_\Omega$ to the scheme \eqref{2nd scheme-CHHS-s1} -- \eqref{2nd scheme-CHHS-s3}. And also, the scheme \eqref{2nd scheme-CHHS-s1} -- \eqref{2nd scheme-CHHS-s3}, with starting values $\phi^{0}$ and $\phi^1$, is unconditionally energy stable with respect to \eqref{alternate-energy}, \emph{i.e.}, for any $s > 0$ and $h>0$, and any positive integer $1\le \ell\le M-1$,
\begin{equation}\label{ieq:energy dissipative}
F_h(\phi^{\ell+1},\phi^{\ell})  + s\sum^{\ell}_{m=1}\|\nabla_{h}\mu^{m+1/2}\|^{2}_{2}+\frac{s}{\gamma}\sum^{\ell}_{m=1}\|\textbf{u}^{m+1/2}\|^{2}_{2} 
\leq F_h(\phi^1,\phi^0) \le C_0 ,
\end{equation}
where $C_0>$ is a constant independent of $s$, $h$, and $\ell$, and $\u^{m+1/2} \in \vec{\mathcal{E}}_{\Omega}$ is given by \eqref{2nd scheme-CHHS-reformulate-s3}. 
\end{thm}
	
\begin{proof}
The unique solvability proof follows from the convexity analysis, as presented in \cite{wise2010unconditionally} for the first order convex splitting scheme applied to the CHHS equation. 	We define a linear operator ${\cal L}$ as 
\begin{eqnarray} 
  {\cal L} (\mu) &:=& - s \nabla_h\cdot\left(\mathcal{M}(A_{h}\phi_*^{m+1/2}) \nabla_{h}\mu \right) - s\nabla_{h}\cdot\left(A_{h}\phi_*^{m+1/2} \nabla_{h}p_\mu \right) ,  \label{solvability-1} 
 \end{eqnarray} 
  in which
\begin{eqnarray}
  \Delta_{h} p_{\mu} &=& - \gamma \nabla_h\cdot\left(A_h\phi_*^{m+1/2} \nabla_h \mu \right) ,
  \label{solvability-2}  
\end{eqnarray} 
with $\phi_*^{m+1/2}$ a known function, and homogeneous Neumann boundary conditions for both $\mu$ and $p_{\mu}$. Following the arguments in \cite{wise2010unconditionally}, we are able to prove that ${\cal L}$ gives rise to a symmetric, coercive, and continuous bilinear form when the domain is restricted to $\mathring{\mathcal{C}}_{\Omega}:= \left\{ \mu \in \mathcal{C}_\Omega: ( \mu , {\bf 1} ) =0 \right\}$; the details are skipped for the sake of brevity and left to interested readers. 

Subsequently, an inner product on $\mathring{\mathcal{C}}_\Omega$ is introduced using ${\cal L}$: let $f_{\mu}$ and $f_{\nu} \in \mathring{\mathcal{C}}_\Omega$ and suppose $\mu, \nu \in \mathring{H}_h^1$ are the unique  solutions to ${\cal L} (\mu) = f_{\mu}$ and ${\cal L} (\nu) = f_{\nu}$. Then we define 
\begin{eqnarray} 
  \left( f_{\mu} ,  f_{\nu} \right)_{{\cal L}^{-1}} :
  = s ( \nabla_h \mu, \nabla_h \nu ) + \frac{s}{\gamma} \left( \nabla_h f_{\mu} + \gamma A_h\phi_*^{m+1/2} \nabla_h \mu ,  \nabla_h f_{\nu} + \gamma A_h\phi_*^{m+1/2} \nabla_h \nu \right) . 
  \label{solvability-3}  
\end{eqnarray}
It is straightforward to verify that 
\begin{eqnarray} 
  \left( f_{\mu} ,  f_{\nu} \right)_{{\cal L}^{-1}} =  \left( f_{\mu} ,  {\cal L}^{-1} f_{\nu} \right)  
  =  \left( {\cal L}^{-1} f_{\mu} , f_{\nu} \right) . 
  \label{solvability-4}  
\end{eqnarray} 
Next, we consider the following functional: 
\begin{eqnarray} 
  G (\phi) = \frac12 \left( \phi - \phi^m ,  \phi - \phi^m \right)_{{\cal L}^{-1}} 
  + F_c (\phi) - ( \phi , g_e (\phi^m, \phi^{m-1} ) ) , \label{solvability-5-1}   
\end{eqnarray}   
with
\begin{eqnarray}   
  &&F_c (\phi) = \frac{1}{16} \| \phi \|_4^4 + \frac{1}{12} ( \phi^m , \phi^3 ) + \frac18 ( (\phi^m)^2 , \phi^2 ) + \frac38 \varepsilon^2 \| \nabla_h \phi \|_2^2 ,  \label{solvability-5-2}
\\
  &&
  g_e (\phi^m, \phi^{m-1}) = - \frac14 (\phi^m)^3 + \phi^{m+1/2}_* + \frac14 \varepsilon^2 \Delta_h \phi^{m-1} .  \label{solvability-5-3}
\end{eqnarray} 
The convexity of $F_c$ follows from the convexity of $g_c (\phi): = \frac{1}{16} \phi^4 + \frac{1}{12} \phi^m \phi^3 + \frac18 (\phi^m)^2 \phi^2$ (in terms of $\phi$). And also, $( \cdot , \cdot )_{{\cal L}^{-1}}$ is an inner product. Therefore, we conclude that $G$ is convex. Moreover, $G$ is coercive over the set of admissible functions 
\begin{eqnarray} 
  {\cal A} = \left\{ \phi \in \mathcal{C}_\Omega: ( \phi - \phi^m , {\bf 1} ) =0 \right\} . 
  \label{solvability-6}
\end{eqnarray}  
Therefore it has a unique minimizer, and in particular the minimizer of $G$, which we denote as $\phi =  \phi^{m+1}$, satisfied the discrete equation 
\begin{eqnarray} 
  {\cal L}^{-1} ( \phi^{m+1} - \phi^m ) + \delta_\phi F_c (\phi^{m+1}) - g_e (\phi^m, \phi^{m-1} ) = C , 
  \label{solvability-7}
\end{eqnarray}   
in which $C$ is a constant and $\phi^{m+1}$ satisfies the homogeneous Neumann boundary condition. In other words, $\phi^{m+1}$ is a solution of 
\begin{eqnarray} 
  \phi^{m+1} - \phi^m + {\cal L} \left( \delta_\phi F_c (\phi^{m+1}) - g_e (\phi^m, \phi^{m-1} ) \right) =0 , 
  \label{solvability-8}
\end{eqnarray} 
which is equivalent to the numerical scheme \eqref{2nd scheme-CHHS-s1} -- \eqref{2nd scheme-CHHS-s3}. The proof of unique solvability is complete. 
			  
		For the energy stability analysis, we look at the numerical scheme in the original formulation \eqref{2nd scheme-CHHS-reformulate-s1}-\eqref{2nd scheme-CHHS-reformulate-s3}. 
		Taking an inner product with $\mu^{m+1/2}$ (given by (\ref{2nd scheme-CHHS-s2})) 
		by (\ref{2nd scheme-CHHS-reformulate-s1}) yields 
		\begin{equation} 
		\left( \phi^{m+1}  - \phi^m , \mu^{m+1/2} \right)  
		- s\left( \Delta_h \mu^{m+1/2} , \mu^{m+1/2} \right)
		+ s\left( \nabla_h \cdot (A_h\phi^{m+1/2}_* \u^{m+1/2}) , \mu^{m+1/2} \right) = 0 . 
		\label{energy-est-1}
		\end{equation}
		In more detail, the leading term has the expansion 
		\begin{eqnarray} 
		\left( \phi^{m+1}  - \phi^m , \mu^{m+1/2} \right)  
		&=& \left( \phi^{m+1}  - \phi^m ,  \chi \left( \phi^{m+1} , \phi^m \right) \right)
		-  \frac12 \left( \phi^{m+1}  - \phi^m , 3\phi^m -  \phi^{m-1}   \right)\nonumber \\
		&-&  \frac{1}{4}\varepsilon^2 \left(\phi^{m+1}- \phi^m ,  \Delta_h \Bigl( 3 \phi^{m+1}  + \phi^{m-1} \Bigr)  \right)
		:= I_{1}+I_{2}+\varepsilon^2I_{3}
		\label{energy-est-2}
		\end{eqnarray}     
		The estimate for $I_{1}$, a convex term, is straightforward: 
		\begin{eqnarray}  
		I_{1}&=& \frac14  \biggl(  \phi^{m+1} - \phi^n  , 
		\Bigl(  (\phi^{m+1})^2 + (\phi^m)^2 \Bigr) 
		\Bigl( \phi^{m+1} + \phi^m  \Bigr)   \biggr)   \nonumber 
		\\
		&=& \frac14  \biggl(  
		\Bigl( ( \phi^{m+1} )^2 - ( \phi^m )^2  \Bigr) , 
		\Bigl(  (\phi^{m+1})^2 + (\phi^m)^2 \Bigr)   \biggr)  
		=  \frac14  \Bigl( \nrm{ \phi^{m+1} }_4^4
		-  \nrm{ \phi^m }_4^4   \Bigr) .  
		\label{energy-est-3}
		\end{eqnarray}
		
		For the second term $I_2$ of (\ref{energy-est-2}), a concave term, we see that 
		\begin{eqnarray}\label{energy-est-5}
		I_2 &=& - \Bigl( \phi^{m+1} - \phi^m , \phi^m  \Bigr) 
		- \frac12 \Bigl( \phi^{m+1} - \phi^m , \phi^m  -  \phi^{m-1}  \Bigr)\nonumber\\
		&=&- \frac12 \Bigl( \nrm{\phi^{m+1}}_2^2  
		-  \nrm{\phi^m}_2^2  \Bigr) +  \frac12  \nrm{\phi^{m+1} - \phi^m}_2^2 
		- \frac12 \Bigl( \phi^{m+1} - \phi^m , \phi^m  -  \phi^{m-1}  \Bigr)\\
		&\geq&- \frac12 \Bigl( \nrm{\phi^{m+1}}_2^2  -  \nrm{\phi^m}_2^2  \Bigr)  
		+  \frac14  \nrm{\phi^{m+1} - \phi^m}_2^2 
		-    \frac14 \nrm{\phi^m  -  \phi^{m-1}}_2^2,\nonumber
		\end{eqnarray}
		where the Cauchy inequality was utilized in the last step. 

		The third term $I_3$ of \eqref{energy-est-2}, also a convex term, can be analyzed 
		with the help of summation by parts:  
		\begin{eqnarray}  
		I_3&=& \frac14 \left(  \nabla_h \Bigl( \phi^{m+1} - \phi^m \Bigr) , 
		\nabla_h \Bigl( 3 \phi^{m+1} + \phi^{m-1} \Bigr) \right)   \nonumber 
		\\ 
		&=& \frac12 \left(  
		\nabla_h \Bigl( \phi^{m+1} - \phi^m \Bigr) , \nabla_h \Bigl( \phi^{m+1} + \phi^m \Bigr) \right)  
		  +\frac14 \left(  
		\nabla_h \Bigl( \phi^{m+1} - \phi^m \Bigr) , 
		\nabla_h \Bigl( \phi^{m+1} - 2 \phi^m +  \phi^{m-1} \Bigr) \right)\nonumber  \\
		&:=&I_{3,1}+I_{3,2}. 
		\label{energy-est-4-1}
		\end{eqnarray}
		The evaluation of $I_{3,1}$ is straightforward: 
		\begin{equation} 
		I_{3,1}=\frac12 \left(  
		\nabla_h \Bigl( \phi^{m+1} - \phi^m \Bigr) , 
		\nabla_h \Bigl( \phi^{m+1} + \phi^m \Bigr) \right) 
		=  \frac12  \Bigl(  \nrm{\nabla_h \phi^{m+1}}_2^2 
		- \nrm{\nabla_h \phi^m}_2^2    \Bigr)  .
		\label{energy-est-4-2}
		\end{equation}
		The estimate of the $I_{3,2}$ can be carried out in the 
		following way: 
		\begin{eqnarray}
		I_{3,2}
		&=& \frac14\nrm{\nabla_h ( \phi^{m+1} - \phi^m)}_2^2  
		- \frac14\left(  \nabla_h ( \phi^{m+1} - \phi^m ) , 
		\nabla_h ( \phi^m -  \phi^{m-1} )  \right)    \nonumber 
		\\
		&\ge& \frac18 \Bigl( \nrm{ \nabla_h ( \phi^{n+1} - \phi^n ) }_2^2  
		-  \nrm{ \nabla_h ( \phi^m -  \phi^{m-1}) }_2^2  \Bigr)  ,  
		\label{energy-est-4-3}
		\end{eqnarray}
		in which the Cauchy inequality was applied in the last step. 
		Consequently, substituting (\ref{energy-est-4-2}) and (\ref{energy-est-4-3}) 
		into (\ref{energy-est-4-1})  yields 
		\begin{eqnarray} 
		I_3 &\ge&  \frac12  \Bigl(  \nrm{ \nabla_h \phi^{m+1} }_2^2 
		- \nrm{ \nabla_h \phi^m }_2^2    \Bigr) 
		+ \frac18 \Bigl( \nrm{ \nabla_h ( \phi^{m+1} - \phi^m ) }_2^2  
		-  \nrm{ \nabla_h ( \phi^m -  \phi^{m-1} ) }_2^2  \Bigr) . 
		\label{energy-est-4-4}
		\end{eqnarray}
		
		Finally, a combination of (\ref{energy-est-2}), (\ref{energy-est-3}), (\ref{energy-est-5}) and (\ref{energy-est-4-4}) 
		results in 
		\begin{eqnarray}
		\left( \phi^{m+1}  - \phi^m , \mu^{m+1/2} \right)   
		&\ge& E_h (\phi^{m+1}) - E_h (\phi^m) 
		+  \frac14 \Bigl(  \nrm{ \phi^{m+1} - \phi^m }_2^2 
		-    \nrm{ \phi^m  -  \phi^{m-1} }_2^2   \Bigr)     \nonumber 
		\\
		&&+ \frac{1}{8}\varepsilon^2 \Bigl( \nrm{ \nabla_h ( \phi^{m+1} - \phi^m ) }_2^2  
		-  \nrm{ \nabla_h ( \phi^m -  \phi^{m-1} ) }_2^2  \Bigr)  . 
		\label{energy-est-9}      
		\end{eqnarray}
		
		For the second term of (\ref{energy-est-1}), the boundary condition 
		$\n\cdot \nabla_h \mu^{m+1/2} \mid_{\partial \Omega}=0$  
		leads to the following summation by parts: 
		\begin{eqnarray} 
		\left( \Delta_h \mu^{m+1/2} , \mu^{m+1/2} \right) 
		=  - \left( \nabla_h \mu^{m+1/2} , \nabla_h \mu^{m+1/2} \right)
		= - \nrm{\nabla_h \mu^{m+1/2}}_2^2 .   
		\label{energy-est-10} 
		\end{eqnarray}

		The third term of (\ref{energy-est-1}) can be analyzed in a similar way. 
		By a reformulated form 
		of the second equation (\ref{2nd scheme-CHHS-reformulate-s3}) 
		\begin{equation} 
		A_h\phi^{m+1/2}_* \nabla_h \mu^{m+1/2} 
		= \frac{1}{\gamma} \Bigl( - \u^{m+1/2} - \nabla_h p^{m+1/2}  \Bigr)  , 
		\label{energy-est-11} 
		\end{equation} 
		we have
		\begin{eqnarray} 
		\left( \nabla_h \cdot (A_h\phi^{m+1/2}_* \u^{m+1/2}) , \mu^{m+1/2} \right)  
		&=& - \left( \u^{m+1/2},  A_h\phi^{m+1/2}_* \nabla_h \mu^{m+1/2} \right)  \nonumber 
		\\
		& = & \frac{1}{\gamma} \nrm{\u^{m+1/2}}_2^2 
		+  \left(\nabla_h\cdot \u^{m+1/2},    p^{m+1}  \right)\\ 
		&=&  \frac{1}{\gamma} \nrm{\u^{m+1/2}}_2^2 ,  \label{energy-est-12} 
		\end{eqnarray} 
		in which the last step comes from  $\nabla_h\cdot \u^{m+1/2}= 0$ 
		and $\u^{m+1/2} \cdot \n = 0\, \, \, \mbox{on} \, \, \,  \partial \Omega$.
		
		As a result, a substitution of (\ref{energy-est-9}), (\ref{energy-est-10}) 
		and (\ref{energy-est-12}) into (\ref{energy-est-1}) becomes 
		\begin{eqnarray} 
		&&E_h (\phi^{m+1}) - E_h (\phi^m) 
		+  \frac14 \Bigl(  \nrm{ \phi^{m+1} - \phi^m }^2 
		-    \nrm{ \phi^m  -  \phi^{m-1} }_2^2   \Bigr)   \nonumber 
        + \frac{1}{8}\varepsilon^2 \Bigl( \nrm{ \nabla_h ( \phi^{m+1} - \phi^m ) }_2^2  
		\\
		&&-  \nrm{ \nabla_h ( \phi^m -  \phi^{m-1} ) }_2^2  \Bigr)   
		+ s\nrm{\nabla_h \mu^{m+1/2}}_2^2 
		+ \frac{s}{\gamma} \nrm{\u^{m+1/2}}_2^2 \le 0 . 
		\label{energy-est-14} 
		\end{eqnarray}       

		By the definition of the alternate numerical energy, we arrive at 
		\begin{eqnarray} 
		F_h(\phi^{m+1},\phi^{m}) - F_h(\phi^m,\phi^{m-1})  + s\nrm{\nabla_h \mu^{m+1/2}}_2^2 
		+ \frac{s}{\gamma} \nrm{\u^{m+1/2}}_2^2 \le 0 . 
		\label{energy-est-15} 
		\end{eqnarray}      
		This in turn shows that the modified energy is non-increasing in time. 
		Summing over time for (\ref{energy-est-15}) yields 
		\begin{equation} 
		F_h(\phi^{\ell+1},\phi^{\ell})  + s\sum^{\ell}_{m=1}\|\nabla_{h}\mu^{m+1/2}\|^{2}_{2}+\frac{s}{\gamma}\sum^{\ell}_{m=1}\|\textbf{u}^{m+1/2}\|^{2}_{2} 
		\le F_h(\phi^1,\phi^0) \le C_0 , \quad \forall\ell \ge 1   .
		\label{energy-est-16} 
		\end{equation} 
		Then we obtain the unconditional energy stability for the second order 
		scheme (\ref{2nd scheme-CHHS-s1})-(\ref{2nd scheme-CHHS-s3}).
    \end{proof}

\begin{rem} 
  It is observed that data for two initial time steps, either $(\phi^0, \phi^1)$ or $(\phi^{-1}, \phi^0)$, are needed for \eqref{2nd scheme-CHHS-s1} -- \eqref{2nd scheme-CHHS-s3}, since ours is a two-step scheme. In this article, we take $\phi^0= \Phi^0$, $\phi^1= \Phi^1$ for simplicity of presentation. Other initialization choices, such as $\phi^{-1} = \phi^0$, or computing $\phi^1$ by a first order temporal scheme, could be taken. Moreover, the energy stability and the second order temporal convergence rate are also expected to be available for these initial data choices; see the related works for the Cahn-Hilliard model~\cite{diegel16, guo16}.
\end{rem}

    \section{$\ell^2 (0,T; H_h^3)$ stability of the numerical scheme} \label{sec:h3stability}
    
    The $\ell^\infty(0,T; H_h^1)$ bound of the numerical solution could be derived based on the weak energy stability (\ref{energy-est-16}). The following quadratic inequality is observed:  
    \begin{equation} 
      \frac18 \phi^4 - \frac12 \phi^2 \ge -\frac12 ,  \quad 
      \mbox{which in turn yields} \quad 
      \frac18 \| \phi \|_4^4 - \frac12 \| \phi \|_2^2 \ge - \frac12 | \Omega | ,  
      \label{H1 bound-1} 
   \end{equation} 
with the discrete $H_h^1$ norm introduced in \eqref{discrete norm-H1}. 
   Then we arrive at the following bound, for any $\phi \in {\mathcal{C}}_\Omega$:  
   \begin{eqnarray} 
     E_h (\phi) \ge \frac18 \| \phi \|_4^4 + \frac{\varepsilon^2}{2} \| \nabla_h \phi \|_2^2 - \frac12 | \Omega | 
     \ge \frac12 \| \phi \|_2^2 + \frac{\varepsilon^2}{2} \| \nabla_h \phi \|_2^2 - | \Omega | 
     \ge \frac12 \varepsilon^2 \| \phi \|_{H_h^1}^2 - | \Omega | . 
 \end{eqnarray}   
 Consequently, its combination with \eqref{energy-est-16} yields the following estimate:
    \begin{equation} 
       \frac12 \varepsilon^2 \nrm{\phi^{\ell+1}}_{H_h^1}^2 + s \sum_{m=1}^{\ell} \Bigl( \nrm{\nabla_h \mu^{m+1/2}}_2^2 
        + \frac{1}{\gamma} \nrm{\u^{m+1/2}}_2^2 \Bigr) \le C_0 + | \Omega | := C_1 ,  
        \label{leading stability} 
    \end{equation} 
so that a uniform in time bound for $\phi$ in $\ell^\infty (0,T; H_h^1)$ is available:  
    \begin{equation} 
    \nrm{\phi^m}_{H_h^1} \le C_2 := \varepsilon^{-1} \sqrt{2 C_1} , \quad 
    \mbox{for any} \, \, \, m  .  \label{H^1-stability} 
    \end{equation}

\begin{thm} \label{thm:stability L2-H3} 
Let $\phi^{m} \in \mathcal{C}_\Omega$ be the solution to the scheme \eqref{2nd scheme-CHHS-s1} -- \eqref{2nd scheme-CHHS-s3}, with sufficient regularity assumption for $\Phi^0$ and $\Phi^1$, then for any $1 \le \ell \le M-1$, we have 
     \begin{equation} 
     \frac1{16} \varepsilon^4 s \sum_{m=1}^{l}  \nrm{\nabla_h \Delta_h \phi^m}_2^2 
     \le C_{11} + C_{10} T,   \label{stab-L2-H3-0}  
     \end{equation} 
 where $C_{10}$ and $C_{11}$, given by (\ref{eqn:C10}) 
     and (\ref{eqn:C11}), respectively, only depend on $L_x$, $L_y$, $L_z$, 
     $\varepsilon$ and several Sobolev embedding constants, and are  independent of $h$, $s$ and final time $T$.     
\end{thm}

%
%
\begin{proof}
    We observe that 
    \begin{eqnarray} 
    &&\nrm{\nabla_h \mu^{m+1/2}} 
     =  \nrm{\nabla_h \Bigl(   
    	\varepsilon^2  \Delta_h ( \frac34 \phi^{m+1} + \frac14 \phi^{m-1} )  
    	- \chi \left( \phi^{m+1} , \phi^m \right)  + \phi^{m+1/2}_*  \Bigr) } _2   \nonumber 
    \\
    &  \ge & \biggl| 
    \varepsilon^2 \nrm{\nabla_h \Delta_h ( \frac34 \phi^{m+1} + \frac14 \phi^{m-1} )} _2
    - \nrm{ \nabla_h \Bigl(  \chi \left( \phi^{m+1} , \phi^m \right) 
    	- ( \frac32 \phi^m - \frac12 \phi^{m-1} ) \Bigr)} _2 \biggr| , 
    \label{triangle estimate} 
    \end{eqnarray} 
    in which a triangle inequality was applied in the last step. 
    Furthermore, motivated by the quadratic inequality $ |a-b|^2 \ge \frac12 a^2 - b^2 $, 
    we have
    \begin{eqnarray} 
    \nrm{\nabla_h \mu^{m+1/2}}^2 
    & \ge &  \frac{1}{2}\varepsilon^4\nrm{\nabla_h \Delta_h 
    	( \frac34 \phi^{m+1} + \frac14 \phi^{m-1} )}_2^2 
    - \nrm{ \nabla_h \Bigl(  \chi \left( \phi^{m+1} , \phi^m \right)
    	- ( \frac32 \phi^m - \frac12 \phi^{m-1} ) \Bigr)}_2^2  \nonumber 
    \\
    & \ge &  \frac{1}{32} \varepsilon^4 \nrm{\nabla_h \Delta_h 
    	( 3 \phi^{m+1} +\phi^{m-1} )}_2^2  
    - \Bigl(\frac12\nrm{ \nabla_h (3\phi^m -  \phi^{m-1} ) }_2 
    + \nrm{\nabla_h \chi \left( \phi^{m+1} , \phi^m \right) }_2  \Bigr)^2    \nonumber 
    \\
    & \ge & \frac{1}{32}\varepsilon^4  \nrm{\nabla_h \Delta_h 
    	( 3 \phi^{m+1} + \phi^{m-1} )}_2^2  
    - \Bigl( 2 C_2 
    + \nrm{\nabla_h \chi \left( \phi^{m+1} , \phi^m \right) }_2  \Bigr)^2    \nonumber  
    \\
    & \ge & \frac{1}{32}\varepsilon^4\nrm{\nabla_h \Delta_h 
    	( 3 \phi^{m+1} +  \phi^{m-1} )}_2^2  
    - 2 \Bigl(  4 C_2^2 
    + \nrm{\nabla_h \chi \left( \phi^{m+1} , \phi^m \right) }_2^2  \Bigr)  ,  
    \label{estimate-mu-2}
    \end{eqnarray} 
    in which the uniform in time estimate (\ref{H^1-stability}) was utilized in the third step  
    and another quadratic inequality $(a+b)^2 \le 2 (a^2 + b^2)$ 
    was used in the last step. 
    An application of the Cauchy inequality gives an estimate 
    of the leading term in (\ref{estimate-mu-2}): 
    \begin{eqnarray}  
    \nrm{\nabla_h \Delta_h ( 3\phi^{m+1} + \phi^{m-1} )}_2^2
    &  = & 9 \nrm{\nabla_h \Delta_h \phi^{m+1} }_2^2 
    +  6 \left( \nabla_h \Delta_h \phi^{m+1} ,  \nabla_h \Delta_h \phi^{m-1}  \right) 
    +  \nrm{\nabla_h \Delta_h \phi^{m-1} }_2^2  \nonumber 
    \\
    &  \ge & 6 \nrm{\nabla_h \Delta_h \phi^{m+1} }_2^2 
    - 2 \nrm{ \nabla_h \Delta_h \phi^{m-1}  }_2^2  . 
    \label{estimate-mu-3}
    \end{eqnarray}
    
    For the last term in (\ref{estimate-mu-2}),  
    by the definition of $\chi \left( \phi^{m+1} , \phi^m \right)$ 
    in (\ref{2nd approximation-chi}), a detailed expansion and a careful application of 
    discrete H\"older inequality shows that 
    \begin{eqnarray} 
    \nrm{\nabla_h \left( \chi \left( \phi^{m+1} , \phi^m \right) \right)}_2  
    &\le& C_3 (  \nrm{\phi^{m+1}}_{\infty}^2 + \nrm{\phi^m}_{\infty}^2 )(  \nrm{\nabla_h \phi^{m+1}}_2 + \nrm{\nabla_h \phi^m}_2  ) \\
    &\le& C_4 C_2 (  \nrm{\phi^{m+1}}_{\infty}^2 + \nrm{\phi^m}_{\infty}^2 ) , 
    \label{nonlinear-Holder}%
    \end{eqnarray} 
    in which the uniform in time estimate (\ref{H^1-stability}) was used again 
    in the last step. Moreover, the $\| \cdot \|_\infty$ bound of $\phi^k$ can be obtained with an application of a discrete Gagliardo-Nirenberg type inequality:     
    \begin{equation} 
    \nrm{\phi^k}_{\infty}   
    \le  C_5  \left( \nrm{ \phi^k}_{H_h^1}^{\frac34}  \cdot  
    \nrm{\nabla_h \Delta_h \phi^k}_2^{\frac14} + \nrm{\phi^k}_{H_h^1} \right) 
    \le  C_6  C_2^{\frac34}  
    \nrm{\nabla_h \Delta_h \phi^k}_2^{\frac14}  + C_6 C_2,  \quad k = m, m+1 .; 
    \label{maximum bound}
    \end{equation} 
see \cite{chen2015efficient} for a detailed proof. 
    
    Therefore, a substitution of (\ref{maximum bound}) into (\ref{nonlinear-Holder}) yields 
    \begin{eqnarray} 
    \nrm{\nabla_h \Bigl( \chi \left( \phi^{m+1} , \phi^m \right) \Bigr)}_2^2 
    \le C_7 C_2^{5} 
    \Bigl( \nrm{\nabla_h \Delta_h \phi^{m+1}}_2
    +  \nrm{\nabla_h \Delta_h \phi^m}_2 \Bigr) + C_7 C_2^6 . \label{(20)}
    \end{eqnarray}  
    Motivated by the Young inequality $a\cdot b \le C_8  a^2 
    +\alpha  b^2, ~\forall~  a ,  b  >  0  ,  \alpha > 0$ with 
    $a =  C_7C_2^{5}$, 
    $b = \nrm{\nabla_h \Delta_h \phi^{m+1}}_2 +  \nrm{\nabla_h \Delta_h \phi^m}_2$, 
            $\alpha = \frac{1}{128}\varepsilon^4$       
    $b = \nrm{\nabla_h \Delta_h \phi^{m+1}}_2 +  \nrm{\nabla_h \Delta_h \phi^m}_2$, 
    $\alpha = \frac{1}{128}\varepsilon^4$, we arrive at
    \begin{eqnarray} 
    \nrm{\nabla_h \Bigl(  \chi \left( \phi^{m+1} , \phi^m \right) \Bigr)}_2^2 
    & \le & C_9 C_2^{10} 
    + \frac{1}{128} \varepsilon^4 \Bigl( \nrm{\nabla_h \Delta_h \phi^{m+1}}_2 
    +  \nrm{\nabla_h \Delta_h \phi^m}_2  \Bigr)^2  + C_7 C_2^6 \nonumber 
    \\
    & \le & C_9 C_2^{10} + C_7 C_2^6 
    + \frac{1}{64}\varepsilon^4  \Bigl( \nrm{\nabla_h \Delta_h \phi^{m+1}}_2^2 
    +  \nrm{\nabla_h \Delta_h \phi^m}_2^2  \Bigr)   . \label{(22)}
    \end{eqnarray} 

    A combination of (\ref{estimate-mu-2}), (\ref{estimate-mu-3}) and (\ref{(22)})
    shows that 
    \begin{eqnarray} 
    \nrm{\nabla_h \mu^{m+1/2}}_2^2   
    \ge  \frac{5}{32}\varepsilon^4\nrm{\nabla_h \Delta_h \phi^{m+1}}_2^2 
    - \frac{1}{32}\varepsilon^4\nrm{\nabla_h \Delta_h \phi^m}_2^2 
    - \frac{1}{16}\varepsilon^4\nrm{\nabla_h \Delta_h \phi^{m-1}}_2^2
    - C_{10},\label{eqn:C10} 
    \end{eqnarray}
    with $C_{10} = 8 C_2^2 + 2 ( C_9 C_2^{10} + C_7 C_2^6 )$.
    
    Going back to (\ref{leading stability}), we obtain 
    \begin{equation} 
    s \sum_{m=1}^{l} \Bigl( \frac{1}{16}\varepsilon^4  \nrm{\nabla_h \Delta_h \phi^m}_2^2 - C_{10}  \Bigr) 
    \le C_1
    + \frac{3}{32}\varepsilon^4 s   \nrm{\nabla_h \Delta_h \phi^1}_2^2 
    + \frac{1}{16}\varepsilon^4 s  \nrm{\nabla_h \Delta_h \phi^{0}}_2^2 \leq C_{11},   
    \label{eqn:C11} 
    \end{equation} 
in which $C_{11}$ is independent on $h$, with a sufficient regularity assumption for $\phi^0 := \Phi^0$, $\phi^1 := \Phi^1$. Inequality \eqref{eqn:C11} is equivalent to  
    \begin{equation} 
    \frac1{16} \varepsilon^4 s \sum_{m=1}^{l}  \nrm{\nabla_h \Delta_h \phi^m}_2^2 
    \le C_{11} + C_{10} T .     
    \label{(25)} 
    \end{equation} 
    This in turn gives the $\ell^2 (0,T; H_h^3)$ bound of the numerical solution. 
    \end{proof}
    Note that $C_{10}$ and $C_{11}$, given by (\ref{eqn:C10}) 
    and (\ref{eqn:C11}), respectively, only depend on $L_x$, $L_y$, $L_z$, 
    $\varepsilon$ and several Sobolev embedding constants, and independent 
    of final time $T$, $h$ and $s$. 

\section{Optimal rate convergence analysis} \label{sec:convergence}
The convergence analysis is carried out in three steps. Firstly, in section \ref{subsec:error-equations}, we obtain error functions by using a standard consistence analysis. In the following, we provide an estimate for the nonlinear error term in section \ref{subsec:error-stability}. Finally, we recover an a-priori error assumption and present the optimal rate error estimate in sections \ref{subsec:a-priori} and \ref{subsec:error-estimate} , respectively.  
%
	
	\subsection{Error equations and consistency analysis}
	\label{subsec:error-equations}
    We assume that the exact solution has regularity of class $\mathcal{R}$: 
		\begin{equation}
			\phi_{e} \in \mathcal{R} := H^3 (0,T; C^0) \cap H^2 (0,T; C^4) \cap L^\infty (0,T; C^6).
			\label{assumption:regularity.1}
		\end{equation}
		
	To facilitate our error analysis, we need to construct an approximate solution to the chemical potential via the exact solution $\phi_e$. In addition, we note that the exact velocity $\u_e$ is not divergence-free at the discrete level ($\nabla_h\cdot\u_e \ne 0$). To overcome this difficulty, we must also construct an approximate solution to the velocity vector (again through the exact solution), which satisfies the divergence-free conditions at the discrete level. Therefore, we define the cell-centered grid functions
		\begin{align}
			\Gamma^{m+1/2} &\mathop{:=}{}  \chi\left(\Phi^{m+1}, \Phi^m\right) - \Phi_*^{m+1/2} -\varepsilon^2 \Delta_h \left(\frac34 \Phi^{m+1}  + \frac14 \Phi^{m-1}\right), 
			\label{def-Gamma}
			\\
			\quad {\bf U}^{m+1/2} &\mathop{:=}{} - \mathcal{P}_{h} \left( \gamma A_h \Phi_*^{m+1/2} \nabla_h \Gamma^{m+1/2} \right),
			\label{def-U}
			\\
			\Phi_*^{m+1/2} &\mathop{:=}{} \frac32 \Phi^m - \frac12 \Phi^{m-1}, 
		\end{align}
	for $1\le m\le M$, where  $\mathcal{P}_{h}$ is the discrete Helmholtz projection defined in equations (3.2), (3.3) in \cite{chen2015efficient}.  We need to enforce the discrete homogeneous Neumann boundary conditions for the chemical potential: ${\n}\cdot\nabla_h \Gamma^{m+1/2} = 0$, for all $1\le m\le M$, so that, in particular,  $\mathcal{P}_{h} \left( A_h \Phi_*^{m+1/2} \nabla_h \Gamma^{m+1/2} \right)$ is well defined. 
		
	With the assumed regularities, the constructed approximations $\Gamma^{m+1/2}$ and
		${\bf U}^{m+1/2}$ obey the following estimates:
		\begin{equation}
			\| \nabla_h \Gamma^{m+1/2} \|_{\infty} \leq C_{12}, \quad \| {\bf U}^{m+1/2} \|_\infty \leq C_{12},
			\label{regu:mu U}
		\end{equation}
	for $0\le m\le M$, where the constant $C_{12} >0$ is independent of $h>0$ and $s>0$.
		
	It follows that $(\Phi , \Gamma , {\bf U})$ satisfies the numerical scheme with an  $O (s^2 + h^2)$ truncation error:
		\begin{eqnarray}
			\frac{\displaystyle \Phi^{m+1} - \Phi^{m} }{s} &=&
			\Delta_h \Gamma^{m+1/2} -\nabla_{h} \cdot \left(A_h \Phi_*^{m+1/2} {\bf U}^{m+1/2}\right)+\tau^{m+1/2},
			\label{eq:consistency.1}
			\\
			\Gamma^{m+1/2} &=&  \chi\left(\Phi^{m+1}, \Phi^m\right) - \Phi_*^{m+1/2} -\varepsilon^2 \Delta_h \left(\frac34 \Phi^{m+1}  + \frac14 \Phi^{m-1}\right), 
			\label{eq:consistency.2}
			\\
			\quad {\bf U}^{m+1/2} &=& - \mathcal{P}_{h} \left( \gamma A_h \Phi_*^{m+1/2} \nabla_h \Gamma^{m+1/2} \right),
			\label{eq:consistency.3}
		\end{eqnarray}
	where the local truncation error satisfies
		\begin{equation}
			\nrm{\tau^{m+1/2}}_2 \le C_{13} (s^2+h^2),
		\end{equation}
	with $s\cdot M = T$, and $C_{13}$ independent of $h$ and $s$.
		
	The numerical error functions are denoted as
		\begin{equation}
			\tilde{\phi}^{m} := \Phi^{m} - \phi^{m}, \quad \tilde{\mu}^{m+1/2} := \Gamma^{m+1/2}-\mu^{m+1/2}, \quad \tilde{\u}^{m+1/2} := {\bf U}^{m+1/2} - \u^{m+1/2}.
		\end{equation}
	Subtracting (\ref{eq:consistency.1}) -- (\ref{eq:consistency.3}) from (\ref{2nd scheme-CHHS-reformulate-s1}) -- (\ref{2nd scheme-CHHS-reformulate-s3}) yields
		\begin{eqnarray}
			&&\frac{\tilde{\phi}^{m+1}-\tilde{\phi}^{m}}{s} = \Delta_{h}\tilde{\mu}^{m+1/2} - \nabla_h \cdot \left(\!A_{h}\tilde{\phi}_*^{m+1/2} {\bf U}^{m+1/2}\!+\!A_{h}\phi_*^{m+1/2} \tilde{\u}^{m+1/2}\!\right) + \tau^{m+1/2},
			\label{CHHS-2nd-error-1}
			\\
			&&\tilde{\mu}^{m+1/2} = \mathcal{N}^{m+1/2}-\tilde{\phi}_*^{m+1/2}-\varepsilon^{2}\Delta_{h}\tilde{\phi}_I^{m+1/2},
			\label{CHHS-2nd-error-2}
		\end{eqnarray}
		where
		\begin{eqnarray}
			&& \tilde{\phi}_*^{m+1/2} = \frac32 \tilde{\phi}^{m} - \frac12 \tilde{\phi}^{m-1} ,  \, \,  \tilde{\phi}_{I}^{m+1/2} = \frac34 \tilde{\phi}^{m+1} 
			+ \frac14 \tilde{\phi}^{m-1} , \nonumber 
			\\
			&&
			{\cal N}^{m+1/2} = \chi \left( \Phi^{m+1} , \Phi^m \right) 
			-  \chi \left( \phi^{m+1} , \phi^m \right) ,  \nonumber 
			\\
			&&\tilde{\u}^{m+1/2} = -\gamma\mathcal{P}_{h}\left(A_{h}\tilde{\phi}_*^{m+1/2} \nabla_{h}\Gamma^{m+1/2}+A_{h}\phi_*^{m+1/2} \nabla_{h}\tilde{\mu}^{m+1/2}\right),
			\label{CHHS-2nd-error-3}
		\end{eqnarray}
	for $1\le m\le M-1$. 
	
We also observe that $\tilde{\phi}^0 = \tilde{\phi}^1 \equiv 0$, due to our initial value choices $\phi^0 = \Phi^0$, $\phi^1 = \Phi^1$. This fact will facilitate the convergence analysis in later sections. 
	
	\subsection{Stability of the error functions}
	\label{subsec:error-stability}
	Note that both the CHHS equation (\ref{equation-CHHS-1})-(\ref{equation-CHHS-3}) is mass conservative at the continuous level: $\int_\Omega \phi (t) d {\bf x} = \int_\Omega \phi (0) d {\bf x}$, $\forall t >0$, while the numerical scheme (\ref{2nd scheme-CHHS-s1})-(\ref{2nd scheme-CHHS-s3}) is mass conservative at the discrete level: $( \phi^k, {\bf  1} ) = ( \phi^0 , {\bf 1} )$, $\forall k \ge 1$. Consequently, the following estimate is available; the detailed proof could be read in a recent work \cite{chen2015efficient}.  
	\begin{lem}
		\label{lemma:prelim est}
		Assume the exact solution is of regularity class $\mathcal{R}$.  Then, for any $1\le m\le M$,
		\begin{eqnarray}
		\| \tilde{\phi}^m \|_2 &\leq& C_{14} \left(  \| \nabla_h \tilde{\phi}^m \|_2 + h^2 \right) ,
		\label{convergence L2 est-1}
		\\
		\| \tilde{\phi}^m \|_{\infty} &\leq& C_{14} \left( \| \nabla_h \tilde{\phi}^m \|_2^{\frac34} \cdot \| \nabla_h \Delta_h \tilde{\phi}^m \|_2^{\frac14} + \| \nabla_h \tilde{\phi}^m \|_2 + h^2 \right) .
		\label{convergence inf est-1}
		\end{eqnarray}
		for some constant $C_{14}$ that is independent of $s$, $h$, and $m$.
	\end{lem}

	
	Before we carry out the stability analysis for the numerical error functions, we assume that the exact solution $\Phi$ and the constructed solutions $\Gamma$, ${\bf U}$ 
	have the following regularity: 
	\begin{equation} 
	\nrm{\Phi }_{\ell^\infty (0,T; W_h^{1,\infty} )} \le C_{12} ,  \quad 
	\nrm{\nabla_h \Gamma^{m+1/2}}_{\infty} \le C_{12} ,  \quad 
	\nrm{{\bf U}^{m+1/2} }_{\infty} \le C_{12} ,   \quad \forall \, 1 \le m \le M-1 . 
	\label{CHHS-regularity-1}
	\end{equation}
	In addition, we also set the $\| \cdot \|_\infty$ and $H_h^3$ norms (introduced by \eqref{discrete norm-infty} and \eqref{discrete norm-H3}) for the numerical solution $\phi^m$ as 
	\begin{equation} 
	M_0^m := \nrm{\phi^m}_{\infty}  ,  \quad 
	M_3^m := \nrm{\phi^m}_{H_h^3}  . 
	\label{CHHS-a priori-1}
	\end{equation}
	Note that we have an $\ell^\infty (0, T; H_h^1)$ and $\ell^2 (0,T; H_h^3)$ bound for the numerical solution, as given by (\ref{leading stability}), (\ref{(25)}), 
	respectively. Meanwhile, its $\ell^\infty (0,T; \ell^\infty)$ bound is not available at present. 
	This bound will be justified by later analysis.

The following theorem states the stability of the numerical error functions satisfying the error equations by (\ref{CHHS-2nd-error-1}) -- (\ref{CHHS-2nd-error-1}).

	\begin{thm}
	\label{thm:stability}
Assume the exact solution is of regularity class $\mathcal{R}$. 
Then the error function $\tilde{\phi}^m$ obeys the following 
discrete energy stability law: for any $1 \le m \le M-1$,
	\begin{eqnarray}
	 	&&
		\| \nabla_h \tilde{\phi}^{m+1}  \|_2^2 
	- \| \nabla_h \tilde{\phi}^m \|_2^2 
	+ \frac14 ( \| \nabla_h ( \tilde{\phi}^{m+1} - \tilde{\phi}^m )  \|_2^2 
	- \| \nabla_h ( \tilde{\phi}^m - \tilde{\phi}^{m-1} )  \|_2^2 )  
	+ \frac{11}{64} \varepsilon^2 s \| \nabla_h \Delta_h \tilde{\phi}^{m+1} \|_2^2   \nonumber 
	 	\\
	 	&\le& 
		\frac{\varepsilon^2 s}{16}  \| \nabla_h \Delta_h \tilde{\phi}^m \|_2^2  
	+ \frac{5 \varepsilon^2 s}{64}  \| \nabla_h \Delta_h \tilde{\phi}^{m-1} \|_2^2    
	+ 2s \| \tau^{m+1/2} \|_2^2  + sC_{28}D_3^{m+1}h^4 \nonumber 
	 	\\
	 	&& 
	 	+  s(C_{29}D_1^{m+1} + C_{30}D_2^{m+1}) 
		(  \|  \nabla_h \tilde{\phi}^{m+1} \|_2^2 
	+ \|  \nabla_h \tilde{\phi}^m \|_2^2  
	+ \|  \nabla_h \tilde{\phi}^{m-1} \|_2^2 )  , 
	 	\label{eq:stable result}
	\end{eqnarray}
	where 
	\begin{eqnarray}
       D_1^{m+1} &=& ((M_0^m)^{16/3} + (M_0^{m-1})^{16/3})((M_0^{m+1})^{8/3} + (M_0^m)^{8/3} + 1) +1 ,
       \label{Def-D1}
       \\
       D_2^{m+1} &=& ( (M_0^m)^4  + (M_0^{m-1})^4 +1 ) 
       ( (M_0^{m+1})^4  + (M_0^m)^4  +1 ),
       \label{Def-D2}
       \\
       D_3^{m+1} &=& M_0^{m+1}  + M_0^m + 1.
	\end{eqnarray}	
and the constants $C_{28}$, $C_{29}$, $C_{30}$ are given by \eqref{CHHS-2nd-error-9-4-1}-\eqref{CHHS-2nd-error-9-4-3}, respectively. 
\end{thm}

\begin{proof} 	
	Taking inner product of (\ref{CHHS-2nd-error-1}) with 
	$- 2 \Delta_h \tilde{\phi}_{I}^{m+1/2}  = - \Delta_h ( \frac32 \tilde{\phi}^{m+1} + \frac12 \tilde{\phi}^{m-1} )$ gives 
	\begin{eqnarray}
	I_4&:=&
	\nrm{ \nabla_h \tilde{\phi}^{m+1} }_2^2 - \nrm{ \nabla_h \tilde{\phi}^m }_2^2 
	+ \frac14 \Bigl( \nrm{ \nabla_h ( \tilde{\phi}^{m+1} - \tilde{\phi}^m )  }_2^2 
	- \nrm{ \nabla_h ( \tilde{\phi}^m - \tilde{\phi}^{m-1} )  }_2^2  \nonumber 
	\\
	&&
	+ \nrm{ \nabla_h ( \tilde{\phi}^{m+1} - 2 \tilde{\phi}^m 
		+ \tilde{\phi}^{m-1} )  }_2^2  \Bigr)\nonumber 
    \\
	&=&
	-2 s \left( \tau^{m+1/2} ,   \Delta_h \tilde{\phi}_{I}^{m+1/2}  \right) 
	+2 s \left( \nabla_h \Delta_h \tilde{\phi}_{I}^{m+1/2} ,  
		\nabla_h \tilde{\mu}^{m+1/2}  \right) \nonumber
	\\
	&&
	- 2 s \left( \nabla_h \Delta_h \tilde{\phi}_{I}^{m+1/2} , 
	A_h\tilde{\phi}_*^{m+1/2} {\bf U}^{m+1/2}\right)
	- 2 s \left( \nabla_h \Delta_h \tilde{\phi}_{I}^{m+1/2} , 
    A_h\phi_*^{m+1/2} \tilde{\u}^{m+1/2} \right)\nonumber
    \\
    &:=&2s(I_{4,1}+I_{4,2}+I_{4,3}+I_{4,4}).
	\label{CHHS-2nd-error-4}
	\end{eqnarray}
	The term associated with the local truncation error term $I_{4,1}$ in \eqref{CHHS-2nd-error-4} can be bounded 
	in a straightforward way: 
	\begin{eqnarray} 
 	I_{4,1}
	&\le&    \left\| \tau^{m+1/2}  \right\|_2^2 
	+ \frac14 \left\|   \Delta_h \tilde{\phi}_{I}^{m+1/2}   \right\|_2^2 ,   \nonumber 
	\\
	&\le&
	  \left\| \tau^{m+1/2}  \right\|_2^2 
	+ \frac14 \left\|   \nabla_h \tilde{\phi}_{I}^{m+1/2}   \right\|_2  
	\cdot  \left\|  \nabla_h \Delta_h \tilde{\phi}_{I}^{m+1/2}   \right\|_2   \nonumber 
	\\
	&\le& 
	\left\| \tau^{m+1/2}  \right\|_2^2 
	+ \frac{1}{\varepsilon^2} \left\|   \nabla_h \tilde{\phi}_{I}^{m+1/2}   \right\|_2^2  
	+ \frac{\varepsilon^2}{16} \left\|  \nabla_h \Delta_h \tilde{\phi}_{I}^{m+1/2}   \right\|_2^2  .
	\label{CHHS-2nd-error-5}
	\end{eqnarray} 
	
	The regular diffusion term $I_{4,2}$ in \eqref{CHHS-2nd-error-4} has the following decomposition: 
	\begin{eqnarray} 
	I_{4,2}&=&\left( \nabla_h \Delta_h \tilde{\phi}_{I}^{m+1/2} ,  
	\nabla_h \tilde{\mu}^{m+1/2}  \right)\nonumber\\  
	&= &  \left( \nabla_h \Delta_h \tilde{\phi}_{I}^{m+1/2} ,  
	\nabla_h  \left( {\cal N}^{m+1/2}  \right)  \right)
	\nonumber
	\\   
	&-&  \left( \nabla_h \Delta_h \tilde{\phi}_{I}^{m+1/2} ,  \nabla_h \tilde{\phi}_*^{m+1/2} \right) 
	-  \varepsilon^2 \nrm{ \nabla_h \Delta_h \tilde{\phi}_{I}^{m+1/2}  }_2^2\nonumber
	\\
	&:=& I_{4,2,1}+I_{4,2,2}+I_{4,2,3} . 
	\label{CHHS-2nd-error-6-1}
	\end{eqnarray}  
	The concave term $I_{4,2,2}$ in \eqref{CHHS-2nd-error-6-1} can be controlled by 
	\begin{eqnarray} 
	 I_{4,2,2}=- \left( \nabla_h \Delta_h \tilde{\phi}_{I}^{m+1/2} ,  \nabla_h \tilde{\phi}_*^{m+1/2} \right)
	\le 
	\frac{4}{\varepsilon^2}  \nrm{ \nabla_h \tilde{\phi}_*^{m+1/2} }_2^2  
	+ \frac{\varepsilon^2}{16} \nrm{ \nabla_h \Delta_h \tilde{\phi}_{I}^{m+1/2} }_2^2  .
	\label{CHHS-2nd-error-6-2}
	\end{eqnarray} 
	For the nonlinear error term $I_{4,2,1}$ in \eqref{CHHS-2nd-error-4}, we start from the following expansion  
	\begin{eqnarray} 
	{\cal N}^{m+1/2}  = && 
	\frac14 \Bigl(  \left( (\phi^{m+1})^2 + (\phi^m)^2 \right) 
	( \tilde{\phi}^{m+1} + \tilde{\phi}^m ) 
	\nonumber
	\\ 
	&&+  \left( ( \phi^{m+1} +  \Phi^{m+1} ) \tilde{\phi}^{m+1}  + ( \phi^m +  \Phi^m ) \tilde{\phi}^m \right) 
	( \Phi^{m+1} + \Phi^m )  \Bigr) .  
	\label{CHHS-2nd-error-6-3}
	\end{eqnarray} 
	An application of discrete H\"older's inequality to its gradient shows that 
	\begin{eqnarray} 
	\nrm{ \nabla_h {\cal N}^{m+1/2}  }
	&\le&  
	C_{15} \left(  \nrm{\phi^{m+1}  }_{\infty}^2 
	+  \nrm{ \Phi^{m+1} }_{\infty}^2   
	+ \nrm{\phi^m  }_{\infty}^2 
	+  \nrm{ \Phi^m }_{\infty}^2 \right) 
	(  \| \nabla_h \tilde{\phi}^{m+1} \|_2   
	+ \| \nabla_h \tilde{\phi}^m \|_2 )  \nonumber 
	\\
	&& 
	+ C_{15}\left(  \nrm{\phi^{m+1}  }_{\infty} 
	+  \nrm{ \Phi^{m+1} }_{\infty}   
	+ \nrm{\phi^m  }_{\infty} 
	+  \nrm{ \Phi^m }_{\infty} \right)  \nonumber 
	\\
	&&
	\, \, \, 
	\cdot \left(  \nrm{\nabla_h \phi^{m+1}  }
	+  \nrm{ \nabla_h \Phi^{m+1} }_2   
	+ \nrm{ \nabla_h \phi^m  } _2
	+  \nrm{ \nabla_h \Phi^m }_2 \right) 
	(  \| \tilde{\phi}^{m+1} \|_{\infty}   
	+ \| \tilde{\phi}^m \|_{\infty}  )    \nonumber 
	\\
	&\le&  
	C_{15} \left(2C_{12}^2 +  (M_0^{m+1})^2  + (M_0^m)^2  \right)  
	(  \|  \nabla_h \tilde{\phi}^{m+1} \|_2  
	+ \|  \nabla_h \tilde{\phi}^m \|_2  )   \nonumber 
	\\
	&&
	+ 2C_{15} \left(C_{12}+  M_0^{m+1}  + M_0^m \right) 
	( C_2 +  C_{12} ) 
	(  \| \tilde{\phi}^{m+1} \|_{\infty}  + \|  \tilde{\phi}^m \|_{\infty}   )  , 
	\label{CHHS-2nd-error-6-4}
	\end{eqnarray} 
	with the $\ell^\infty (0,T; H_h^1)$ estimate (\ref{H^1-stability}) 
	for the numerical solution, the regularity assumption (\ref{CHHS-regularity-1}) 
	for the exact solution and the a-priori set up (\ref{CHHS-a priori-1}) used. 
	Then we get 
	\begin{eqnarray} 
	I_{4,2,2}&=&
	\left( \nabla_h \Delta_h \tilde{\phi}_{I}^{m+1/2} ,  
	\nabla_h  {\cal N}^{m+1/2}    \right)   \nonumber 
	\\
	&\le&  
	C_{15} \left(2C_{12}^2 +  (M_0^{m+1})^2  + (M_0^m)^2  \right)  
		\left(  \nrm{  \nabla_h \tilde{\phi}^{m+1} }_2  
		+ \nrm{  \nabla_h \tilde{\phi}^m }_2  \right) 
		\cdot \nrm{ \nabla_h \Delta_h \tilde{\phi}_{I}^{m+1/2} }_2
	\nonumber 
	\\
	&&
	+C_{18}    
	\nrm{  \tilde{\phi}^m }_{\infty} 
	\cdot \nrm{ \nabla_h \Delta_h \tilde{\phi}_{I}^{m+1/2} }_2
	+ C_{18}
    \nrm{  \tilde{\phi}^{m+1} }_{\infty}  
	\cdot \nrm{ \nabla_h \Delta_h \tilde{\phi}_{I}^{m+1/2}}_2 , 
	\label{CHHS-2nd-error-6-5}
	\end{eqnarray} 
	with $C_{16} = 2C_{12}C_{15}(C_2+C_{12})$, $C_{17} =2C_{15}(C_2+C_{12})$ and $C_{18} = C_{16} + C_{17}  (  M_0^{m+1}  + M_0^m )$.   
	The first part in \eqref{CHHS-2nd-error-6-5} can be controlled by Cauchy inequality: 
	\begin{eqnarray} 
	&&
	C_{15} \left(2C_{12}^2 +  (M_0^{m+1})^2  + (M_0^m)^2  \right)  
	\left(  \nrm{  \nabla_h \tilde{\phi}^{m+1} }_2  
	+ \nrm{  \nabla_h \tilde{\phi}^m }_2  \right) 
	\cdot \nrm{ \nabla_h \Delta_h \tilde{\phi}_{I}^{m+1/2} }_2  \nonumber 
	\\
	&\le& 
	\frac{\varepsilon^2}{16}  \nrm{ \nabla_h \Delta_h \tilde{\phi}_{I}^{m+1/2} }_2^2 
	+  \frac{C_{19}}{\varepsilon^2}  \left(  \nrm{  \nabla_h \tilde{\phi}^{m+1} }_2^2 
	+ \nrm{  \nabla_h \tilde{\phi}^m }_2^2  \right) ,  
	\label{CHHS-2nd-error-6-6}
	\end{eqnarray} 
	with $C_{19} = 8 C_{15}^2 \left(2C_{12}^2 +  (M_0^{m+1})^2  + (M_0^m)^2  \right)^2$. 
	For the second part in \eqref{CHHS-2nd-error-6-5} , we observe that the maximum norm of the numerical error 
	can be analyzed by an application of Gagliardo-Nirenberg type inequality in 3-D, 
	similar to (\ref{maximum bound}): 
	\begin{eqnarray} 
	\nrm{ \tilde{\phi}^m }_{\infty}   
	\le  C_{14} \left( \| \nabla_h \tilde{\phi}^m \|_2^{\frac34} \| \nabla_h \Delta_h \tilde{\phi}^m \|_2^{\frac14} + \| \nabla_h \tilde{\phi}^m \|_2 + h^2 \right)  . 
	\label{CHHS-2nd-error-6-7}
	\end{eqnarray} 
    With an application of the Young inequality to the second part in \eqref{CHHS-2nd-error-6-5}, we arrive at 
	\begin{eqnarray} 
	&&
	C_{18} \nrm{ \tilde{\phi}^m }_{\infty}   
	\cdot \nrm{ \nabla_h \Delta_h \tilde{\phi}_{I}^{m+1/2} }_2
	\nonumber
	\\
	&\le& C_{14}C_{18} \left( \nrm{ \nabla_h \tilde{\phi}^m}_2^{\frac34}  \cdot  
	\nrm{\nabla_h \Delta_h \tilde{\phi}^m}_2^{\frac14}  
	+ \nrm{\nabla_h \tilde{\phi}^m}_2 + h^2\right) 
	\cdot \nrm{ \nabla_h \Delta_h \tilde{\phi}_{I}^{m+1/2} }_2  \nonumber 
	\\
	&\le& 
	C^{\alpha, \varepsilon}_1  (C_{14}C_{18})^{8/3} \nrm{ \nabla_h \tilde{\phi}^m}_2^2    
	+ \alpha \varepsilon^2 \left( \nrm{\nabla_h \Delta_h \tilde{\phi}^m}_2^{2/5}  
	\cdot \nrm{ \nabla_h \Delta_h \tilde{\phi}_{I}^{m+1/2} }_2^{8/5}  \right)
	\nonumber
	\\
	&&+ \frac{16(C_{14}C_{18})^2}{\varepsilon^2}  \nrm{\nabla_h \tilde{\phi}^m}_2^2 +  \frac{16(C_{14}C_{18})^2}{\varepsilon^2}h^4
	+ \frac{\varepsilon^2}{32} 
	\nrm{ \nabla_h \Delta_h \tilde{\phi}_{I}^{m+1/2} }_2^2  \nonumber  
	\\
	&\le& 
	C^{\alpha, \varepsilon}_2  ( C_{18}^{8/3} + C_{18}^2 ) \nrm{ \nabla_h \tilde{\phi}^m}_2^2  
	+ \frac{\varepsilon^2}{32} \nrm{ \nabla_h \Delta_h \tilde{\phi}_{I}^{m+1/2} }_2^2  
	+ \frac{16(C_{14}C_{18})^2}{\varepsilon^2}h^4 
	\nonumber
	\\
	&&+ \alpha \varepsilon^2 \left( \nrm{\nabla_h \Delta_h \tilde{\phi}^m}_2^{2/5}  
	\cdot \nrm{ \nabla_h \Delta_h \tilde{\phi}_{I}^{m+1/2} }_2^{8/5}  \right)   ,  
	\label{CHHS-2nd-error-6-8}
	\end{eqnarray} 
	for any $\alpha > 0$. Furthermore, the last term appearing in 
	(\ref{CHHS-2nd-error-6-8}) can also be handled by Young's inequality: 
	\begin{eqnarray} 
	\alpha \varepsilon^2\nrm{\nabla_h \Delta_h \tilde{\phi}^m}_2^{2/5}  
	\cdot \nrm{ \nabla_h \Delta_h \tilde{\phi}_{I}^{m+1/2} }_2^{8/5}  
	\le \frac{1}{5}\alpha \varepsilon^2 \nrm{\nabla_h \Delta_h \tilde{\phi}^m}_2^2  
	+ \frac{4}{5}\alpha \varepsilon^2 \nrm{ \nabla_h \Delta_h \tilde{\phi}_{I}^{m+1/2} }_2^2 .  
	\label{CHHS-2nd-error-6-8-2}
	\end{eqnarray} 
	We can always choose an $\alpha$, such that $\frac{1}{5}\alpha \le \frac{1}{64}$ and $\frac{4}{5}\alpha \le \frac{1}{32}$, so that the following bound is available: 
	\begin{eqnarray} 
	C_{18} \nrm{ \tilde{\phi}^m }_{\infty}   
	\cdot \nrm{ \nabla_h \Delta_h \tilde{\phi}_{I}^{m+1/2} }_2
	\le && C^{\alpha, \varepsilon}_2  ( C_{18}^{8/3} + C_{18}^2 ) \nrm{ \nabla_h \tilde{\phi}^m}_2^2  
	+ \frac{\varepsilon^2}{16} \nrm{ \nabla_h \Delta_h \tilde{\phi}_{I}^{m+1/2} }_2^2 
	\nonumber
	\\ 
	&&+ \frac{\varepsilon^2}{64}  \nrm{\nabla_h \Delta_h \tilde{\phi}^m}_2^2 + \frac{16(C_{14}C_{18})^2}{\varepsilon^2}h^4 .  
	\label{CHHS-2nd-error-6-8-3}
	\end{eqnarray} 
	A similar estimate for the third part in \eqref{CHHS-2nd-error-6-5} can also be derived as 
	\begin{eqnarray} 
	C_{18} \nrm{ \tilde{\phi}^{m+1} }_{\infty}   
	\cdot \nrm{ \nabla_h \Delta_h \tilde{\phi}_{I}^{m+1/2} }_2
	\le &&C^{\alpha, \varepsilon}_2  ( C_{18}^{8/3} + C_{18}^2 ) 
	\nrm{ \nabla_h \tilde{\phi}^{m+1}}_2^2  
	+ \frac{\varepsilon^2}{16} \nrm{ \nabla_h \Delta_h \tilde{\phi}_{I}^{m+1/2} }_2^2  
	\nonumber
	\\
	&&+ \frac{\varepsilon^2}{64}  \nrm{\nabla_h \Delta_h \tilde{\phi}^{m+1} }_2^2 + \frac{16(C_{14}C_{18})^2}{\varepsilon^2}h^4 .  
	\label{CHHS-2nd-error-6-8-4}
	\end{eqnarray} 
	Consequently, a combination of (\ref{CHHS-2nd-error-6-5}), 
	(\ref{CHHS-2nd-error-6-6}), (\ref{CHHS-2nd-error-6-8-3}) and (\ref{CHHS-2nd-error-6-8-4}) yields 
	\begin{eqnarray} 
	I_{4,2,2}  
	&\le&  
	\left( C^{\alpha, \varepsilon}_2  ( C_{18}^{8/3} + C_{18}^2 )  
	+ C_{19} \varepsilon^{-2} \right) 
	(  \|  \nabla_h \tilde{\phi}^{m+1} \|_2^2 
	+ \|  \nabla_h \tilde{\phi}^m \|_2^2  )  \nonumber 
	\\
	&&
	+ \frac{3 \varepsilon^2}{16} \nrm{ \nabla_h \Delta_h \tilde{\phi}_{I}^{m+1/2} }_2^2   
	+ \frac{\varepsilon^2}{64}  ( \| \nabla_h \Delta_h \tilde{\phi}^m \|_2^2 
	+ \| \nabla_h \Delta_h \tilde{\phi}^{m+1} \|_2^2 )
	\nonumber
	\\
	&&+ \frac{32(C_{14}C_{18})^2}{\varepsilon^2}h^4 , 
	\label{CHHS-2nd-error-6-10}
	\end{eqnarray} 
	Note that $C_{19}$ is involved with $(M_0^m)^4$ and $(M_0^{m+1})^4$, while $C_{18}$ is involved with $M_0^m$ and $M_0^{m+1}$. 
	As a result, a combination of (\ref{CHHS-2nd-error-6-1}), 
	(\ref{CHHS-2nd-error-6-2}) and (\ref{CHHS-2nd-error-6-10}) shows that 
	\begin{eqnarray} 
	I_{4,2}
	&\le& 
	\left( C^{\alpha, \varepsilon}_2  ( C_{18}^{8/3} + C_{18}^2 )  
	+ C_{19} \varepsilon^{-2} \right) 
	(  \|  \nabla_h \tilde{\phi}^{m+1} \|_2^2 
	+ \|  \nabla_h \tilde{\phi}^m \|_2^2  
	+ \|  \nabla_h \tilde{\phi}^{m-1} \|_2^2 )   \nonumber 
	\\
	&& 
	- \frac{3 \varepsilon^2}{4} \nrm{ \nabla_h \Delta_h \tilde{\phi}_{I}^{m+1/2} }_2^2    
	+ \frac{\varepsilon^2}{64}  ( \| \nabla_h \Delta_h \tilde{\phi}^m \|_2^2 
	+ \| \nabla_h \Delta_h \tilde{\phi}^{m+1} \|_2^2 )
    + \frac{32(C_{14}C_{18})^2}{\varepsilon^2}h^4  .
	\label{CHHS-2nd-error-6-11}
	\end{eqnarray} 
	
	Next we focus our attention on the terms associated with 
	the convection term and the highest order nonlinear diffusion. 
	The analysis of this part is highly non-trivial. 
	
	The $I_{4,3}$ in \eqref{CHHS-2nd-error-4} can be bounded by 
	\begin{eqnarray} 
	I_{4,3}&=&
	- \left( \nabla_h \Delta_h \tilde{\phi}_{I}^{m+1/2} , 
	A_h\tilde{\phi}_*^{m+1/2} {\bf U}^{m+1/2}  \right) \nonumber\\
	&\le&  \| \nabla_h \Delta_h \tilde{\phi}_{I}^{m+1/2} \|_2 
	\cdot \| \tilde{\phi}_*^{m+1/2} \|_2  \cdot  \| {\bf U}^{m+1/2} \|_{\infty}  \nonumber 
	\\
	&\le& 
	 C_{12}  \| \nabla_h \Delta_h \tilde{\phi}_{I}^{m+1/2} \|_2 
	\cdot \| \tilde{\phi}_*^{m+1/2} \|_2   \nonumber 
	\\
	&\le&
	\frac{\varepsilon^2}{16} \| \nabla_h \Delta_h \tilde{\phi}_{I}^{m+1/2} \|_2^2 
	+  \frac{4 C_{12}^2}{\varepsilon^2} \| \tilde{\phi}_*^{m+1/2} \|_2^2  \nonumber 
	\\
	&\le& 
	\frac{\varepsilon^2}{16} \| \nabla_h \Delta_h \tilde{\phi}_{I}^{m+1/2} \|_2^2 
	+  \frac{4 C_{12}^2 C_{20}^2}{\varepsilon^2} 
	\left( \| \nabla_h \tilde{\phi}_*^{m+1/2} \|_2^2  + h^4 \right) \nonumber 
	\\
	&\le& 
	\frac{\varepsilon^2}{16} \| \nabla_h \Delta_h \tilde{\phi}_{I}^{m+1/2} \|_2^2 
	+  \frac{18 C_{12}^2 C_{20}^2}{\varepsilon^2} 
	( \| \nabla_h \tilde{\phi}^m \|_2^2 
	+ \| \nabla_h \tilde{\phi}^{m-1} \|_2^2 + h^4 )  ,  
	\label{CHHS-2nd-error-7}
	\end{eqnarray} 
in which $C_{20} = \sqrt{2} C_{14}$, so that the inequality $\| \tilde{\phi}_*^{m+1/2} \|_2^2 \le C_{20}^2 ( \| \nabla_h \tilde{\phi}_*^{m+1/2} \|_2^2 + h^4 )$ is a direct consequence of estimate (\ref{convergence L2 est-1}) in Lemma~\ref{lemma:prelim est}.	Note that the 
	regularity assumption (\ref{CHHS-regularity-1}) for the constructed solution 
	${\bf U}$ is used in the derivation. 
	
	For the term $I_{4,4}$ in \eqref{CHHS-2nd-error-4}, the expansion (\ref{CHHS-2nd-error-3}) for the velocity 
	numerical error indicates that 
	\begin{eqnarray} 	
	I_{4,4}
	&=&   - \left( A_h\phi_*^{m+1/2} \nabla_h \Delta_h \tilde{\phi}_{I}^{m+1/2} ,
	\tilde{\u}^{m+1/2}  \right)   \nonumber 
	\\
	&=&
	\gamma \left( A_h\phi_*^{m+1/2} \nabla_h \Delta_h \tilde{\phi}_{I}^{m+1/2} ,
	{\cal P}_h  \left( A_h\tilde{\phi}_*^{m+1/2} \nabla_h \Gamma^{m+1/2}  \right)  \right) 
	\nonumber
	\\ 
	&+&  \gamma \left( A_h\phi_*^{m+1/2} \nabla_h \Delta_h \tilde{\phi}_{I}^{m+1/2} ,
	{\cal P}_h  \left( A_h\phi_*^{m+1/2} \nabla_h \tilde{\mu}^{m+1/2}  \right)  \right)
	\nonumber
	\\
	&:=& I_{4,4,1}+I_{4,4,2}.
	\label{CHHS-2nd-error-8-1}
	\end{eqnarray} 
	The first term $I_{4,2,1}$ in \eqref{CHHS-2nd-error-8-1} can be estimated in a standard way: 
	\begin{eqnarray} 
   I_{4,4,1}
	&\le& \gamma  \nrm{ \phi_*^{m+1/2} }_{\infty} \cdot 
	\nrm{ \nabla_h \Delta_h \tilde{\phi}_{I}^{m+1/2} }_2 \cdot 
	\nrm{ {\cal P}_h  \left(A_h \tilde{\phi}_*^{m+1/2} 
		\nabla_h \Gamma^{m+1/2}  \right)  }_2    \nonumber 
	\\
	&\le& 
	C_{21} \gamma  ( M_0^m + M_0^{m-1} ) 
	\nrm{ \nabla_h \Delta_h \tilde{\phi}_{I}^{m+1/2} }_2 \cdot 
	\nrm{ A_h\tilde{\phi}_*^{m+1/2} \nabla_h \Gamma^{m+1/2}  }_2    \nonumber 
	\\
	&\le& 
	C_{21} \gamma  ( M_0^m + M_0^{m-1} ) 
	\nrm{ \nabla_h \Delta_h \tilde{\phi}_{I}^{m+1/2} }_2 \cdot 
	\nrm{ \tilde{\phi}_*^{m+1/2} }_2  
	\cdot  \nrm{ \nabla_h \Gamma^{m+1/2}  }_{\infty}    \nonumber 
	\\
	&\le& 
	C_{12} C_{21} \gamma  ( M_0^m + M_0^{m-1} ) 
	\nrm{ \nabla_h \Delta_h \tilde{\phi}_{I}^{m+1/2} }_2 
	\cdot \nrm{ \tilde{\phi}_*^{m+1/2} }_2   \nonumber 
	\\
	&\le& 
	 C_{12}C_{14}C_{21} \gamma ( M_0^m + M_0^{m-1} ) 
	\| \nabla_h \Delta_h \tilde{\phi}_{I}^{m+1/2} \|_2 \cdot 
	( \| \nabla_h \tilde{\phi}_*^{m+1/2} \|_2 + h^2 )  \nonumber 
	\\
	&\le& 
	\frac{C_{22} }{\varepsilon^2}   
	( \| \nabla_h \tilde{\phi}^m \|_2^2  
	+ \| \nabla_h \tilde{\phi}^{m-1} \|_2^2  + h^4 ) 
	+ \frac{1}{16} \varepsilon^2 \| \nabla_h \Delta_h \tilde{\phi}_{I}^{m+1/2} \|_2^2 , 
	\label{CHHS-2nd-error-8-2}
	\end{eqnarray} 
	where $C_{22}=72 C_{12}^2C_{14}^2C_{21}^2 \gamma^2 ( (M_0^m)^2 + (M_0^{m-1})^2) $ in which we used the property $\nrm{ {\cal P}_h \v }_2 \le \nrm{ \v }_2$, 
	$\forall \v \in L^2$, for the Helmholtz projection operator ${\cal P}_h$, 
	in the second step \cite{chen2015efficient}. Note that $(M_0^m)^2$ and $(M_0^{m-1})^2$ are involved in the growth coefficient. 
	
	The second term $I_{4,2,2}$ in \eqref{CHHS-2nd-error-8-1} can be expanded as 
	\begin{eqnarray} 
	I_{4,4,2}&=&
	\gamma \left(A_h \phi_*^{m+1/2} \nabla_h \Delta_h \tilde{\phi}_{I}^{m+1/2} ,
	{\cal P}_h  \left( A_h\phi_*^{m+1/2} \nabla_h \tilde{\mu}^{m+1/2}  
	\right)  \right)   \nonumber 
	\\
	&=&  \gamma \left( A_h\phi_*^{m+1/2} \nabla_h \Delta_h \tilde{\phi}_{I}^{m+1/2} ,
	{\cal P}_h  \left(A_h \phi_*^{m+1/2} \nabla_h \left(  {\cal N}^{m+1/2}  \right)  \right)  \right) 
	\nonumber
	\\
	&&-  \gamma \left(A_h \phi_*^{m+1/2} \nabla_h \Delta_h \tilde{\phi}_{I}^{m+1/2} ,
	{\cal P}_h  \left(A_h \phi_*^{m+1/2} \nabla_h \tilde{\phi}_*^{m+1/2}  \right)  \right)    \nonumber 
	\\
	&&
	- \gamma \varepsilon^2 \left(A_h \phi_*^{m+1/2} 
	\nabla_h \Delta_h \tilde{\phi}_{I}^{m+1/2} ,
	{\cal P}_h  \left(A_h \phi_*^{m+1/2} \nabla_h \Delta_h \tilde{\phi}_{I}^{m+1/2}  \right)  \right)\nonumber
	\\
	&:=&\gamma\left( I_{4,4,2,1}+I_{4,4,2,2}+ \varepsilon^2 I_{4,4,2,3}\right) . 
	\label{CHHS-2nd-error-8-3}
	\end{eqnarray} 
	It is observed that the third term $I_{4,4,2,3}$ in \eqref{CHHS-2nd-error-8-3}, which corresponds to the highest order 
	nonlinear diffusion, is always non-positive: 
	\begin{eqnarray} 
	I_{4,4,2,3}
	&=&  -  \nrm{ {\cal P}_h  \left( A_h\phi_*^{m+1/2} 
		\nabla_h \Delta_h \tilde{\phi}_{I}^{m+1/2}  \right) }_2^2
	\le 0 , 
	\label{CHHS-2nd-error-8-4}
	\end{eqnarray} 
	based on the identity $\left( \u , {\cal P}_h \v \right) 
	= \left( {\cal P}_h \u , {\cal P}_h \v \right)$ for any vector $\u, \v \in L^2$. 
	The above inequality is the key reason for an $\ell^\infty (0,T; H_h^1)$ error 
	estimate instead of the standard $\ell^\infty (0,T; \ell^2)$ one. 
	
	The analysis for second term $I_{4,4,2,2}$ in \eqref{CHHS-2nd-error-8-3} is straightforward: 
	\begin{eqnarray} 
	I_{4,4,2,2}
	&=&
	-  \left( A_h\phi_*^{m+1/2} \nabla_h \Delta_h \tilde{\phi}_{I}^{m+1/2} ,
	{\cal P}_h  \left(A_h \phi_*^{m+1/2} \nabla_h \tilde{\phi}_{I}^m  \right)  \right)  \nonumber 
	\\
	&\le&
	\nrm{ A_h\phi_*^{m+1/2} \nabla_h \Delta_h \tilde{\phi}_{I}^{m+1/2} }_2 
	\cdot  \nrm{ {\cal P}_h  \left(A_h \phi_*^{m+1/2} \nabla_h \tilde{\phi}_*^{m+1/2}  \right)  }_2 
	\nonumber
	\\ 
	&\le& \nrm{ A_h\phi_*^{m+1/2} \nabla_h \Delta_h \tilde{\phi}_{I}^{m+1/2} }_2 
	\cdot  \nrm{ A_h\phi_*^{m+1/2} \nabla_h \tilde{\phi}_*^{m+1/2} }_2   \nonumber 
	\\
	&\le&
	\nrm{ \phi_*^{m+1/2} }_{\infty}^2 
	\cdot \nrm{ \nabla_h \Delta_h \tilde{\phi}_{I}^{m+1/2} }_2 
	\cdot  \nrm{ \nabla_h \tilde{\phi}_*^{m+1/2} }_2 \nonumber
	\\
	&=&  C_{23} ( (M_0^m)^2  + (M_0^{m-1} )^2 ) 
	\nrm{ \nabla_h \Delta_h \tilde{\phi}_{I}^{m+1/2} }_2 
	\cdot  \nrm{ \nabla_h \tilde{\phi}_*^{m+1/2} }_2     \nonumber 
	\\
	&\le& 
	\frac{\varepsilon^2}{16 \gamma}  
	\nrm{ \nabla_h \Delta_h \tilde{\phi}_{I}^{m+1/2} }_2^2 
	+ \frac{18C_{23}^2 \gamma ( (M_0^m)^4 + (M_0^{m-1})^4 ) }{\varepsilon^2}  
	( \| \nabla_h \tilde{\phi}^m \|_2^2 
	+  \| \nabla_h \tilde{\phi}^{m-1} \|_2^2  )  .
	\label{CHHS-2nd-error-8-5}
	\end{eqnarray} 
	Also note that $(M_0^m)^4$ and $(M_0^{m-1})^4$ are involved in this growth coefficient. 
	
	For the first term $I_{4,4,2,1}$ of (\ref{CHHS-2nd-error-8-3}), we start from an application 
	of Cauchy inequality and discrete H\"older's inequality: 
	\begin{eqnarray} 
	I_{4,4,2,1}&=&
	\left( A_h\phi_*^{m+1/2} \nabla_h \Delta_h \tilde{\phi}_{I}^{m+1/2} ,
	{\cal P}_h  \left(A_h \phi_*^{m+1/2} \nabla_h \left(  {\cal N}^{m+1/2}  \right)  \right)  \right)   \nonumber 
	\\
	&\le&  
	\nrm{ A_h\phi_*^{m+1/2} \nabla_h \Delta_h \tilde{\phi}_{I}^{m+1/2} }_2  \cdot 
	\nrm{ {\cal P}_h  \left(A_h \phi_*^{m+1/2} \nabla_h \left(  {\cal N}^{m+1/2}  
		\right)  \right)  }_2    \nonumber 
	\\
	&\le&
	\nrm{ A_h\phi_*^{m+1/2} \nabla_h \Delta_h \tilde{\phi}_{I}^{m+1/2} }_2  \cdot 
	\nrm{ A_h\phi_*^{m+1/2} \nabla_h \left(  {\cal N}^{m+1/2}  \right)  }_2   \nonumber 
	\\
	&\le&
	   \nrm{ \phi_*^{m+1/2} }_{\infty}^2  
	\cdot \nrm{ \nabla_h \Delta_h \tilde{\phi}_{I}^{m+1/2} }_2  
	\cdot \nrm{ \nabla_h \left(  {\cal N}^{m+1/2}  \right)  }_2   \nonumber 
	\\
	&\le&
	C_{24} ( (M_0^m)^2  + (M_0^{m-1})^2 ) 
	\nrm{ \nabla_h \Delta_h \tilde{\phi}_{I}^{m+1/2} }_2  
	\cdot \nrm{ \nabla_h \left(  {\cal N}^{m+1/2}  \right)  }_2  .
	\label{CHHS-2nd-error-8-6}
	\end{eqnarray} 
	The rest estimates are very similar to those for the regular diffusion. 
	The inequalities (\ref{CHHS-2nd-error-6-4}) shows that 
	\begin{eqnarray} 
	&&
	C_{24} ( (M_0^m)^2  + (M_0^{m-1})^2 ) 
	\nrm{ \nabla_h \Delta_h \tilde{\phi}_{I}^{m+1/2} }_2  
	\cdot \nrm{ \nabla_h \left(  {\cal N}^{m+1/2}  \right)  }_2  \nonumber 
	\\
	&\le&  
	C_{15}C_{24}( (M_0^m)^2  + (M_0^{m-1})^2 )  
	\left(2C_{12}^2+  (M_0^{m+1})^2  + (M_0^m)^2  \right)\cdot\nonumber
	\\
	&&  
	(  \|  \nabla_h \tilde{\phi}^{m+1} \|_2  
	+ \|  \nabla_h \tilde{\phi}^m \|_2  ) 
	\cdot \nrm{ \nabla_h \Delta_h \tilde{\phi}_{I}^{m+1/2} }_2
  \nonumber 
	\\
	&&
	+ C_{25}   
	\|  \tilde{\phi}^m \|_{\infty}
	\cdot \nrm{ \nabla_h \Delta_h \tilde{\phi}_{I}^{m+1/2} }_2
	+ C_{25}
	\|  \tilde{\phi}^{m+1} \|_{\infty}
	\cdot \nrm{ \nabla_h \Delta_h \tilde{\phi}_{I}^{m+1/2} }_2, 
	\label{CHHS-2nd-error-8-7}
	\end{eqnarray}
	where 
	\begin{eqnarray}\label{Def-C25}
	C_{25}=C_{24} ( (M_0^m)^2  + (M_0^{m-1})^2 )  
		\left( C_{16} + C_{17}  (  M_0^{m+1}  + M_0^m )  \right). 
	\end{eqnarray}		
	The bound for the first term in \eqref{CHHS-2nd-error-8-7} can be derived in the same manner 
	as in (\ref{CHHS-2nd-error-6-6})
	\begin{eqnarray} 
	&&
	C_{15}C_{24}( (M_0^m)^2  + (M_0^{m-1})^2 )  
	\left(2C_{12}^2+  (M_0^{m+1})^2  + (M_0^m)^2  \right)\cdot 
	\nonumber
	\\
	&&
	(  \|  \nabla_h \tilde{\phi}^{m+1} \|_2  
	+ \|  \nabla_h \tilde{\phi}^m \|_2  )
	\cdot \nrm{ \nabla_h \Delta_h \tilde{\phi}_{I}^{m+1/2} }_2  \nonumber 
	\\
	&\le& 
	\frac{\varepsilon^2}{16\gamma}  \nrm{ \nabla_h \Delta_h \tilde{\phi}_{I}^{m+1/2} }_2^2 
	+  \frac{C_{26}}{\varepsilon^2}  (  \|  \nabla_h \tilde{\phi}^{m+1} \|_2^2 
	+ \|  \nabla_h \tilde{\phi}^m \|_2^2  ) ,  
	\label{CHHS-2nd-error-8-8}
	\end{eqnarray} 
	with $C_{26} = C_{27}C_{15}^2C_{24}^2\gamma( (M_0^m)^4  + (M_0^{m-1})^4 ) 
	\left(2C_{12}^4 +  (M_0^{m+1})^4  + (M_0^m)^4  \right) $.
	 
	The bound for the second and third terms appearing in (\ref{CHHS-2nd-error-8-7}) 
	follows form the proof of (\ref{CHHS-2nd-error-6-7})-(\ref{CHHS-2nd-error-6-10}). Hence, the following two estimates are available, and the details are skipped for simplicity of presentation: 
	\begin{eqnarray} 
	C_{25} \| \tilde{\phi}^m \|_{\infty}   
	\cdot \nrm{ \nabla_h \Delta_h \tilde{\phi}_{I}^{m+1/2} }_2
	&\le& C^{\alpha, \varepsilon,\gamma}_1  ( C_{25}^{8/3} + C_{25}^2 ) \| \nabla_h \tilde{\phi}^m \|_2^2  
	+ \frac{\varepsilon^2}{16\gamma} \nrm{ \nabla_h \Delta_h \tilde{\phi}_{I}^{m+1/2} }_2^2  
	\nonumber
	\\
	&&+ \frac{\varepsilon^2}{64\gamma}  \| \nabla_h \Delta_h \tilde{\phi}^m \|_2^2 + \frac{16\gamma(C_{14}C_{18})^2}{\varepsilon^2}h^4   ,   
	\label{CHHS-2nd-error-8-9-1}
	\\
	C_{25} \| \tilde{\phi}^{m+1} \|_{\infty}   
	\cdot \nrm{ \nabla_h \Delta_h \tilde{\phi}_{I}^{m+1/2} }_2
	&\le& C^{\alpha, \varepsilon,\gamma}_1  ( C_{25}^{8/3} + C_{25}^2 ) 
	\| \nabla_h \tilde{\phi}^{m+1} \|_2^2  
	+ \frac{\varepsilon^2}{16\gamma} \nrm{ \nabla_h \Delta_h \tilde{\phi}_{I}^{m+1/2} }_2^2
	\nonumber
	\\  
	&&+ \frac{\varepsilon^2}{64\gamma}  \| \nabla_h \Delta_h \tilde{\phi}^{m+1} \|_2^2  + \frac{16\gamma(C_{14}C_{18})^2}{\varepsilon^2}h^4  .  
	\label{CHHS-2nd-error-8-9-2}
	\end{eqnarray} 
	Going back to (\ref{CHHS-2nd-error-8-6})-(\ref{CHHS-2nd-error-8-7}), we arrive at
	\begin{eqnarray} 
	I_{4,4,2,1}&=&
	\left( A_h\phi_*^{m+1/2} \nabla_h \Delta_h \tilde{\phi}_{I}^{m+1/2} ,
	{\cal P}_h  \left(A_h \phi_*^{m+1/2} \nabla_h \left(  {\cal N}^{m+1/2}  \right)  \right)  \right)   \nonumber 
	\\
	&\le&
	\left( C^{\alpha, \varepsilon, \gamma}_1  ( C_{25}^{8/3} + C_{25}^2 )  
	+ C_{26} \varepsilon^{-2} \right) 
	(  \|  \nabla_h \tilde{\phi}^{m+1} \|_2^2 
	+ \|  \nabla_h \tilde{\phi}^m \|_2^2  )  + \frac{3 \varepsilon^2}{16\gamma} \nrm{ \nabla_h \Delta_h \tilde{\phi}_{I}^{m+1/2} }_2^2\nonumber 
	\\
	&&	   
	+ \frac{\varepsilon^2}{64\gamma}  ( \| \nabla_h \Delta_h \tilde{\phi}^m \|_2^2 
	+ \| \nabla_h \Delta_h \tilde{\phi}^{m+1} \|_2^2 ) + \frac{32\gamma(C_{14}C_{18})^2}{\varepsilon^2}h^4 .
	\label{CHHS-2nd-error-8-10}
	\end{eqnarray} 
	Note that $C_{26}$ and $C_{25}^{8/3}$ are involved with $(M_0^{m+1})^8$, $(M_0^m)^8$ and $(M_0^{m-1})^8$. 
	Consequently, a combination of 
	(\ref{CHHS-2nd-error-8-1})-(\ref{CHHS-2nd-error-8-5}) 
	and (\ref{CHHS-2nd-error-8-10}) shows that 
	\begin{eqnarray} 
	  I_{4,4}&=& 
	- \left( A_h\nabla_h \Delta_h \tilde{\phi}_{I}^{m+1/2} , 
	A_h\phi_*^{m+1/2} \tilde{\u}^{m+1/2}  \right)   \nonumber 
	\\
	&\le& 
	\left( C^{\alpha, \varepsilon,\gamma}_2  ( C_{25}^{8/3} + C_{25}^2 )  
	+ C_{26} \varepsilon^{-2} \right) 
	(  \|  \nabla_h \tilde{\phi}^{m+1} \|_2^2 
	+ \|  \nabla_h \tilde{\phi}^m \|_2^2  
	+ \|  \nabla_h \tilde{\phi}^{m-1} \|_2^2 )  \nonumber 
	\\
	&+& \frac{5 \varepsilon^2}{16} \nrm{ \nabla_h \Delta_h \tilde{\phi}_{I}^{m+1/2} }_2^2   
	+ \frac{\varepsilon^2}{64}  ( \| \nabla_h \Delta_h \tilde{\phi}^m \|_2^2 
	+ \| \nabla_h \Delta_h \tilde{\phi}^{m+1} \|_2^2 )
	+ \frac{32\gamma^2(C_{14}C_{18})^2}{\varepsilon^2}h^4 ,  
	\label{CHHS-2nd-error-8-11}
	\end{eqnarray} 
	with $C_{27} = C^{\alpha, \varepsilon,\gamma}_3  ( (M_0^m)^4  + (M_0^{m-1})^4 +1 ) 
	(2C_{12}^4 +  (M_0^{m+1})^4  + (M_0^m)^4  +1 ) $. 
	
	Consequently, from (\ref{CHHS-2nd-error-4}), (\ref{CHHS-2nd-error-5}), 
	(\ref{CHHS-2nd-error-6-11}), (\ref{CHHS-2nd-error-7}) and 
	(\ref{CHHS-2nd-error-8-11}), we obtain 
	\begin{eqnarray}
	&&
	\| \nabla_h \tilde{\phi}^{m+1}  \|_2^2 
	- \| \nabla_h \tilde{\phi}^m \|_2^2 
	+ \frac14 ( \| \nabla_h ( \tilde{\phi}^{m+1} - \tilde{\phi}^m )  \|_2^2 
	- \| \nabla_h ( \tilde{\phi}^m - \tilde{\phi}^{m-1} )  \|_2^2 )  
	+ \frac58 \varepsilon^2 s \nrm{ \nabla_h \Delta_h \tilde{\phi}_{I}^{m+1/2} }_2^2  \nonumber 
	\\
	&\le& 
	\frac{\varepsilon^2 s}{16}  ( \| \nabla_h \Delta_h \tilde{\phi}^m \|_2^2 
	+ \| \nabla_h \Delta_h \tilde{\phi}^{m+1} \|_2^2 ) 
	+ 2s \| \tau^{m+1/2} \|_2^2 + \frac{64(1+\gamma^2)(C_{14}C_{18})^2}{\varepsilon^2}sh^4 \nonumber 
	\\
	&& 
	+  s\left( C^{\alpha, \varepsilon,\gamma}_4  ( C_{25}^{8/3} + C_{25}^2 +1 )  
	+ C_{27} \varepsilon^{-2}  \right) 
	(  \|  \nabla_h \tilde{\phi}^{m+1} \|_2^2 
	+ \|  \nabla_h \tilde{\phi}^m \|_2^2  
	+ \|  \nabla_h \tilde{\phi}^{m-1} \|_2^2 ).
	\label{CHHS-2nd-error-9-1}
	\end{eqnarray} 
	On the other hand, a similar estimate as (\ref{estimate-mu-3}) could be 
	carried out: 
	\begin{eqnarray}  
	\| \nabla_h \Delta_h \tilde{\phi}_{I}^{m+1/2} \|_2^2 
	= \nrm{\nabla_h \Delta_h ( \frac34 \tilde{\phi}^{m+1} + \frac14 \tilde{\phi}^{m-1} )}_2^2
	\ge \frac38 \| \nabla_h \Delta_h \tilde{\phi}^{m+1} \|_2^2 
	- \frac18 \| \nabla_h \Delta_h \tilde{\phi}^{m-1} \|_2^2  . 
	\label{CHHS-2nd-error-9-2}
	\end{eqnarray} 
	Then we get 
	\begin{eqnarray}
	&&
	\| \nabla_h \tilde{\phi}^{m+1}  \|_2^2 
	- \| \nabla_h \tilde{\phi}^m \|_2^2 
	+ \frac14 ( \| \nabla_h ( \tilde{\phi}^{m+1} - \tilde{\phi}^m )  \|_2^2 
	- \| \nabla_h ( \tilde{\phi}^m - \tilde{\phi}^{m-1} )  \|_2^2 )  
	+ \frac{11}{64} \varepsilon^2 s \| \nabla_h \Delta_h \tilde{\phi}^{m+1} \|_2^2  \nonumber 
	\\
	&\le& 
	\frac{\varepsilon^2 s}{16}  \| \nabla_h \Delta_h \tilde{\phi}^m \|_2^2  
	+ \frac{5 \varepsilon^2 s}{64}  \| \nabla_h \Delta_h \tilde{\phi}^{m-1} \|_2^2    
	+ s \| \tau^{m+1/2} \|_2^2 + \frac{64(1+\gamma^2)(C_{14}C_{18})^2}sh^4 \nonumber 
	\\
	&& 
	+  s\left( C^{\alpha, \varepsilon,\gamma}_4  ( C_{25}^{8/3} + C_{25}^2 )  
	+ C_{27} \varepsilon^{-2}  \right) 
	(  \|  \nabla_h \tilde{\phi}^{m+1} \|_2^2 
	+ \|  \nabla_h \tilde{\phi}^m \|_2^2  
	+ \|  \nabla_h \tilde{\phi}^{m-1} \|_2^2 ) . 
	\label{CHHS-2nd-error-9-3}
	\end{eqnarray}

	For the sake of convenience, we now make the coefficient on the right side of \eqref{CHHS-2nd-error-9-3} explicit to each time step. Define
	\begin{equation}
	I_2 = C^{\alpha, \varepsilon,\gamma}_4  ( C_{25}^{8/3} + C_{25}^2)  
	+ C_{27} \varepsilon^{-2}. 
	\end{equation}
	Apply Young's inequality on $C_{25}^2$, we get
	\begin{eqnarray}
	 C_{25}^2 \le \frac34 C_{25}^{8/3} + \frac14.
	\end{eqnarray}
	Then $I_2$ can be bounded as
	\begin{eqnarray}
	I_2  \le C^{\alpha, \varepsilon,\gamma}_4 ( \frac74 C_{25}^{8/3} + \frac54 )  
	+ \varepsilon^{-2}C_8 \le 2 C^{\alpha, \varepsilon,\gamma}_4  (C_{25}^{8/3} + 1) + C_{27}\varepsilon^{-2}.
	\end{eqnarray}
	Recall the definition of $C_{25}$ in \eqref{Def-C25}, the value of which can be controlled as
	\begin{eqnarray}
	C_7 \le C_{24}\max(C_{16},C_{17})\cdot((M_0^m)^2 + (M_0^{m-1})^2)(M_0^{m+1} + M_0^m + 1).
	\end{eqnarray}
	With the application of the following inequality
	\begin{eqnarray}
     (a + b)^p \le 2^{p-1}(a^p + b^p), \quad {\rm for} \,\, \forall p \ge 1,
	\end{eqnarray}
	the value of $C_{25}^{8/3}$ can be bounded as
	\begin{eqnarray}
	C_{25}^{8/3} \le 8^{5/3}C_{24}^{8/3}(\max(C_{16},C_{17}))^{8/3}((M_0^m)^{16/3} + (M_0^{m-1})^{16/3})((M_0^{m+1})^{8/3} + (M_0^m)^{8/3} + 1).
	\end{eqnarray}
	As a result, the stability inequality \eqref{CHHS-2nd-error-9-3} can be rewritten as \eqref{eq:stable result}, with the following constants: 
	\begin{eqnarray}
	   C_{28} &=& \frac{64(1+\gamma^2)(C_{14}  \max(C_{16},C_{17}))^2}{\varepsilon^2}, 
	   \label{CHHS-2nd-error-9-4-1}	   
	   \\
       C_{29} &=& 2C^{\alpha,\varepsilon,\gamma}_4\cdot\max\left(8^{5/3}C_{24}(\max(C_{16},C_{17}))^{8/3}, 1\right),  \label{CHHS-2nd-error-9-4-2}  
       \\
       C_{30} &=& C(2C_{12}^4 +1)\varepsilon^{-2}.   \label{CHHS-2nd-error-9-4-3}       
       \end{eqnarray} 
This finishes the proof of Theorem~\ref{thm:stability}. 
\end{proof} 
	
    \subsection{The result of an \emph{a-priori} error assumption}
    \label{subsec:a-priori}	
    As discussed in \cite{chen2015efficient}, in which a first order numerical scheme was analyzed, the discrete Gronwall inequality could not be directly applied to derive an error estimate from the stability inequality as in the form of~\eqref{eq:stable result}, since $D_1^{m+1}$ and $D_2^{m+1}$ do not have a uniform bound. Instead, we have to use an induction argument to establish the convergence analysis.  Specifically, we assume, as an induction hypothesis, that the desired error estimate holds at an arbitrary time step $m$ ($0 \le m \le M-1$).  We then use this \emph{a priori} assumption to prove that $s(C_{29}D_1^{m+1} + C_{30}D_2^{m+1}) <1$, provided $s$ is small enough.  Then we conclude the induction argument by proving that the error estimate holds at the updated time step $m+1$.
    
    First, we need the following technical result, which is a direct result of Young's inequality. The proof is skipped for brevity.
    
    \begin{lem}
    	For any $a>0$, $\delta>0$ and $0 < q < 8$, we have
    	\begin{equation}
    		a \cdot \delta^q  \le  b \delta^8 + r(a,b,q) ,  \quad \forall \ b > 0 , \qquad  \mbox{where} \qquad  r(a,b,q) := \frac{a^{\frac{8}{8-q}}}{\frac{8}{8-q}\left(b\cdot\frac{8}{q}\right)^{\frac{q}{8-q}}} . 
    		\label{Young-1}
    	\end{equation}
    \end{lem}
    We also need the following estimate of the $\| \cdot \|_{\infty}$ norm of $\phi^m$.
    \begin{lem}
    \label{lem:Linfty est}
    For any $s, h>0$ and any $1\le \ell\le M$, there exists a constant $C_{31}>0$  such that
    \begin{eqnarray}
     s\sum_{m=1}^{\ell}  \nrm{ \phi^m }_{\infty}^8 &\leq& C_{31} (t_\ell+1) \le C_{31}\left( T + 1 \right),
     \label{L8 Linf stab-1}
    \end{eqnarray}
    where $t_\ell := s\cdot \ell$, and $T := s\cdot M$.
    \end{lem}
    
\begin{proof} 
Inequality \eqref{L8 Linf stab-1} is a direct consequence of the discrete Gagliardo-Nirenberg type inequality \eqref{maximum bound}, combined with the leading order $H_h^1$ bound \eqref{H^1-stability} and the $\ell^2 (0,T; H_h^3)$ estimate \eqref{stab-L2-H3-0} (in Theorem~\ref{thm:stability L2-H3}) for the numerical solution. 
\end{proof} 
    
    \begin{thm}
    	Suppose that $h$ and $s$ are sufficiently small and the following error estimate is valid up to the time step $t_m:=m\cdot s$, for $2\le m \le M-1$:
    	\begin{eqnarray}
    		\nrm{ \nabla_h \tilde{\phi}^m }_2^2 +  \varepsilon^2 s \sum_{j=1}^m \nrm{ \nabla_h \Delta_h  \tilde{\phi}^j }_2^2  \le  C_{32} \exp\left(C_{33} (t_m +1) \right)  \left( s^4 + h^4  \right) ,
    		\label{a priori-1}
    	\end{eqnarray}
    	where $C_{32}, C_{33}>0$ may depend upon the final time $T$ but are independent of $s$ and $h$. Then
    	\begin{eqnarray}
    		s(C_{29}D_1^{m+1} + C_{30}D_2^{m+1}) \le \frac12 .
    		\label{convergence-bound-14}
    	\end{eqnarray}
    \end{thm}
    
    \begin{proof}
    	As an application of \eqref{Young-1}, the non-leading terms appearing on the right hand side of \eqref{Def-D1} for the expansion of $D_1^{m+1}$ can be
    	bounded as follows: 
    	\begin{eqnarray}
    		(M_0^m)^{16/3} &\le&  \frac{1}{52 C_{29} C_{31}(T+1) } (M_0^m)^8  +  C_{32} ,
    		\label{convergence-bound-3-1}
    		\\
    		(M_0^{m-1})^{16/3} &\le&  \frac{1}{52 C_{29} C_{31}(T+1) } (M_0^{m-1})^8  +  C_{33}.
    		\label{convergence-bound-3-2}
    	\end{eqnarray}
    	Then we get, for any $0\le m\le M-1$,
    	\begin{eqnarray}
    		\label{convergence-bound-3-6}
    		s  C_{29} \left( (M_0^m)^{16/3} + (M_0^{m-1})^{16/3}  + 1 \right) &\le& \frac{s}{52 C_{31}(T+1) } (M_0^m)^8 + \frac{s}{52 C_{31}(T+1) } (M_0^{m-1})^8 + s C_{34} 
    		\nonumber
    		\\
    		&\le& \frac{1}{26} + s C_{34},
    	\end{eqnarray}
    	using the $L^8_s (0,T)$ bound for $M_0^m := \nrm{\phi^m}_\infty$ in (\ref{L8 Linf stab-1}), where $C_{34}>0$ is a constant independent of $h$ and $s$. Using the same skill, the non-leading order of $D_2^{m+1}$ can be bounded as follows
    	\begin{eqnarray}
    	\label{convergence-bound-3-6-2}
    	  s  C_{30} \left((M_0^{m+1})^4 + 2(M_0^m)^4 + (M_0^{m-1})^4  + 1 \right) \le \frac{3}{52} + s C_{35},
    	\end{eqnarray}
    	where $C_{18}>0$ is a constant that is independent of $h$ and $s$.
    	
    	Now, the leading terms appearing on the right hand side of \eqref{Def-D1}--\eqref{Def-D2} cannot be bounded in this way. We divide them into two groups $\mathcal{G}_1$ and $\mathcal{G}_2$ as follows:
    	\begin{eqnarray}
    	\mathcal{G}_1 &:& (M_0^m)^8,  (M_0^m)^{8/3}(M_0^{m-1})^{16/3}, (M_0^m)^4(M_0^{m-1})^4,
    	\\
        \mathcal{G}_2 &:& (M_0^{m+1})^{8/3}(M_0^m)^{16/3}, (M_0^{m+1})^{8/3}(M_0^{m-1})^{16/3}, (M_0^{m+1})^4(M_0^m)^4, (M_0^{m+1})^4(M_0^{m-1})^4.
    	\end{eqnarray}    	    	
    	We must, therefore, rely upon \eqref{a priori-1}.  This bound implies
    	\begin{equation}
    		\label{convergence-bound-4}
    		\begin{split}
    			\nrm{ \nabla_h \tilde{\phi}^m }_2^2 &\le  C_{32} \exp\left(C_{33} (T +1)\right)  \left( s^4 + h^4  \right) ,
    			\\
    			\nrm{ \nabla_h \Delta_h  \tilde{\phi}^m }_2^2 &\le \varepsilon^{-2} C_{32} \exp\left(C_{33}
    			( T +1 ) \right) \left( s^4 + h^4\right)s^{-1}   .
    		\end{split}
    	\end{equation}
    	Using (\ref{convergence inf est-1}) and setting $C_{36} := C_{32} \exp\left(C_{33} (T +1) \right)$, we have
    	\begin{eqnarray}
    		\nrm{ \tilde{\phi}^m }_{\infty}^2 &\le& 4C_{14}^2 \left( \nrm{ \nabla_h \tilde{\phi}^m }_2^{\frac32} \nrm{ \nabla_h \Delta_h \tilde{\phi}^m }_2^{\frac12} + \nrm{ \nabla_h \tilde{\phi}^m }_2^2 + h^4 \right)
    		\nonumber
    		\\
    		&\le&  4C_{14}^2 \left\{ C_{36} \left( s^4 + h^4  \right)\left(\varepsilon^{-1/2}s^{-1/4}+1\right) + h^4 \right\}
    		\nonumber
    		\\
    		&=& 4C_{14}^2 \left\{ C_{36}\varepsilon^{-1/2} s^{15/4} +C_{36}\varepsilon^{-1/2} h^4s^{-1/4} +C_{36}s^4 +\left(1+C_{36}\right)h^4 \right\}  .
    		\label{convergence-bound-5}
    	\end{eqnarray}
    	Under the time and space step size constraint
    	\begin{eqnarray}
    		C_{36}\varepsilon^{-1/2} s^{15/4} +C_{36}s^4 +\left(1+C_{36}\right)h^4  \le \frac{1}{4C_{14}^2} ,
    		\label{constraint-s h-1}
    	\end{eqnarray}
    	the following bound is available:
    	\begin{eqnarray}
    		\nrm{ \tilde{\phi}^m }_{\infty}^2 \le 1 + 4C_{14}^2C_{36}\varepsilon^{-1/2} \frac{ h^4 }{s^{1/4}}  .
    		\label{convergence-bound-6}
    	\end{eqnarray}
    	Consequently, we see that
    	\begin{eqnarray}
    		(M_0^m)^2 &:=& \nrm{ \phi^m }_{\infty}^2 \le 2\nrm{\Phi^m}_{\infty}^2 + 2\nrm{ \tilde{\phi}^m }_{\infty}^2 \le C_{37}\left(1 + \frac{h^4}{s^{1/4}}\right) ,
    		\label{convergence-bound-7}
    		\\
    		(M_0^{m-1})^2 &:=& \nrm{ \phi^{m-1} }_{\infty}^2 \le 2\nrm{\Phi^{m-1}}_{\infty}^2 + 2\nrm{ \tilde{\phi}^{m-1} }_{\infty}^2 \le C_{37}\left(1 + \frac{h^4}{s^{1/4}}\right) ,
    		\label{convergence-bound-7-2}
    	\end{eqnarray}
    	where $C_{37}>0$ is independent of $s$ and $h$, but it does depend upon the final time $T$ (at least exponentially) and the interface parameter $\varepsilon$ $(\mathcal{O}(\varepsilon^{-1/2}))$. This shows that 
    	\begin{eqnarray}
    	 (M_0^m)^8 \le C_{37}^4\left(1 + \frac{h^4}{s^{1/4}}\right)^4 \le 8 C_{37}^4(1 + \frac{h^{16}}{s}),
    	\end{eqnarray}
    	the bound of which is also valid for other terms in group $\mathcal{G}_1$. Thus, under the time and space step size constraint
    	\begin{eqnarray}
    	 8 C_{37}^4(s + h^{16}) \le \frac{1}{52C_{29}},
    	\end{eqnarray}
    	the following bounds are available:
    	\begin{eqnarray}
    	  s  C_{29}\left((M_0^m)^8 + (M_0^m)^{8/3}(M_0^{m-1})^{16/3}\right) &\le& \frac{1}{26},
    	  \label{convergence-bound-26-1}
    	  \\
    	  s  C_{30} \left((M_0^m)^8 + (M_0^m)^4(M_0^{m-1})^4\right) &\le& \frac{1}{26}.
    	  \label{convergence-bound-26-2}
    	\end{eqnarray}

    	Now we estimate terms in group $\mathcal{G}_2$. We take $(M_0^{m+1})^4(M_0^m)^4$ as example. Reusing the estimate \eqref{convergence-bound-7}--\eqref{convergence-bound-7-2} leads to
    	\begin{eqnarray}
    		(M_0^{m+1})^4(M_0^m)^4 \le  C_{37}^2 \left(1+ \frac{h^4}{s^{1/4}}  \right)^2 (M_0^{m+1})^4 
    		\le  2 C_{37}^2(M_0^{m+1})^4 + 2 C_{37}^2\frac{h^8}{s^{1/2}}(M_0^{m+1})^4 .
    		\label{convergence-bound-8}
    	\end{eqnarray}
    	The first term on the right hand side can be handled in the same way
    	as (\ref{convergence-bound-3-1}):
    	\begin{eqnarray}
    		2 C_{37}^2 (M_0^{m+1})^4 \le  \frac{1}{104 C_{30}C_{31}(T+1) } (M_0^{m+1})^8  +  C_{38}.
    	\end{eqnarray}
    	Hence
    	\begin{equation}
    		s C_{30} \left(2 C_{37}^2 (M_0^{m+1})^4\right)  \le  \frac{1}{104}  +  sC_{38},
    		\label{convergence-bound-9}
    	\end{equation}
    	where $C_{38}>0$ is independent of $s$ and $h$. The second term on the right hand side of (\ref{convergence-bound-8})
    	can be analyzed as follows: using Cauchy's inequality and (\ref{L8 Linf stab-1}), we have
    	\begin{eqnarray}
    		s C_{30} \left( 2 C_{37}^2 \frac{ h^8 }{s^{1/2}} (M_0^{m+1})^4  \right) &\le&   C_{30}  C_{37}^2 h^8 \left(  s  (M_0^{m+1})^8 + 1 \right),
    		\nonumber
    		\\
    		&\le&  C_{30} C_{37}^2 C_4(T+1) h^8 + C_{30}  C_{37}^2 h^8  .
    		\label{convergence-bound-10}
    	\end{eqnarray}
    	Under an additional constraint for the grid size
    	\begin{eqnarray}
    		h^8 \le  \min  \left(   \frac{1}{208 C_{30} C_{37}^2 C_{31}(T+1) }    ,     
    		\frac{1}{208 C_{30} C_{37}^2}    \right) ,
    		\label{constraint-h-2}
    	\end{eqnarray}
    	we arrive at
    	\begin{eqnarray}
    		s C_{30} \left( 2 C_{37}^2 \frac{ h^8 }{s^{1/2}} (M_0^{m+1})^4  \right) \le \frac{1}{104} .
    		\label{convergence-bound-11}
    	\end{eqnarray}
    	
    	A combination of (\ref{convergence-bound-8}), (\ref{convergence-bound-9}) and (\ref{convergence-bound-11}) yields
    	\begin{eqnarray}
    		s C_{30} (M_0^{m+1})^4(M_0^m)^4   \le \frac{1}{52} + s C_{38} .
    	\end{eqnarray}
    	A similar analysis can be applied to all the other terms in group $\mathcal{G}_2$:  under a similar constraints as given by (\ref{constraint-h-2}), we have
    	\begin{eqnarray}
    	s C_{29} \left((M_0^{m+1})^{8/3}(M_0^m)^{16/3} + (M_0^{m+1})^{8/3}(M_0^{m-1})^{16/3}\right) &\le& \frac{1}{26} + s C_{39} ,
    	\label{convergence-bound-12-1}
    	\\
    	s C_{30} \left((M_0^{m+1})^4(M_0^m)^4 +  (M_0^{m+1})^4(M_0^{m-1})^4\right) &\le& \frac{1}{26} + s C_{39} ,
    	\label{convergence-bound-12-2}
    	\end{eqnarray}
    	where $C_{39}>0$ is independent of $s$ and $h$. The details of the proof are skipped for the sake of brevity.
    	
    	Therefore, a combination of \eqref{convergence-bound-3-6}--\eqref{convergence-bound-3-6-2}, \eqref{convergence-bound-26-1}--\eqref{convergence-bound-26-2}, and \eqref{convergence-bound-12-1}--\eqref{convergence-bound-12-2} leads to
    	\begin{eqnarray}
    		s(C_{29}D_1^{m+1} + C_{30}D_2^{m+1}) \le \frac14 + s (C_{34}+C_{35} +C_{39}) ,
    		\label{convergence-bound-13}
    	\end{eqnarray}
    	and under the additional constraint for the time step
    	\begin{eqnarray}
    		s \le  \frac{1}{\ 4 (C_{34}+C_{35}+C_{39}) \ } ,
    		\label{constraint-s-2}
    	\end{eqnarray}
    	we get the desired result, estimate \eqref{convergence-bound-14}.
    \end{proof}
    \subsection{The main result: an error estimate}
    \label{subsec:error-estimate}
    
    The following theorem is the main theoretical result of this article.  The basic idea is to extend the \emph{a-priori} error estimate \eqref{a priori-1} by an induction argument.
    
    \begin{thm}
    	\label{thm:convergence}
    	Given initial data $\phi^{0},\phi^1 \in C^6 (\overline{\Omega})$, with homogeneous 
    	Neumann boundary conditions, suppose the unique solution for the 
    	CHHS equation (\ref{equation-CHHS-1}) -- (\ref{equation-CHHS-3}) is of regularity class 
    	$\mathcal{R}$. Then, provided $s$ and $h$ are sufficiently small, 
    	for all positive integers $\ell$, such that $s\cdot \ell \le T$, we have
    	\begin{equation}
    		\nrm{ \nabla_h \tilde{\phi}^\ell}_2^2 +  \varepsilon^2 s   \sum_{m=1}^{\ell} \nrm{ \nabla_h \Delta_h \tilde{\phi}^m}_2^2
    		\le C \left( s^4 + h^4 \right),   \label{convergence-3}
    	\end{equation}
    	where $C>0$ is independent of $s$ and $h$.
    \end{thm}
    
    \begin{proof}
    	Suppose that  $m+1\le M$. By summing \eqref{eq:stable result} we obtain
    	\begin{eqnarray}
    		&& \| \nabla_h\tilde{\phi}^{m+1} \|_2^2 + \frac14 \| \nabla_h ( \tilde{\phi}^{m+1} - \tilde{\phi}^m )  \|_2^2 +\frac{1}{32}\varepsilon^2 s\sum_{j=1}^{m+1} \| \nabla_h \Delta_h \tilde{\phi}^j \|_2^2
    		\nonumber
    		\\ 
    		&\le& \| \nabla_h\tilde{\phi}^0 \|_2^2 + \frac14 \| \nabla_h ( \tilde{\phi}^0 - \tilde{\phi}^{-1} )  \|_2^2 + s\sum_{j=1}^{m+1} (C_{29}D_1^j + C_{30}D_2^j) \| \nabla_h\tilde{\phi}^j \|_2^2
    		\nonumber
    		\\
    		&&  + s\sum_{j=0}^m (C_{29}D_1^{j+1} + C_{30}D_2^{j+1}) ( \| \nabla_h\tilde{\phi}^j \|_2^2 + \| \nabla_h\tilde{\phi}^{j-1} \|_2^2 )
    		\nonumber
    		\\
    		&&
            +  s \sum_{j=1}^{m+1} \| \tau^{j+1/2} \|_2^2 +C_{28} s\sum_{j=1}^{m+1}D^j_3h^4  . 
    	\end{eqnarray}
    	We proceed by induction.  Namely, suppose that \eqref{a priori-1} holds.  Then, if $h$ and $s$ are sufficiently small -- as required in the proof of the last theorem -- considering (\ref{convergence-bound-14}) and using $\tilde{\phi}^{-1}\equiv\tilde{\phi}^0\equiv 0$, we have
    	\begin{eqnarray}
    	&&\frac12 \| \nabla_h\tilde{\phi}^{m+1} \|_2^2 + \frac{1}{32}\varepsilon^2 s\sum_{j=1}^{m+1} \|\nabla_h \Delta_h \tilde{\phi}^j \|_2^2
    	\nonumber
    	\\ 
    	&\le& s\sum_{j=0}^m (C_{29}D_1^{j+1} + C_{30}D_2^{j+1}) (2 \| \nabla_h\tilde{\phi}^j \|_2^2 + \| \nabla_h\tilde{\phi}^{j-1} \|^2 )
    	\nonumber
    	\\
    	&& 
        +  s \sum_{j=1}^{m+1} \| \tau^{j+1/2} \|_2^2 +C_{28} s\sum_{j=1}^{m+1}D^j_3h^4  . 
    	\end{eqnarray}
    	Hence
    	\begin{eqnarray}
    		\| \nabla_h\tilde{\phi}^{m+1} \|_2^2 + \varepsilon^2 s\sum_{j=1}^{m+1} \| \nabla_h \Delta_h \tilde{\phi}^j \|_2^2 &\le& s\sum_{j=0}^m (32C_{29}D_1^{j+1} + 32C_{30}D_2^{j+1}) (2 \| \nabla_h\tilde{\phi}^j \|_2^2 + \| \nabla_h\tilde{\phi}^{j-1} \|_2^2 )
    		\nonumber
    		\\
    		&&+ C_{40} (s^4+h^4 ),
    	\end{eqnarray}
    	where $C_{40}>0$ is a constant that is independent of $s$ and $h$. Using the discrete Gronwall inequality gives
    	\begin{eqnarray}
    		\| \nabla_h\tilde{\phi}^{m+1} \|_2^2 +  \varepsilon^2 s\sum_{j=1}^{m+1} \| \nabla_h \Delta_h \tilde{\phi}^j \|_2^2 &\le& C_{40} (s^4+h^4 )\exp\left( s\sum_{j=1}^{m}(96C_{29}D_1^{j+1} + 96C_{30}D_2^{j+1}) \right)
    		\nonumber
    		\\
    		&\le& C_{40} (s^4+h^4 )\exp\left(C_{41}(t_{m+1}+1) \right)  ,
    		\label{convergence-2}
    	\end{eqnarray}
    	where $C_{41}>0$ is a constant that is independent of $s$ and $h$.
    	Consequently, the \emph{a priori} assumption (\ref{a priori-1}) can be
    	justified at time step $t_{m+1}$ by taking $C_{32} =  C_{40}$,  $C_{33} = C_{41}$.  This completes the induction argument, and the proof of Theorem~\ref{thm:convergence} is finished.
    \end{proof}    
\section{Numerical Experiments}\label{sec:numerical}
In this section, we perform some numerical tests in two-dimensional space to verify the accuracy and efficiency of the proposed numerical 
scheme (\ref{2nd scheme-CHHS-s1})-(\ref{2nd scheme-CHHS-s3}).  The coupled systems are solved by the Full Approximation Scheme (FAS) under the
nonlinear multigrid framework in \cite{feng2016bsam,wise2010unconditionally}. Here we omit the details for brevity; more details in \cite{collins2013efficient,wise2010unconditionally} are referred to the readers. 
In the following tests, all the numerical experiments were performed with Fortran90 
on Thinkpad W541 running with Intel Core i7-4800MQ at 2.80Ghz with  7.4GB memory under the Ubuntu 14.04. The general parameters of FAS are finest grid $2\times 2$, pre- and post-smooth steps $\nu_1=\nu_2=2$ and stopping tolerance $tol=10^{-10}$.

\subsection{Convergence rate, energy dissipation and mass conservation test}
To estimate the convergence rate, we perform the Cauchy-type convergence as in \cite{baskaran2016energy,collins2013efficient,feng2016preconditioned,hu2009stable, shen2012second, wang2010unconditionally} on a square $\Omega=[0,L_x]\times[0,L_y]$ with initial condition
\begin{eqnarray}\label{eqn:init1}
\phi(x,y,0)&=&\frac{\big[1-\cos\big(\frac{4\pi x}{L_x}\big)\big]\cdot\big[1-\cos\big(\frac{2\pi y}{L_y}\big)\big]}{2}-1.
\end{eqnarray}

The homogeneous Neumann boundary conditions are imposed for $\phi$, $\mu$ and $p$. In this test, the Cauchy difference is defined as $\delta_\phi = \phi_{h_f} - \mathcal{I}_c^f \phi_{h_c}$,  where $h_c = 2h_f$ and $\mathcal{I}_c^f$ is a bilinear interpolation operator that maps the coarse grid  approximation $u _{h_c}$ onto the fine grid (we applied $nearest$ matlab interpolation function). We take a liner refinement path, i.e. $s = Ch$. At the final time $T=0.8$, we expect the global error to be $O (s^2) + O (h^2) = O (h^2)$ under the $\ell^2$ norm, as $h,s \to 0$. The other parameters are given by $L_x = L_y = 3.2$, $s = 0.05h$, $\varepsilon=0.2$ and $\gamma=2$. The norms of Cauchy difference, the convergence rates, the average number of V-cycle and average CPU time for one time step can be found in Table \ref{tab:cov}, which confirms our second order convergence rate expectation and indicates the efficiency of the proposed numerical scheme. The evolutions of discrete energy and mass for the simulation, associated with Table \ref{tab:cov} for the $h = \frac{3.2}{512}$, are presented in Figure \ref{fig:energy-mass}. The energy dissipation property is clearly demonstrated in the evolutions of discrete energy in the figure. And also, the evolution of discrete mass indicates the mass conservative property, with $\int_\Omega \phi (x,y,0) d {\bf x} = -5.12$.

\begin{table}[!htb]
\begin{center}
\caption{Errors, convergence rates, average iteration numbers and average CPU time (in seconds)
for each time step.} \label{tab:cov}
\begin{tabular}{cccccc}
\hline $h_c$&$h_{f}$&$\norm{\delta_\phi}_{2}$ & Rate& \#V's & $T_{cpu}(h_f)$\\
\hline $\frac{3.2}{16}$&$\frac{3.2}{32}$& $7.6501\times 10^{-3}$&- &5 &0.0012
\\$\frac{3.2}{32}$&$\frac{3.2}{64}$& $1.8565\times 10^{-3}$ &2.04 &5 &0.0046
\\ $\frac{3.2}{64}$ &$\frac{3.2}{128}$& $4.6141\times 10^{-4}$&2.01 &4 &0.0160
\\ $\frac{3.2}{128}$ & $\frac{3.2}{256}$ &$1.1520\times 10^{-4}$&2.00 &4 & 0.0744
\\ $\frac{3.2}{256}$ & $\frac{3.2}{512}$&$2.8792\times 10^{-5}$&2.00&5 &0.3818\\
\hline
\end{tabular}
\end{center}
\end{table}

\begin{figure}[!htp]
\centering
\includegraphics[width=0.45\textwidth]{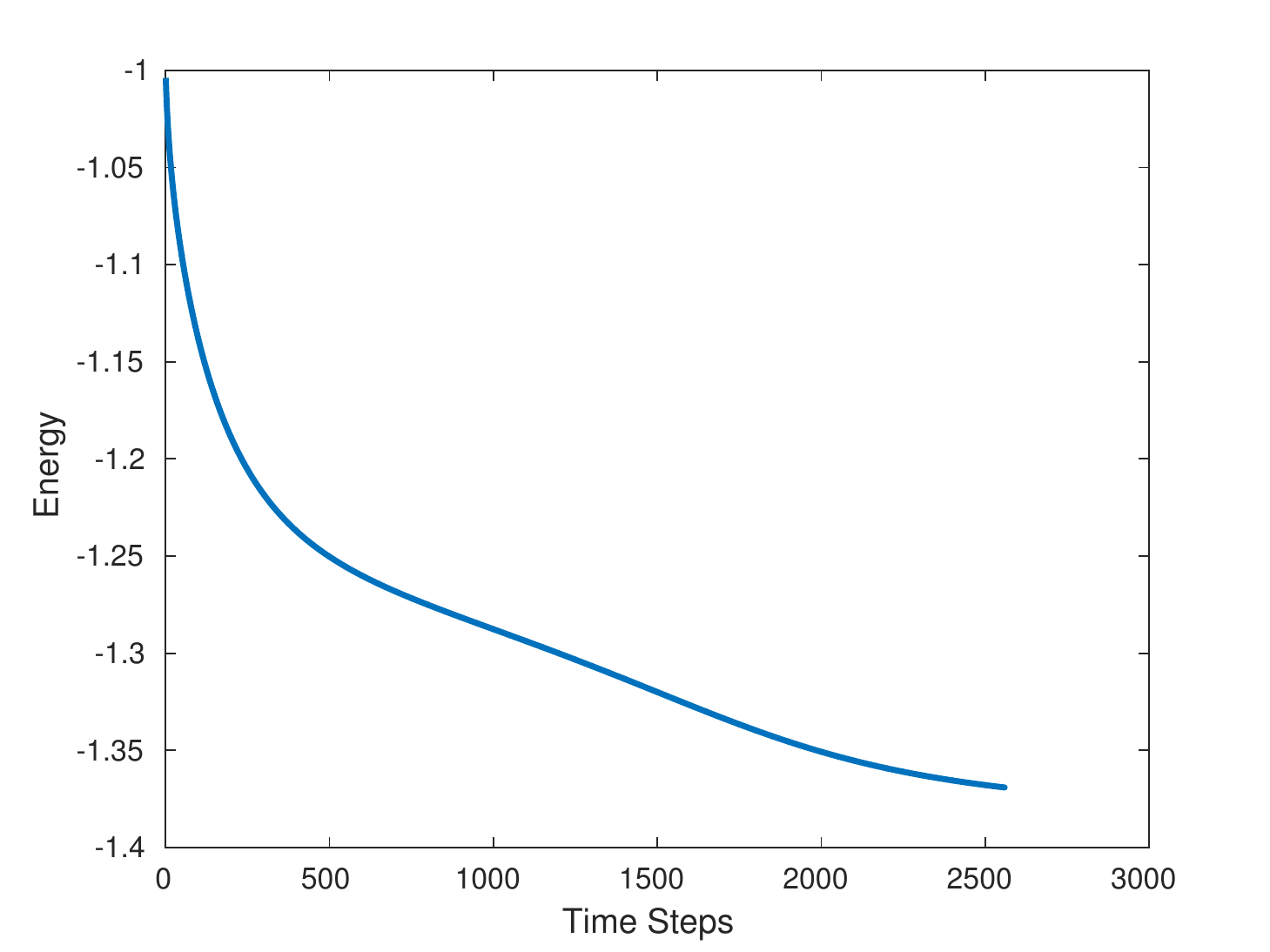}
\includegraphics[width=0.45\textwidth]{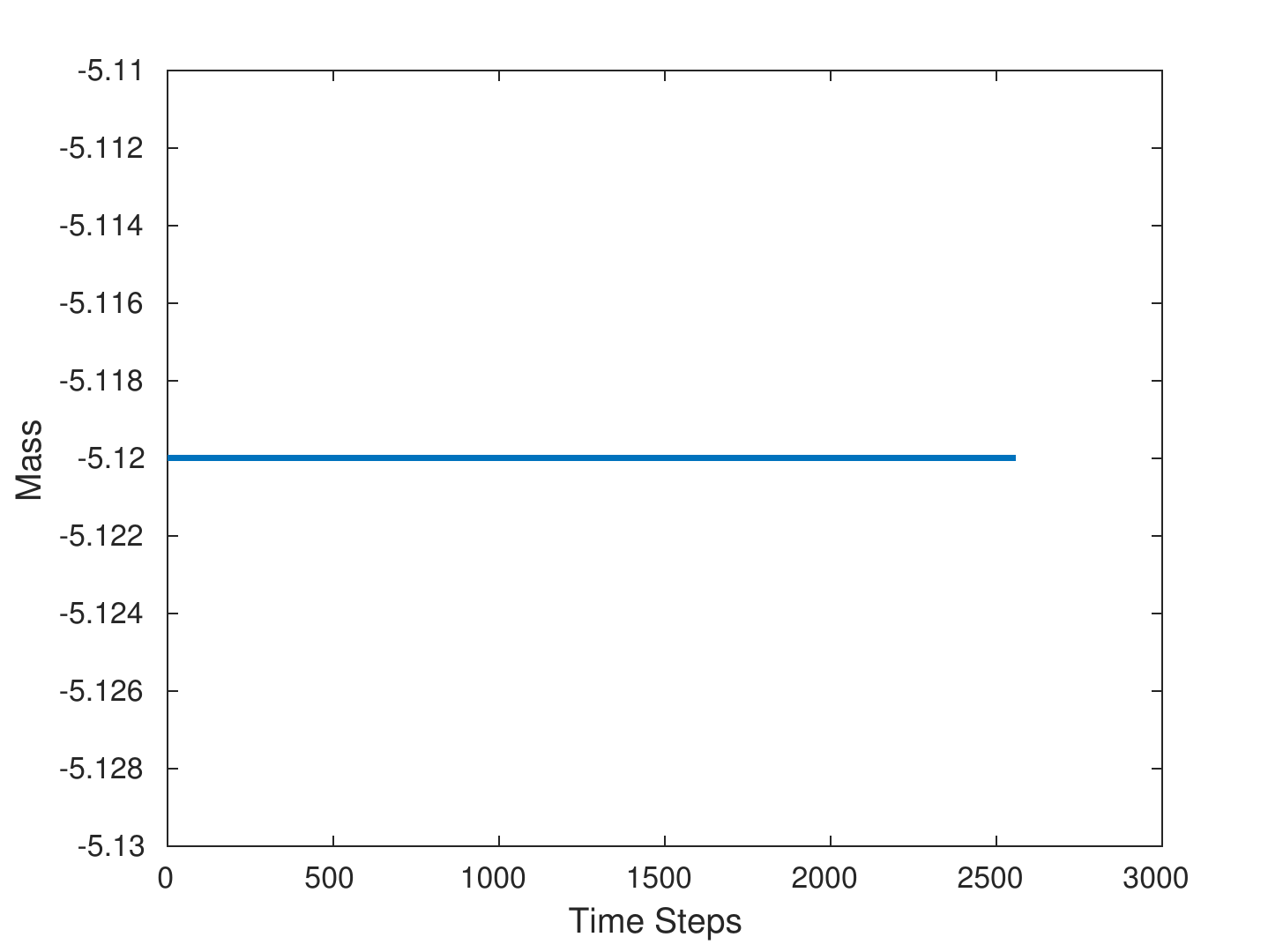}\hfill
\caption{The evolutions of discrete energy and mass for the simulation depicted in Table \ref{tab:cov} for the $h=3.2/512$ case.}
\label{fig:energy-mass}
\end{figure}

\subsection{Spinodal decomposition}
In this test, we simulate the spinodal decomposition of a binary fluid in a Hele-Shaw cell and show the effect of $\gamma$ on the phase decomposition. The simulation parameters are similar to those in \cite{wise2010unconditionally}, with the parameters given by $L_x=L_y=6.4$, $\varepsilon=0.03$, $h=6.4/512$, and $s=0.01$. The initial data for this simulation is taken 
as a random field values $\phi^0_{i,j}=\bar{\phi}+0.05\cdot(2r_{i,j}-1)$ with an average composition $\bar{\phi}=-0.05$ and $r_{i,j}\in[0,1]$. The simulation results are presented in Figures  \ref{fig:spinodal} and \ref{fig:energy-gamma}. From Figure. \ref{fig:spinodal}, we observe that the particles indeed have a smaller shape factor for $\gamma =4$ than for $\gamma=0$ at same time, which coincides with the real physical states. Since larger $\gamma$ would improve the fluid flow and enhance the energy dissipation. The energy evolution plot in Figure \ref{fig:energy-gamma} implies that the energy decay are almost the same in the early stages of decomposition. Meanwhile, it is not precisely clear from the energy inequality that the larger $\gamma$ will result in a larger energy dissipation rate \cite{han2014decoupled,lee2002modeling2,wise2009energy}.

\begin{figure}[!htp]
    \centering
        \begin{subfigure}[b]{0.28\textwidth}
            \includegraphics[width=\textwidth,height=\textwidth]{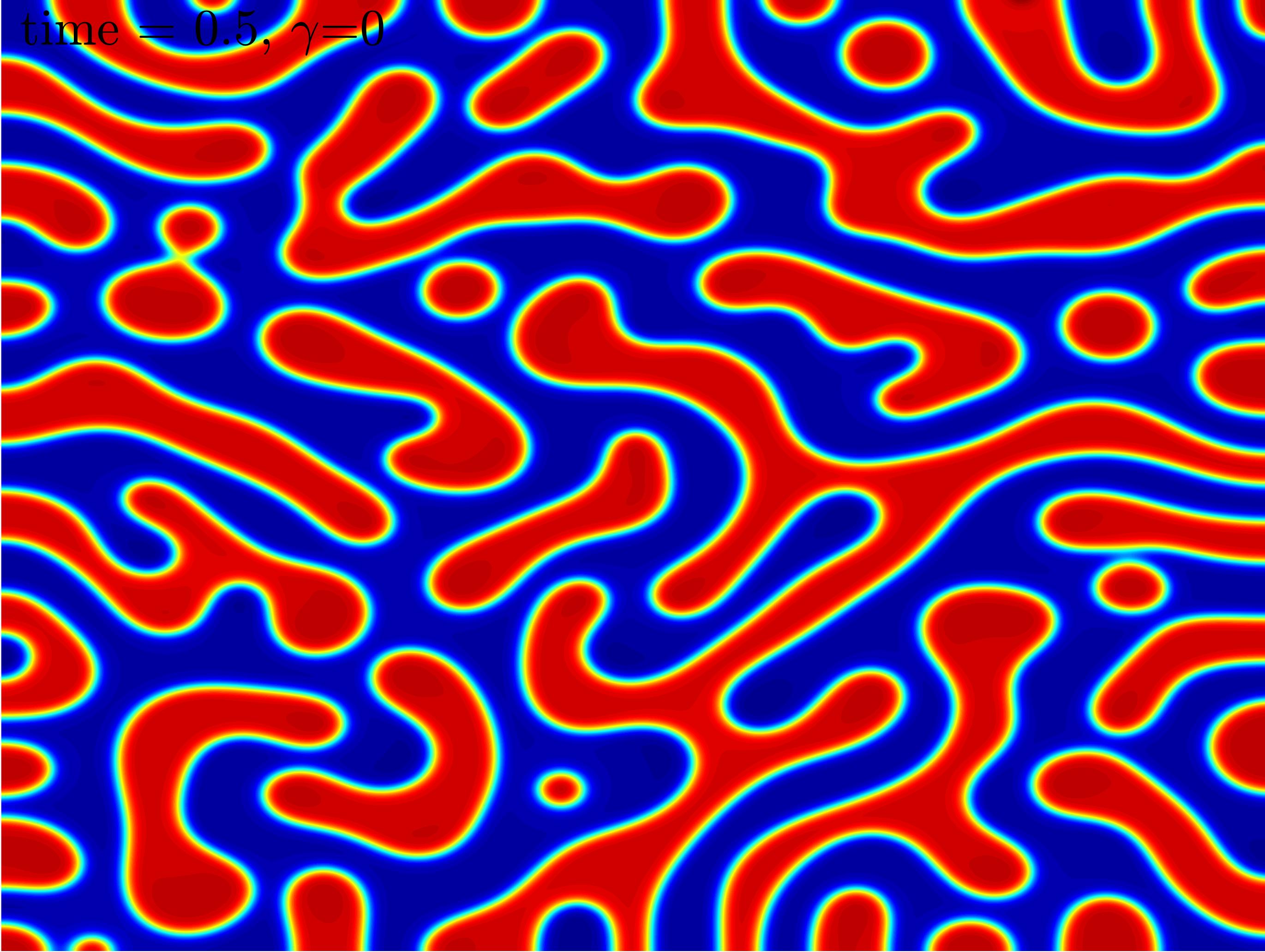}
            \caption*{$t =0.5, \gamma=0$}
        \end{subfigure}
        \begin{subfigure}[b]{0.28\textwidth}
            \includegraphics[width=\textwidth,height=\textwidth]{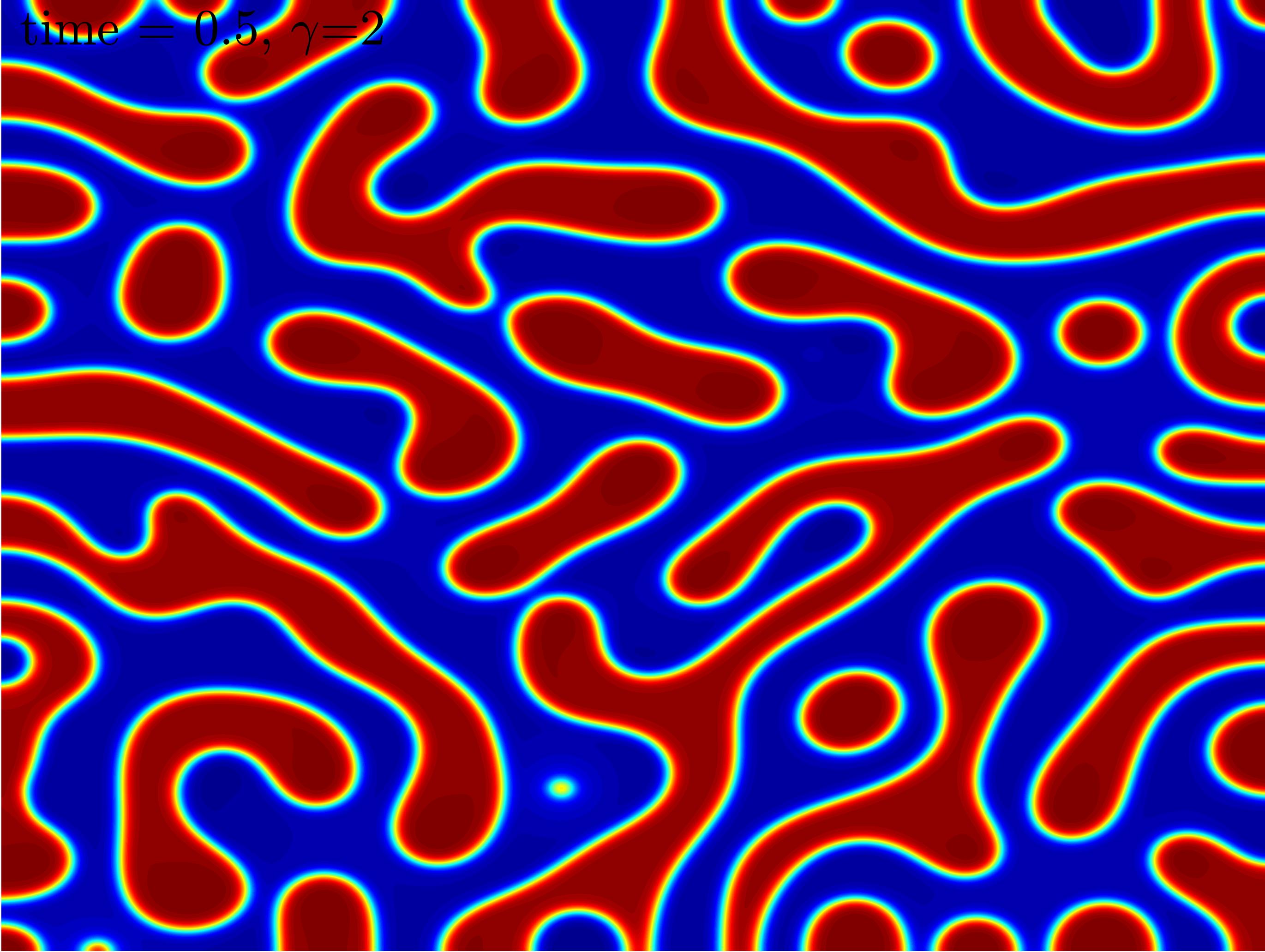}
            \caption*{$t =0.5, \gamma=2$}
        \end{subfigure}
        \begin{subfigure}[b]{0.28\textwidth}
            \includegraphics[width=\textwidth,height=\textwidth]{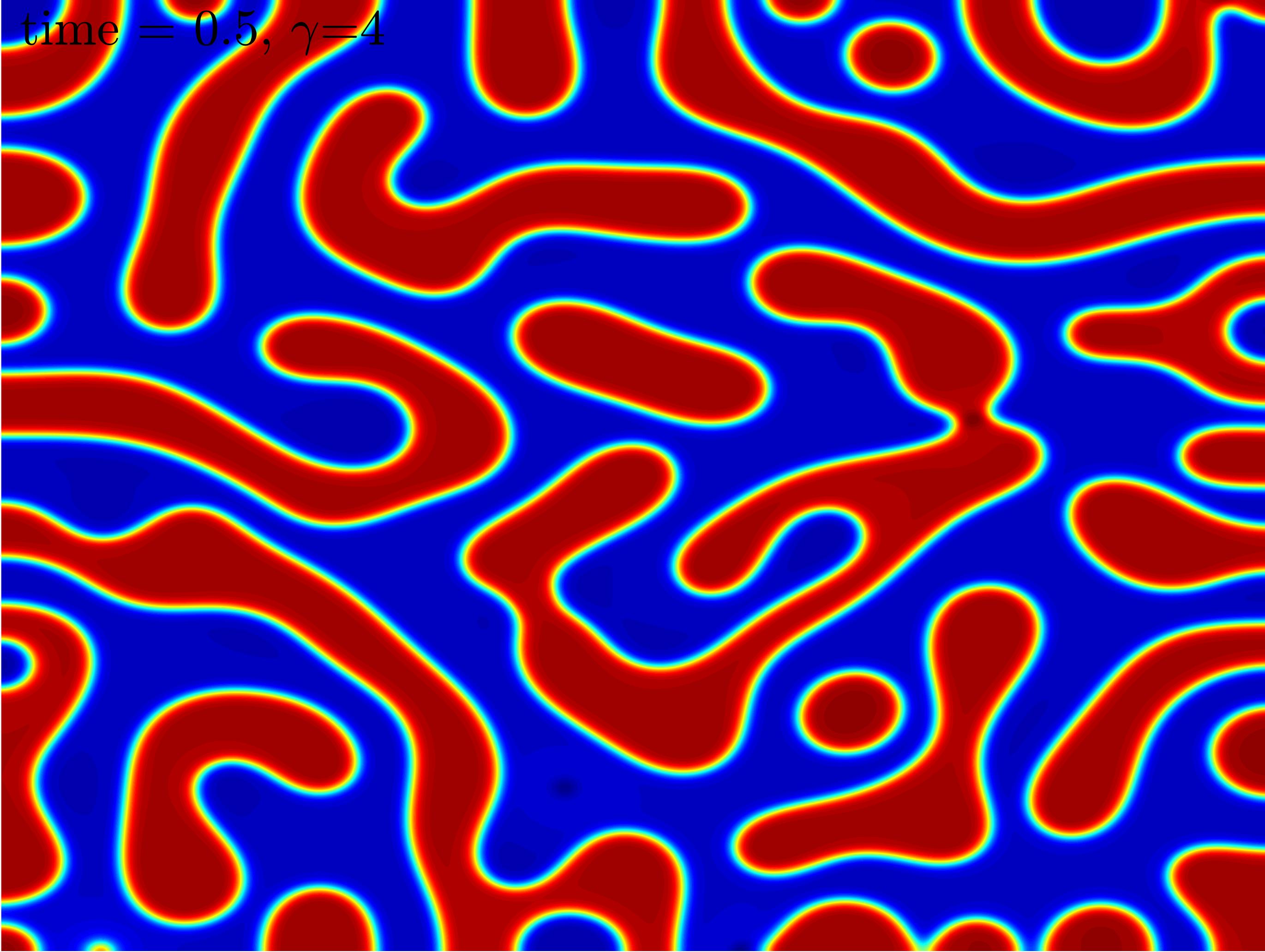}
            \caption*{$t =0.5, \gamma=4$}
        \end{subfigure}\\
    \begin{subfigure}[b]{0.28\textwidth}
        \includegraphics[width=\textwidth,height=\textwidth]{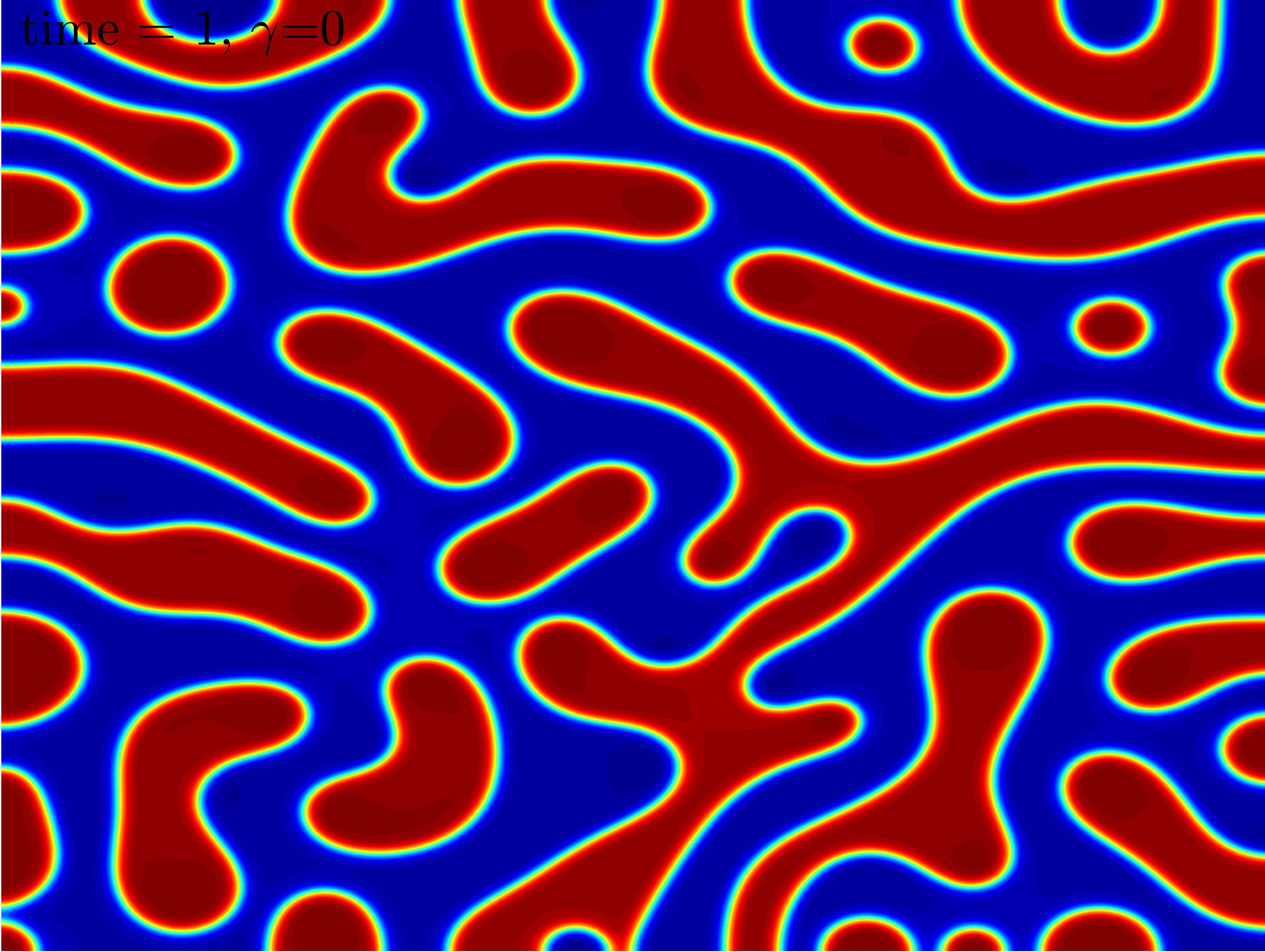}
        \caption*{$t =1, \gamma=0$}
    \end{subfigure}
    \begin{subfigure}[b]{0.28\textwidth}
        \includegraphics[width=\textwidth,height=\textwidth]{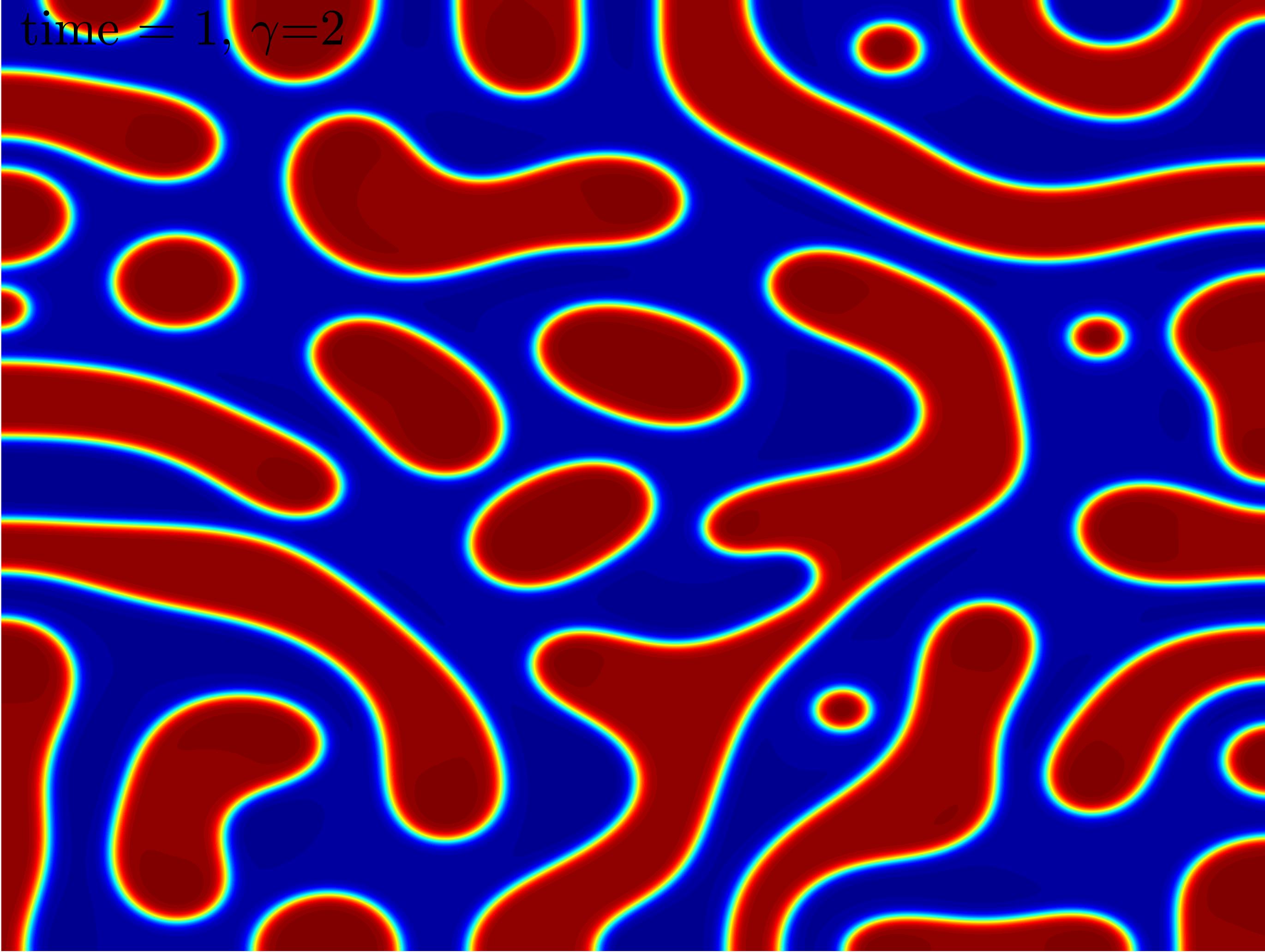}
        \caption*{$t =1, \gamma=2$}
    \end{subfigure}
    \begin{subfigure}[b]{0.28\textwidth}
        \includegraphics[width=\textwidth,height=\textwidth]{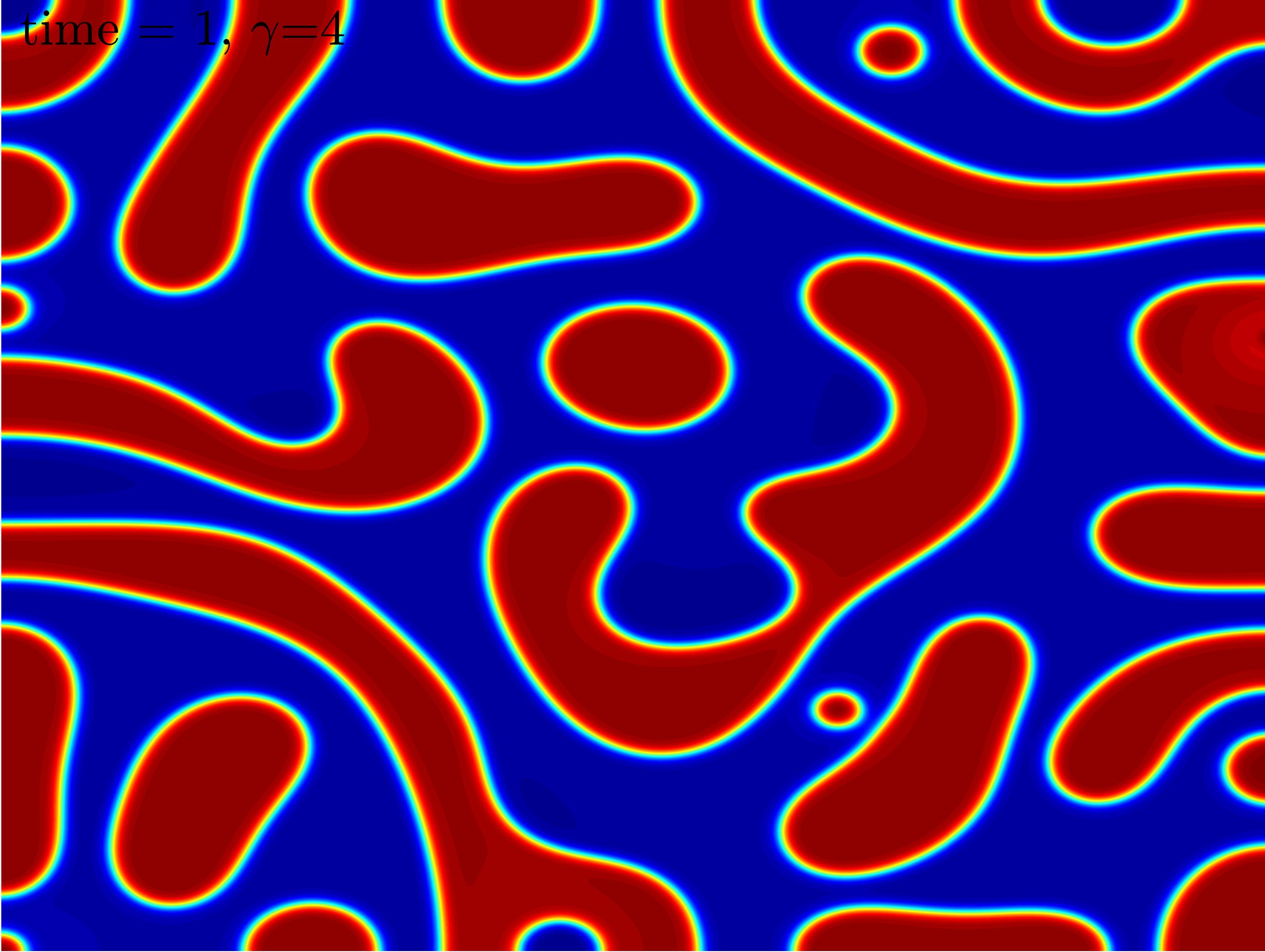}
        \caption*{$t =1, \gamma=4$}
    \end{subfigure}\\
        \begin{subfigure}[b]{0.28\textwidth}
            \includegraphics[width=\textwidth,height=\textwidth]{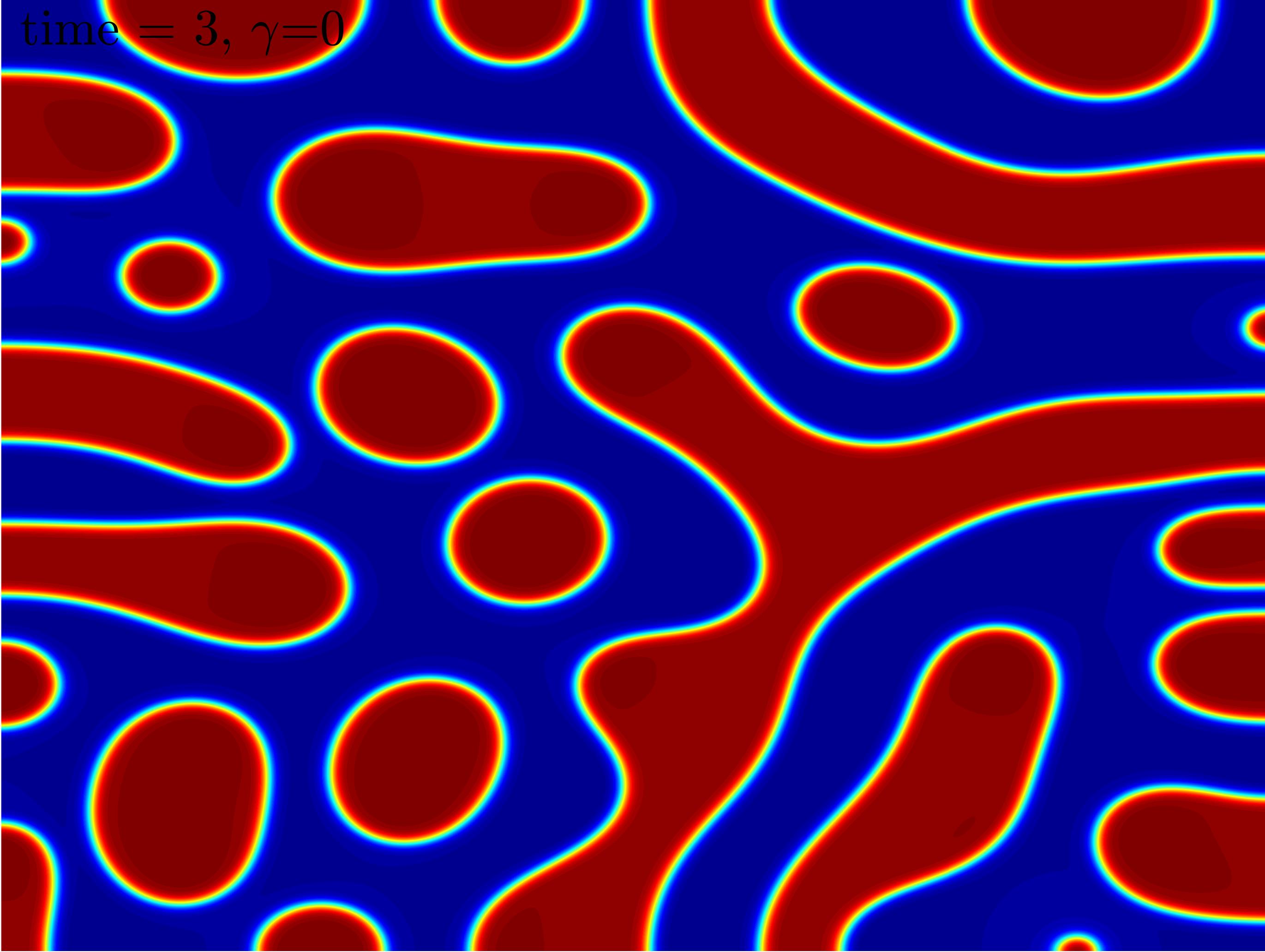}
            \caption*{$t =3, \gamma=0$}
        \end{subfigure}
        \begin{subfigure}[b]{0.28\textwidth}
            \includegraphics[width=\textwidth,height=\textwidth]{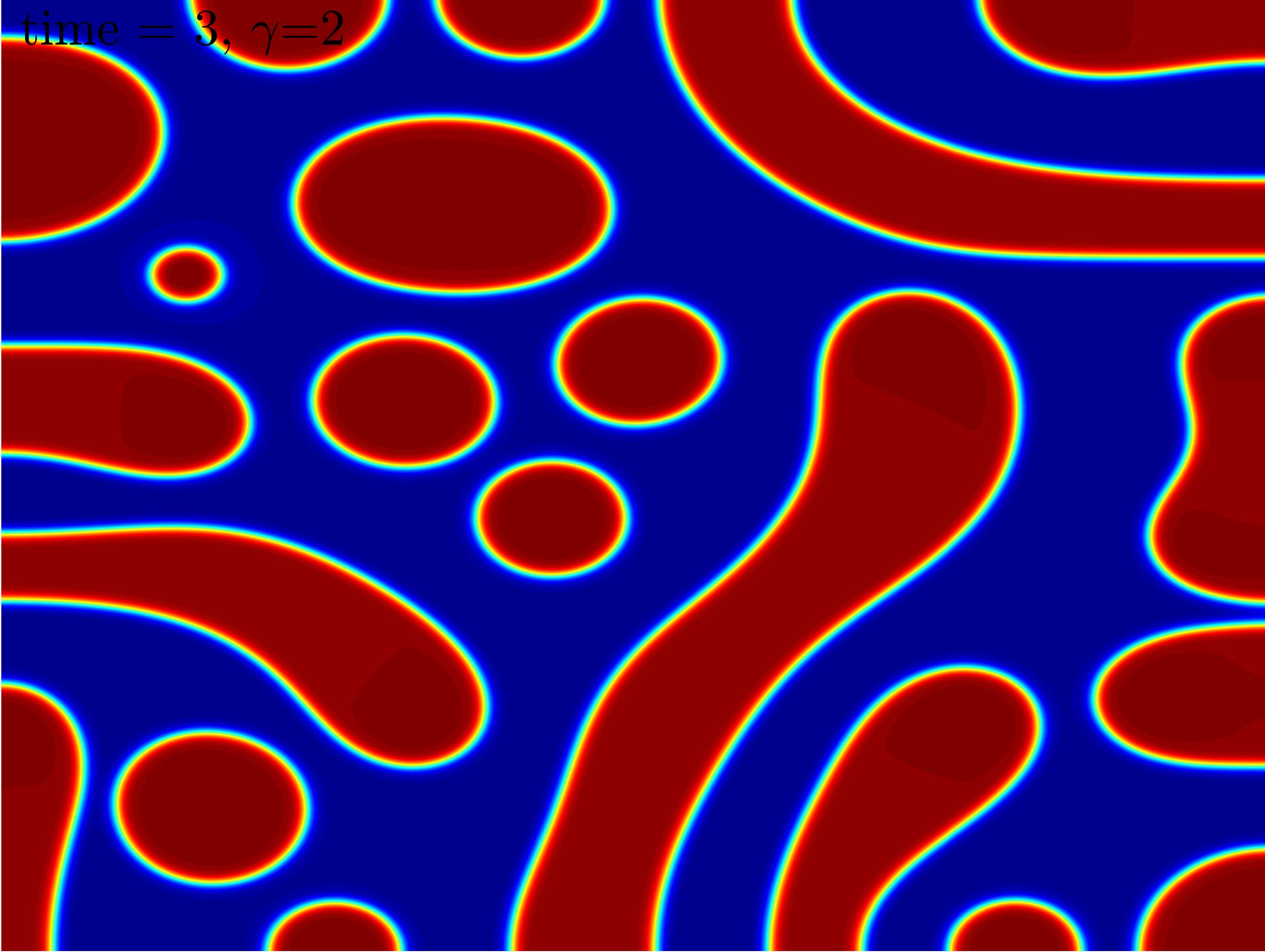}
            \caption*{$t =3, \gamma=2$}
        \end{subfigure}
        \begin{subfigure}[b]{0.28\textwidth}
            \includegraphics[width=\textwidth,height=\textwidth]{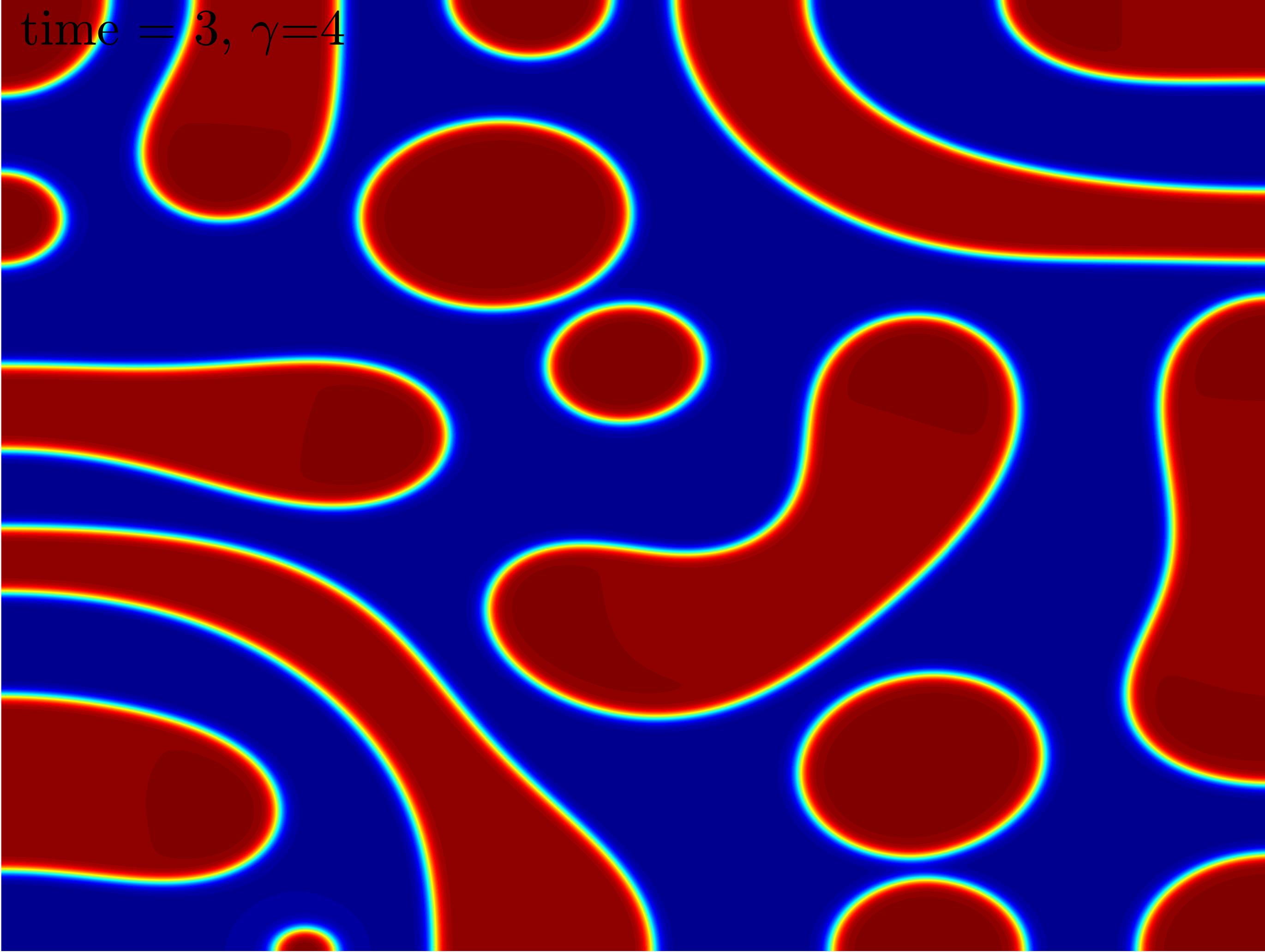}
            \caption*{$t =3, \gamma=4$}
        \end{subfigure}
        \begin{subfigure}[b]{0.28\textwidth}
            \includegraphics[width=\textwidth,height=\textwidth]{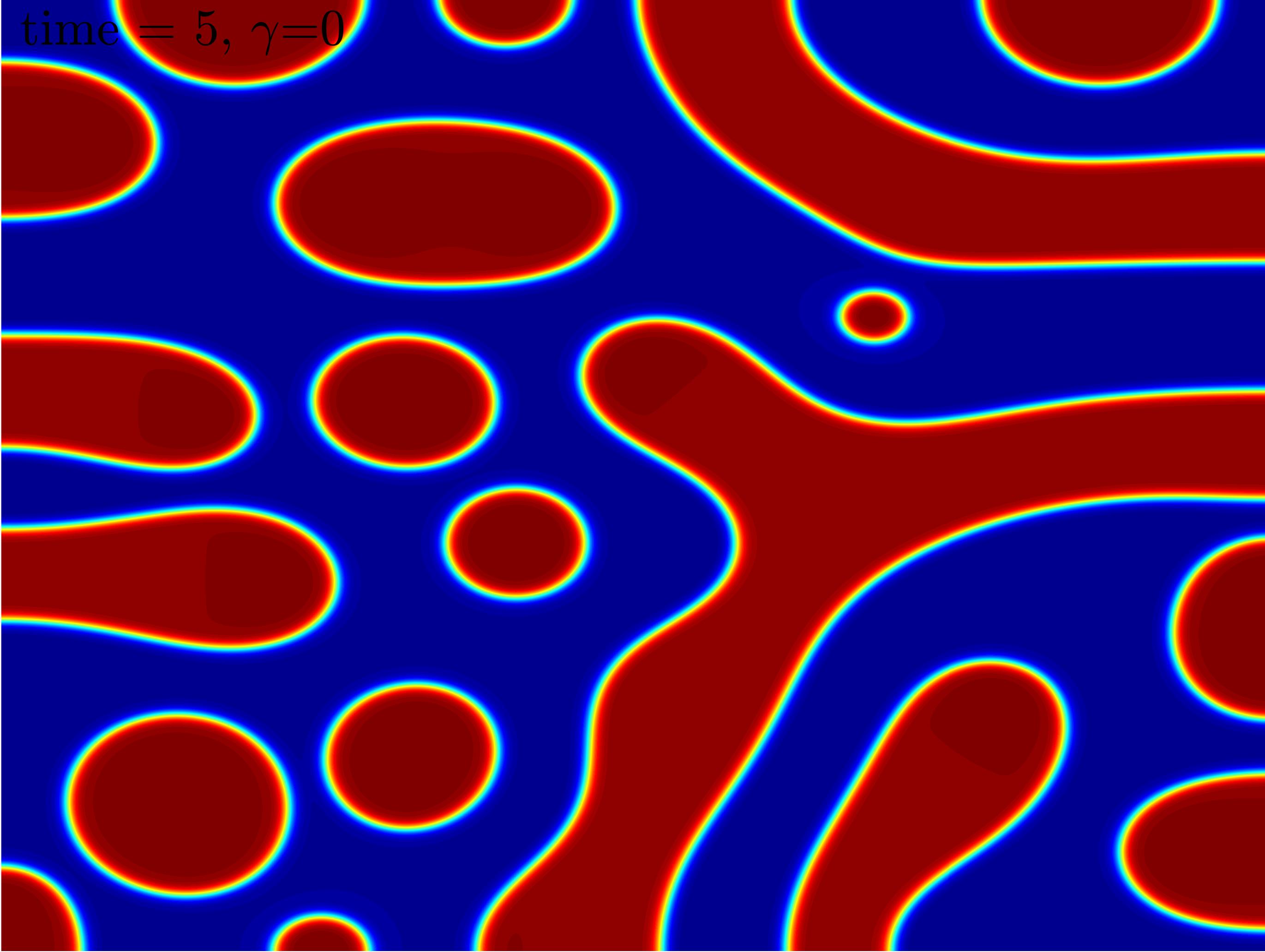}
            \caption*{$t =5, \gamma=0$}
        \end{subfigure}
        \begin{subfigure}[b]{0.28\textwidth}
            \includegraphics[width=\textwidth,height=\textwidth]{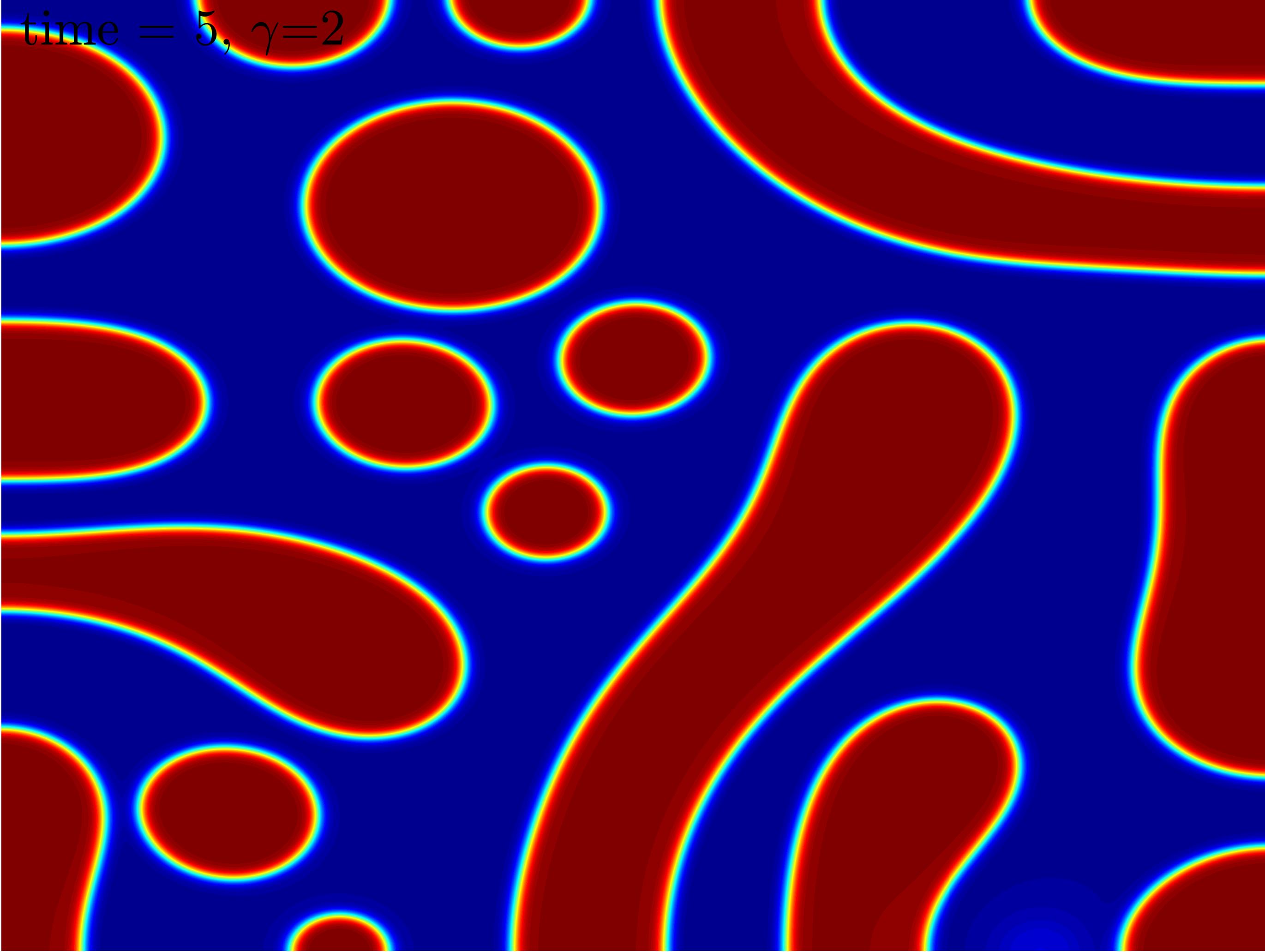}
            \caption*{$t =5, \gamma=2$}
        \end{subfigure}
        \begin{subfigure}[b]{0.28\textwidth}
            \includegraphics[width=\textwidth,height=\textwidth]{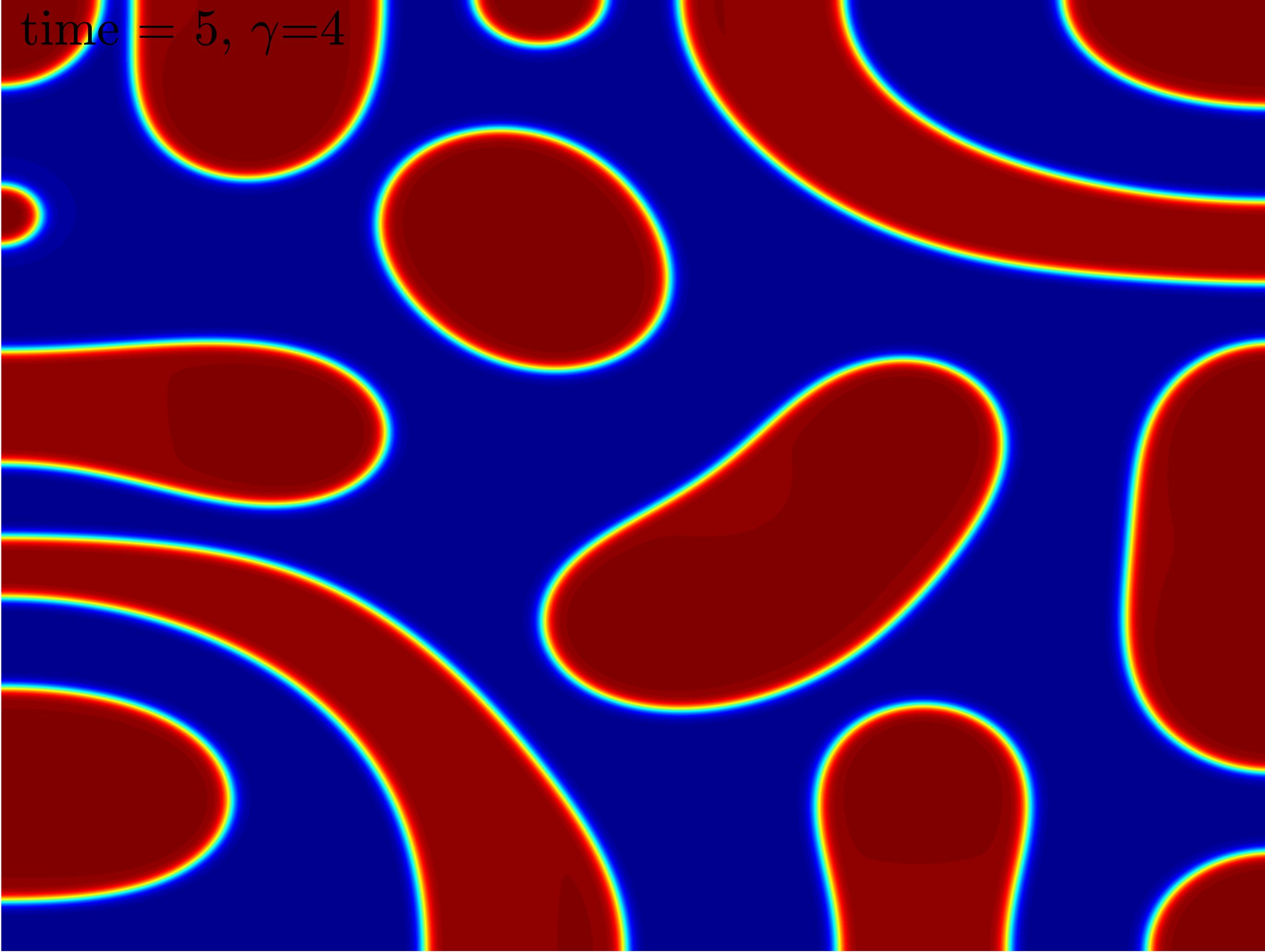}
            \caption*{$t =5, \gamma=4$}
        \end{subfigure}
    \caption{Snapshots of Spinodal decomposition of a binary fluid in a Hele-Shaw cell.}\label{fig:spinodal}
\end{figure}

\begin{figure}[htp]
\centering
\includegraphics[width=0.8\textwidth]{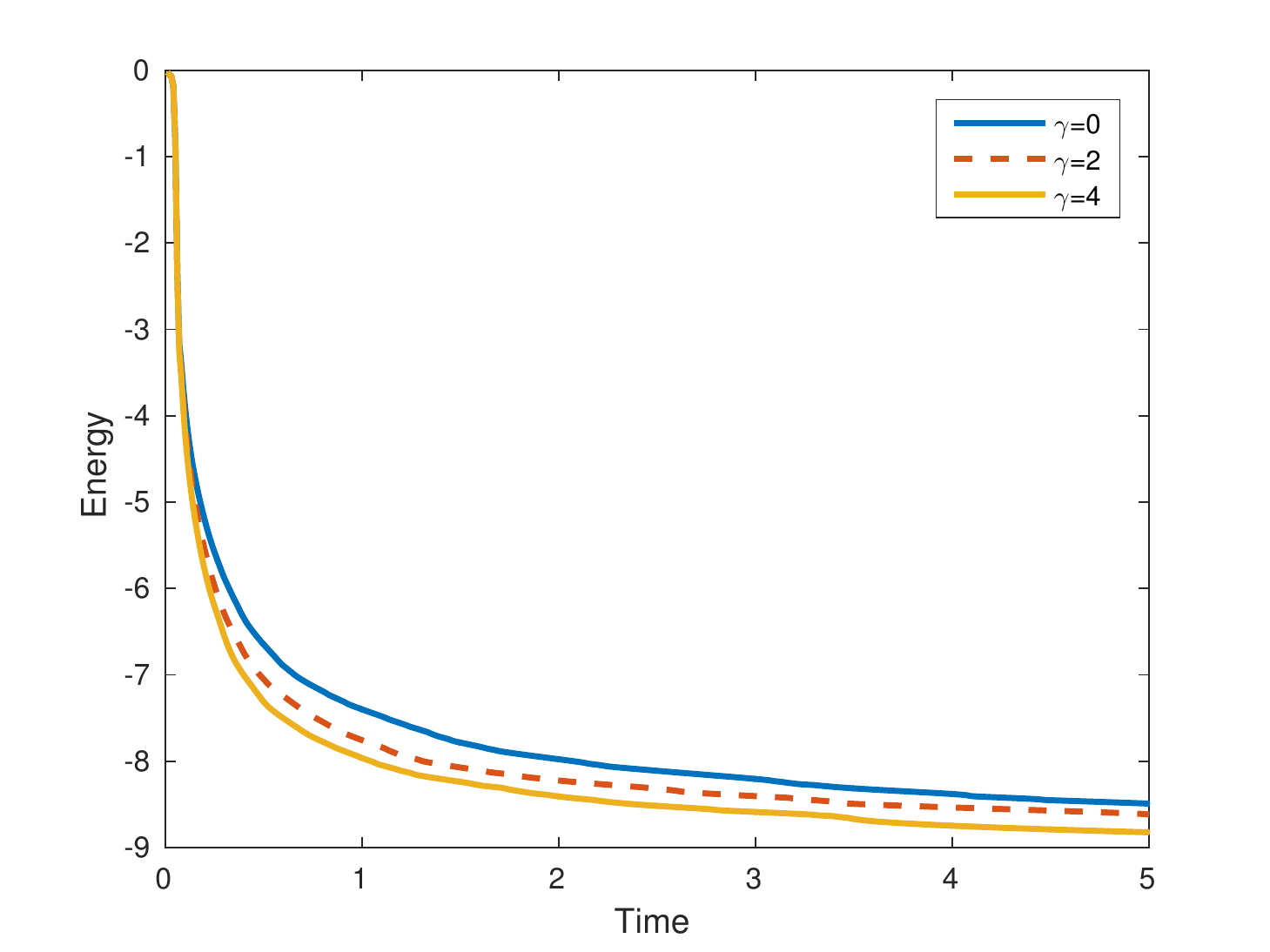}\hfill
\caption{The evolutions of discrete energy with $\gamma=0, 2, 4$.}
\label{fig:energy-gamma}
\end{figure}
\section{Conclusions}\label{sec:conclusion}

A second order accurate energy stable numerical scheme for the Cahn-Hilliard-Hele-Shaw equations is proposed and analyzed in this article. The unique solvability and unconditional energy stability are proved, based on a rewritten form of the scheme, following a convexity analysis. At each time step of this scheme, an efficient nonlinear multigrid solver could be applied to the nonlinear equations associated with the finite difference approximation. At the theoretical side, an $\ell^2 (0,T; H_h^3)$ stability of the numerical scheme is established, in addition to the leading order energy stability. As an outcome of this estimate, we perform an $\ell^\infty (0,T; H_h^1)$ error estimate for the numerical scheme, and an optimal rate convergence analysis is obtained. A few numerical simulation results are presented to demonstrate the accuracy and robustness of the proposed numerical scheme.

	\section*{Acknowledgment}
This work is supported in part by the grants NSF DMS-1418689 (C.~Wang), NSFC 11271281 (C.~Wang), NSF DMS-1418692 (S.~Wise), NSFC 11671098, 91130004, 11331004, a 111 project B08018 (W.~Chen), and the fund by China Scholarship Council 201406100085 (Y.~Liu).  
Y.~Liu thanks University of California-San Diego  
for support during his visit.  C.~Wang also thanks 
Shanghai Key Laboratory for Contemporary Applied Mathematics, 
Fudan University, for support during his visit. 


	\appendix
		
	\section{Discretization of space}
	\label{app:discrete}
	\subsection{Basic definitions}
	Here we use the notation and results for some discrete functions and operators from \cite{wise2010unconditionally}. We begin with definitions of grid functions and difference operators needed for the three-dimensional discretization. We consider the domain $\Omega = (0,L_x)\times(0,L_y)\times(0,L_z)$ and assume that $N_x$, $N_y$ and $N_z$ are positive integers such that $h = L_x/N_x = L_y/N_y = L_z/N_z$, for some $h > 0$, which is called the spatial step size. Consider, for any positive integer $N$, the following sets:
	\begin{eqnarray}
		\mathcal{E}_{N} &\mathop{:=}& \{i\DOT h ~\big|~ i = 0,\ldots,N\},
		\qquad
		\mathcal{C}_{N} \mathop{:=} \{(i-\nicefrac{1}{2})\DOT h ~\big|~ i = 1,\ldots,N)\}, \\
		\mathcal{C}_{\overline{N}} &\mathop{:=}&  \{(i-\nicefrac{1}{2})\DOT h ~\big|~ i = 0,\ldots,N+1)\}.
	\end{eqnarray}
	The two points belonging to $\mathcal{C}_{\overline{N}}\!\setminus\!\mathcal{C}_{N}$ are the so-called \emph{ghost points}. Define the function spaces
	\begin{eqnarray}
		\mathcal{C}_{\Omega}\!\mathop{:=}{}\! \{\phi\!:\!\mathcal{C}_{\overline{N}_x}\!\!\times\!\! \mathcal{C}_{\overline{N}_y}\!\!\times\!\! \mathcal{C}_{\overline{N}_z}\!\rightarrow\!\mathbb{R}\}, \qquad \mathcal{E}^{x}_{\Omega}\! \mathop{:=}{}\! \{\phi\!:\!\mathcal{E}_{N_x}\!\!\times\!\! \mathcal{C}_{N_y}\!\!\times\!\! \mathcal{C}_{N_z}\!\rightarrow\!\mathbb{R}\}, \\
		\mathcal{E}^{y}_{\Omega}\! \mathop{:=}{}\! \{\phi\!:\!\mathcal{C}_{N_x}\!\!\times\!\! \mathcal{E}_{N_y}\!\!\times\!\! \mathcal{C}_{N_z}\!\rightarrow\!\mathbb{R}\},
		\qquad
		\mathcal{E}^{z}_{\Omega}\! \mathop{:=}{}\! \{\phi\!:\!\mathcal{C}_{N_x}\!\!\times\!\! \mathcal{C}_{N_y}\!\!\times\!\! \mathcal{E}_{N_z}\!\rightarrow\!\mathbb{R}\},
		\\
		\vec{\mathcal E}_\Omega := \mathcal{E}^{x}_{\Omega}\times \mathcal{E}^{y}_{\Omega}\times\mathcal{E}^{z}_{\Omega}.
	\end{eqnarray}
	The functions of $\mathcal{C}_{\Omega}$ are called {\emph cell centered functions}. In component form, cell-centered functions are identified via $\phi_{i,j,k} \mathop{:=} \phi(\xi_{i},\xi_{j},\xi_{k})$, where $\xi_{i} := (i-\nicefrac{1}{2})\DOT h$. The functions of $\mathcal{E}^{x}_{\Omega}$, \emph{et cetera}, are called face-centered functions. In component form, face-centered functions are identified via $f_{i+\frac{1}{2},j,k} := f(\xi_{i+\nicefrac{1}{2}},\xi_{j},\xi_{k})$, \emph{etc}.
	
	A discrete function $\phi\in\mathcal{C}_{\Omega}$ is said to satisfy homogeneous Neumann boundary conditions, and we write $\n\cdot \nabla_h\phi  = 0$ iff at the ghost points $\phi$ satisfies
	\begin{alignat}{16}
		&\phi_{0,j,k} &&= \phi_{1,j,k}, \quad &\phi_{N_x,j,k} &&= \phi_{N_x+1,j,k},\label{eqn:nbc1}
		\\
		&\phi_{i,0,k} &&= \phi_{i,1,k}, \quad &\phi_{i,N_y,k} &&= \phi_{i,N_y+1,k},\label{eqn:nbc2}
		\\
		&\phi_{i,j,0} &&= \phi_{i,j,1}, \quad &\phi_{i,j,N_z} &&= \phi_{i,j,N_z+1}.\label{eqn:nbc3}
	\end{alignat}
	A discrete function $\f = (f^x, f^y, f^z)^T\in\vec{\mathcal{E}}_{\Omega}$ is said to satisfy the homogeneous boundary conditions  $\n\cdot \f  = 0$ iff we have
	\begin{alignat}{16}
		&f^x_{\nicefrac{1}{2},j,k} &&= 0, \quad &f^x_{N_x+\nicefrac{1}{2},j,k} &&= 0,
		\\
		&f^y_{i,\nicefrac{1}{2},k} &&= 0, \quad &f^y_{i,N_y+\nicefrac{1}{2},k} &&= 0,
		\\
		&f^z_{i,j,\nicefrac{1}{2}} &&= 0, \quad &f^z_{i,j,N_z+\nicefrac{1}{2}} &&= 0.
	\end{alignat}
	This staggered grid is also known as the marker and cell (MAC) grid and was first proposed in~\cite{harlow1965numerical} to deal with the incompressible Navier-Stokes equations. Also see \cite{samelson2003surface} for related applications to the 3-D primitive equations.

	\subsection{Discrete operators, inner products, and norms}
	\label{sebsec:discrete operate}
	
	We introduce the face-to-center difference operator $d_{x}\!:\!\mathcal{E}^{x}_{\Omega}\rightarrow \mathcal{C}_{\Omega}$, defined component-wise via
	\begin{equation}
		d_{x}f_{i,j,k} := \frac{1}{h}(f_{i+\frac{1}{2},j,k}-f_{i-\frac{1}{2},j,k}),
	\end{equation}
	with $d_{y}\!:\!\mathcal{E}^{y}_{\Omega}\rightarrow \mathcal{C}_{\Omega}$ and $d_{z}\!:\!\mathcal{E}^{z}_{\Omega}\rightarrow \mathcal{C}_{\Omega}$ formulated analogously. Define $\nabla_h\cdot :\vec{\mathcal{E}}_{\Omega}\rightarrow \mathcal{C}_{\Omega}$ via
	\begin{equation}
		\nabla_h\cdot{\f} \mathop{:=}{} d_{x}f^x+d_{y}f^y+d_{z}f^z,
	\end{equation}
	where $\f = (f^x, f^y, f^z)^T$. Define $A_{x}\!:\!\mathcal{C}_{\Omega}\rightarrow \mathcal{E}^{x}_{\Omega}$ component-wise via
	\begin{equation}
		A_{x}\phi_{i+\frac{1}{2},j,k} := \frac{1}{2}(\phi_{i,j,k}+\phi_{i+1,j,k}),
	\end{equation}
	with $A_{y}\!:\!\mathcal{C}_{\Omega}\rightarrow \mathcal{E}^{y}_{\Omega}$ and $A_{z}\!:\!\mathcal{C}_{\Omega}\rightarrow \mathcal{E}^{z}_{\Omega}$ formulated analogously. Define $A_{h}\!:\!\mathcal{C}_{\Omega}
	\rightarrow \vec{\mathcal{E}}_{\Omega}$ via
	\begin{equation}
		A_{h}{\phi} \mathop{:=}{} \left(A_{x}\phi, A_{y}\phi, A_{z}\phi\right)^T.
		\label{face-average}
	\end{equation}
	Define $D_{x}\!:\!\mathcal{C}_{\Omega}\rightarrow \mathcal{E}^{x}_{\Omega}$
	component-wise via
	\begin{equation}
		D_{x}\phi_{i+\frac{1}{2},j,k} := \frac{1}{h}(\phi_{i+1,j,k}-\phi_{i,j,k}).
	\end{equation}
	$D_{y}\!:\!\mathcal{C}_{\Omega}\rightarrow \mathcal{E}^{y}_{\Omega}$ and $D_{z}\!:\!\mathcal{C}_{\Omega}\rightarrow \mathcal{E}^{z}_{\Omega}$ are similarly evaluated. Define $\nabla_{h}\!:\!\mathcal{C}_{\Omega}
	\rightarrow \vec{\mathcal{E}}_{\Omega}$ via
	\begin{equation}
		\nabla_{h}{\phi} \mathop{:=}{} \left(D_{x}\phi, D_{y}\phi, D_{z}\phi\right)^T.
	\end{equation}
	The standard discrete Laplace operator $\Delta_h:\mathcal{C}_\Omega \rightarrow \mathcal{C}_\Omega$ is just
	\begin{equation}
		\Delta_{h}\phi := \nabla_{h}\cdot\nabla_{h}\phi.
	\end{equation}

	We define the following inner-products:
	\begin{alignat}{8}
		&(\phi , \psi) &&\mathop{:=}{}
		h^3 \sum^{L}_{i=1}\sum^{M}_{j=1}\sum^{N}_{m=1}\phi_{i,j,k}\psi_{i,j,k} , && \forall \ \phi, \psi \in \mathcal{C}_{\Omega},\\
		&\left[ f , g\right]_{x} &&\mathop{:=}{} \frac{1}{2} h^3 \sum^{L}_{i=1}\sum^{M}_{j=1}\sum^{N}_{m=1}(f_{i+\frac{1}{2},j,k}g_{i+\frac{1}{2},j,k}+f_{i-\frac{1}{2},j,k}g_{i-\frac{1}{2},j,k}) ,
		&& \quad \forall \ f, g \in \mathcal{E}^{x}_{\Omega}.
	\end{alignat}
	$\left[ \cdot , \cdot\right]_{y}$ and $\left[ \cdot , \cdot\right]_{z}$ can be formulated analogously. For  $\f = (f^x, f^y, f^z)^T, \g = (g^x, g^y, g^z)^T\in\vec{\mathcal{E}}_{\Omega}$ we define the natural inner product
	\begin{equation}
		\left(\f, \g\right)  := \left[ f^x ,  g^x \right]_x +  \left[ f^y ,  g^y \right]_y +  \left[ f^y ,  g^y \right]_z, 
	\end{equation}
	which gives the associated norm $\nrm{\f}_2 = \sqrt{(\f,\f)}$. Analogously, for $\phi, \psi \in \mathcal{C}_{\Omega}$, a natural discrete inner product of their gradients is given by 
	\begin{equation}
		\left(  \nabla_h \phi ,  \nabla_h \psi \right) := \left[ D_x \phi ,  D_x \psi \right]_x +  \left[ D_y \phi ,  D_y \psi \right]_y +  \left[ D_z \phi ,  D_z \psi \right]_z .
		\label{inner product-def 2}
	\end{equation}
	We also introduce the following norms for cell-centered functions $\phi \in
	\mathcal{C}_{\Omega}$:
	\begin{alignat}{4}
		&\nrm{\phi}_{\infty} &&\mathop{:=}{}\max_{i,j,k}|\phi_{i,j,k}|, \label{discrete norm-infty}
		\\
		&\nrm{\phi}_{p} && := \left(|\phi|^p, 1\right)^{\frac{1}{p}} , \qquad 1\leq p < \infty. 
		\label{discrete norm-lp}		
	\end{alignat}
	In addition, we define
	\begin{equation}
		\nrm{\nabla_h \phi}_p := \left( \left[\left|D_x \phi\right|^p, 1 \right]_x +   \left[ \left|D_y \phi\right|^p , 1 \right]_y +  \left[ \left|D_z \phi\right|^p ,1 \right]_z  \right)^{\frac1p} .
	\end{equation}
	In the case of $p=2$, it is clear that $\left( \nabla_h \phi , \nabla_h \phi \right)  = \| \nabla_h \phi \|_2^2$.
	
		In addition, we introduce the discrete $H_h^1$ and $H_h^3$ norms, which are needed in the stability and convergence analysis: 
\begin{eqnarray} 
  &&
  \| \phi \|_{H_h^1}^2 = \| \phi \|_2^2 + \| \nabla_h \phi \|_2^2 ,  \label{discrete norm-H1} 
  \\
&& 
  \| \phi \|_{H_h^3}^2 = \| \phi \|_2^2 + \| \nabla_h \phi \|_2^2 + \| \Delta_h \phi \|_2^2 
  + \| \nabla_h \Delta_h \phi \|_2^2 , \label{discrete norm-H3} 
\end{eqnarray}    
for any $\phi \in {\mathcal C}_\Omega$.

	\subsection{Summation by parts formulas}
	
	For $\phi, \psi \in \mathcal{C}_{\Omega}$ and a velocity vector field $\u
	\in \vec{\mathcal{E}}_{\Omega}$, the following summation by parts formulas can be derived. If $\psi$ satisfies the homogeneous Neumann boundary conditions, we have
	\begin{equation}
		\left( \phi , \Delta_h \psi \right) = - \left( \nabla_h \phi , \nabla_h \psi \right)
		\label{summation-1}
	\end{equation}
	If $\u\cdot\n = 0$ on the boundary, we get
	\begin{equation}
		\left( \phi , \nabla_h \cdot \u \right) = - \left( \nabla_h \phi , \u \right) .
		\label{summation-2}
	\end{equation}	
\bibliographystyle{plain}
\bibliography{CHHSFDM2nd.bib}		
\end{document}